\documentclass{article}

\usepackage{fullpage}

\usepackage{amsmath}
\usepackage{amssymb}
\usepackage{latexsym}
\usepackage{amsthm}
\usepackage[retainorgcmds]{IEEEtrantools}

\usepackage{authblk}

\usepackage[capitalize]{cleveref}

\usepackage{tikz}
\usetikzlibrary{arrows,calc}

\tikzstyle{arc}=[->,shorten <=3pt, shorten >=3pt,
                 >=stealth, line width=1.1pt]
\tikzstyle{edge}=[shorten <=2pt, shorten >=2pt,
                  >=stealth]
\tikzstyle{vertex}=[circle, fill=white, draw,
                    minimum size=5pt,
                    inner sep=0pt, outer sep=0pt, fill = black]

\newtheorem{theorem}{Theorem}
\newtheorem*{theorem*}{Theorem}
\newtheorem{lemma}[theorem]{Lemma}
\newtheorem{observation}[theorem]{Observation}
\newtheorem{proposition}[theorem]{Proposition}
\newtheorem{corollary}[theorem]{Corollary}
\newtheorem*{problem}{Problem}

\usepackage{minitoc}

\title{Lattice path bicircular matroids}

\author[1,2]{Santiago~Guzm\'an-Pro\thanks{sanguzpro@ciencias.unam.mx}}

\author[2]{Winfried~Hochst\"attler\thanks{winfried.hochstaettler@fernuni-hagen.de}}

\affil[1]{\small{Facultad de Ciencias\\
Universidad Nacional Aut\'onoma de M\'exico\\
C.P. 04510, Ciudad Universitaria, M\'exico}}

\affil[2]{FernUniversit\"at in Hagen\\
Fakult\"at f\"ur Mathematik und Informatik\\
 58084 Hagen}

\begin{document}
\date{}

\maketitle

\begin{abstract}
Lattice path matroids and bicircular matroids are two well-known classes
of transversal matroids. In the seminal work of Bonin and de Mier about
structural properties of lattice path matroids, the authors claimed that 
lattice path matroids significantly differ from bicircular matroids.
Recently, it was proved that all cosimple lattice path matroids
have positive double circuits, while it was shown that there is a large class
of cosimple bicircular matroids with no positive double circuits. These 
observations  support Bonin and de Miers' claim. 
Finally, Sivaraman and Slilaty suggested studying the intersection
of lattice path matroids and bicircular matroids as a possibly interesting
research topic. 
In this work, we exhibit the excluded bicircular matroids for the class of lattice path
matroids, and we propose a characterization of the graph family whose bicircular 
matroids are lattice path matroids. As an application of this characterization, 
we propose a geometric description of $2$-connected lattice path bicircular matroids.
\end{abstract}

\section{Introduction}

Bicircular matroids are a minor closed class of matroids that arise from graphs.
There is considerable amount of  research towards studying certain subclasses
of bicircular matroids. Matthews~\cite{matthewsQJMOS28} characterized
bicircular matroids which are also graphic matroids. In~\cite{sivaramanDM328},
ternary bicircular matroids are studied, and in~\cite{chunDM339},
bicircular matroids representable over $GF(4)$ and $GF(5)$
are characterized. Neudauer~\cite{neudauerphd} considered bicircular matroids that
are fundamental transversal matroids. Finally, Sivaraman and
Slilaty~\cite{sivaramanGC2022}
characterize the class of $3$-connected bicircular matroids whose duals 
are bicircular matroids as well.
In the previous work, the authors suggested that studying the intersection
of lattice path matroids and bicircular matroids might be an interesting
research topic. 

Bonin and de Mier \cite{boninEJC27} showed that lattice path matroids are
a class of transversal matroids closed under minors and duality. Furthermore, 
they observed that these matroids have several strong structural properties.
For instance, all connected lattice path matroids have a spanning circuit. 
In the same work, the authors claimed that bicircular matroids and lattice path matroids
are significantly different classes of transversal matroids. In particular, 
they noticed that every uniform matroid is a lattice path matroid, while the 
only uniform bicircular matroids are $U_{3,5}$, $U_{3,6}$ and
$U_{4,6}$, and  the families $U_{1,n}$, $U_{2,n}$, $U_{n,n}$ and $U_{n,n+1}$
for $n\ge 0$. 
This remark yields a large class of lattice path matroids that are not
bicircular matroids.

Recently, and in a different context, it was shown that every simple
lattice path matroid has a positive coline~\cite{albrechtPhd}. 
Since the class of lattice path matroids is closed under duality, the previous
statement implies that every cosimple lattice path matroid has a positive
double circuit. On the contrary, we exhibited a large class of cosimple
bicircular matroids with no positive double circuits~\cite{guzmanAR}. 
This class is obtained by considering bicircular matroids of graphs 
with minimum degree at least $3$ and with girth at least $5$.
So, the bicircular matroids of such graphs are not lattice path matroids,
which yields a large class of bicircular matroids that are not
lattice path matroids.

The observations of the two paragraphs above, motivate us to 
study the class of lattice path bicircular matroids. 
We show that indeed, lattice path bicircular matroids are a thin family of matroids. 
This is done by constructing five graph families $\mathcal{F}_i$ with $i\in\{0,\dots, 4\}$
such that the bicircular matroid of a graph $G\in \mathcal{F}_i$ is a lattice path
matroid. This is one of our two main results.

\begin{theorem}\label{thm:graphs}
Let $G$ be a graph such that $B(G)$ is a connected matroid. 
The bicircular matroid $B(G)$ is a lattice path matroid if and only if
$G$ belongs to $\mathcal{F}_0\cup\mathcal{F}_1\cup \mathcal{F}_2 \cup
\mathcal{F}_3 \cup \mathcal{F}_4$.
\end{theorem}

Moreover, we propose a characterization of lattice path bicircular matroids
within the class of bicircular matroids. We do so  by exhibiting the list of excluded
bicircular minors to lattice path matroids.

\begin{theorem}\label{thm:excludedminors}
A bicircular matroid is a lattice path matroid if and only if
it has no one of the following matroids as a minor:
\[
C^{2,4},~\mathcal{W}^3,~A^3,~R^3,~R^4,~D^4,~B^1, \text{ and } S^1.
\]	
\end{theorem}

We denote by $ex_B(\mathcal{L})$ the set that contains the eight
matroids listed in Theorem~\ref{thm:excludedminors}. The proofs
of Theorems~\ref{thm:graphs} and~\ref{thm:excludedminors}  consists
of three main blocks. In Section~\ref{sec:EM}, we define the matroids 
listed in Theorem~\ref{thm:excludedminors}, and we show that these
are excluded bicircular minors to  lattice path matroids. 
In Section~\ref{sec:GF}, we construct the graph families $\mathcal{F}_i$
for $i\in\{0,\dots, 4\}$, and show that each graph in these families has a
lattice path bicircular matroid. Finally, in Section~\ref{sec:EMGF},
we observe that if a connected bicircular matroid $B(G)$  has no
 $ex_B(\mathcal{L})$-minor, then $G$ belongs to some $\mathcal{F}_i$.
Building on the work from Sections~\ref{sec:EM}, \ref{sec:GF} and \ref{sec:EMGF},
we prove Theorems~\ref{thm:graphs}  and~\ref{thm:excludedminors}
in Section~\ref{sec:CAR}, where we also discuss some of their implications.
For instance, we propose a geometric description of $2$-connected 
lattice path bicircular matroids.

We begin with Section~\ref{sec:prelim}, where we introduce
bicircular matroids and lattice path matroids, together with certain properties
of these  that will be useful for our work.

\section{Preliminaries}
\label{sec:prelim}

We assume familiarity with basic matroid and graph theory.
A standard reference for matroid theory is \cite{oxley2011} and
for graph theory is \cite{bondy2008}. In particular, 
given a graph $G$ and a subset of edges $I$,  we denote by
$G[I]$ the subgraph of $G$ induced by $I$. That is,  $G[I]$ is the subgraph of $G$
with edge set $I$ and no isolated vertices.

Let $G$ be a (not necessarily simple) graph with vertex set $V$ and edge set $E$.
The \textit{bicircular matroid} of $G$ is the matroid $B(G)$ with ground set $E$
whose independent sets are the edge sets $I\subseteq E$ such that  $G[I]$
contains at most one cycle in each connected component. 
Equivalently, the circuits of $B(G)$ are the edge sets of subgraphs which are
subdivisions of: Two loops on the same vertex, two loops joined
by an edge, or three parallel edges joining a pair of vertices.


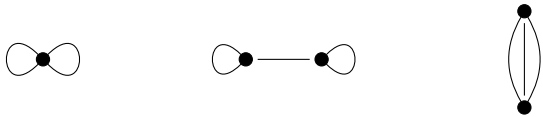
\begin{center}
\begin{tikzpicture}[scale=0.8]

\begin{scope}[xshift=-4cm, yshift=0cm, scale=1]
\node [vertex] (0) at (0,0){};

\path[-]         (0) edge [min distance=1cm, out=45, in=-45] (0);
\path[-]         (0) edge [min distance=1cm, out=225, in=135] (0);
\end{scope}

\begin{scope}[xshift=0cm, yshift=0cm, scale=0.9]
\node [vertex] (0) at (-0.7,0){};
\node [vertex] (1) at (0.7,0){};

\draw [edge] (1) to (0);

\path[-]         (1) edge [min distance=1cm, out=45, in=-45] (1);
\path[-]         (0) edge [min distance=1cm, out=225, in=135] (0);

\end{scope}

\begin{scope}[xshift=4cm, yshift=0cm, scale=0.8]
\node [vertex] (0) at (0,-1){};
\node [vertex] (1) at (0,1){};

\draw [edge] (0) to (1);

\path[-]          (0)  edge   [bend left]   (1);
\path[-]          (0)  edge   [bend right]   (1);
\end{scope}
\end{tikzpicture}
\end{center}

Lattice path matroids arise from lattice paths in a region
bounded by a pair of paths. A \textit{lattice path} consists of a starting point $S$,
$S\in \mathbb{Z}^2$, and a finite sequence of steps $E$, $E= (1,0)$, and $N$,
$N = (0,1)$. Unless stated otherwise we assume that $S = (0,0)$,
so a lattice path can be simply thought as a word over the alphabet $\{E,N\}$.
The steps $E$ and $N$ are called \textit{East} and \textit{North} steps,
respectively. The \textit{endpoint} of a lattice path $P$ is the sum (in
$\mathbb{Z}^2$) of the steps of $P$ (plus $S$ when $S\neq (0,0)$). 

Consider a pair $P$ and $Q$ of lattice paths from $(0,0)$ to $(m,r)$,
where $P$ never goes above $Q$. Let $\mathcal{P}$ be the set
of all lattice paths from $(0,0)$ to $(m,r)$ that do not go below
$P$ nor above $Q$. For every $i\in\{1,\dots r\}$ let $N_i$ be the set
$\{j\colon$ step $j$ is the $i$-th North step of some path in $\mathcal{P}\}$. 
The matroid $M[P,Q]$ is defined as the transversal matroid with ground
set $[m+r]$ that has $(N_1,\dots, N_r)$ as a presentation. We call
$(N_1,\dots, N_r)$ the \textit{standard presentation} of $M[P,Q]$. 
A \textit{lattice path matroid} is a matroid $M$ isomorphic to $M[P,Q]$
for some lattice paths $P$ and $Q$. Notice that in this case, the rank of $M$
is $r$ and its corank is $m$.

Equivalently \cite{boninEJC27}, a matroid $M$ with ground set $E$ is a
lattice path matroid if and only if there is a linear ordering
$e_1\le e_2 \le \cdots \le e_n$ of $E$, and a collection $\mathcal{N}$
of incomparable intervals of $(E,\le)$ such that $M$ is isomorphic to the 
transversal matroid $(E,\mathcal{N})$. We call
$e_1\le e_2 \le \cdots \le e_n$ the \textit{interval ordering} of $\mathcal{N}$, 
and say that $(E,\mathcal{N})$ is an \textit{interval presentation}
of $M$. In particular, every standard presentation is an interval
presentation.  

Lattice path matroids have several nice structural properties investigated
by Bonin and de Mier \cite{boninEJC27}. In particular, we are interested
in the collection of \textit{fundamental flats}.
Let $X$ be a connected flat of a connected matroid $M$ for which $|X| > 1$
and $r(X) < r(M)$. We say that $X$ is a \textit{fundamental flat} of $M$ if for
some spanning circuit $C$ of $M$ the intersection $X\cap C$ is a basis of $X$.
In \cite{boninEJC27}, the authors assert that the fundamental flats of
a lattice path matroid form two disjoint chains ordered by inclusion.

\begin{proposition}\cite{boninEJC27}\label{prop:boninff}
Assume $M[P, Q]$ is a connected lattice path matroid of rank $r$ and corank $m$.
Let $X$ be a connected flat of $M[P, Q]$ with $|X| > 1$ and $r(X) < r$. Then
$X$ is a fundamental flat of $M[P, Q]$ if and only if $X$ is an initial or a final
segment of $[m + r ]$.
\end{proposition}

\subsection*{\normalfont{\textsc{Series extension}}}

Sometimes, we will need to argue that if the bicircular matroid
of some graph is a lattice path matroid, then the bicircular matroids
of certain subdivisions of this graph are also lattice path matroids.
Since edge subdivisions correspond to coparallel extensions, we begin
by proving the following statement.

\begin{lemma}\label{lem:copar}
Let $L$ be a lattice path matroid on $[n]$ with interval presentation
$(N_1,\dots, N_k)$. The matroid obtained by adding a coparallel element $a$
to $1$ and a coparallel element $b$ to $n$ is a lattice path
matroid with interval presentation $\big(\{a,1\}, N_1,\dots, N_k, \{n,b\}\big)$
and interval ordering  $a \le 1 \le\cdots\le n\le b$.
\end{lemma}
\begin{proof}
Since $(N_1,\dots, N_k)$ is an ordered collection of incomparable intervals
of $[n]$, then $\big(\{a,1\}, N_1,\dots, N_k, \{n,b\}\big)$ is an ordered collection
of incomparable intervals of $a\le 1\le 2\le\dots \le k\le b$. So, 
$\big(\{a,1\}, N_1,\dots, N_r, \{n,b\}\big)$ is a lattice path matroid $L'$ with ground
set $[n]\cup\{a,b\}$. To see that $a$ and $1$ are coparallel elements in $L'$
suppose that there is a circuit $C$ that contains $a$ but does not contain $1$. 
Then, $C-a$ is an independent set so, $C-a$ is a partial transversal
of $\big(\{a,1\}, N_1,\dots, N_k, \{n,b\}\big)$. But since $1\not\in C-a$, then 
$C-a$ is a partial transversal of $\big(N_1,\dots, N_k, \{n,b\}\big)$, and thus
$C$ is a partial transversal of $\big(\{a,1\}, N_1,\dots, N_k, \{n,b\}\big)$ which
contradicts the choice of $C$. Analogously we show that 
$n$ and $b$ are coparallel elements in $L'$.
\end{proof}

In particular, this implies that if $M$ is obtained from a uniform matroid
as a series extension of at most two elements, then $M$ is a lattice
path matroid. Also, since edge subdivisions in a graph $G$ correspond
to series extensions in $B(G)$, the following lemma is a particular
case of Lemma~\ref{lem:copar}. Thus, we consider it to be proved.

\begin{lemma}\label{lem:Gsubdivision}
Let $G$ be a graph with a distinguished (colored) edge $e_r$. Suppose
that $B(G)$ is  a lattice path matroid,  with an interval ordering such that $e_r$ 
is a minimal or a maximal element. If   $G'$ is obtained from  subdividing
$e_r$  (into arbitrarily many edges), then $B(G')$ is a lattice path matroid.
\hfill $\square$
\end{lemma}

\subsection*{\normalfont{\textsc{$2$-sums}}}

Consider a pair of matroids $M_1$ and $M_2$ such that 
$E(M_1)\cap E(M_2) = \{e\}$. The \textit{$2$-sum} of $M_1$ and $M_1$
\textit{over} $e$, is the matroid $M_1\oplus_e M_2$ with ground set
$E(M_1)\cup E(M_2)-e$, whose circuits are defined as follows.
\begin{enumerate}
	\item Every circuit of $M_1-e$ and $M_2-e$ is a circuit
	of  $M_1\oplus_e M_2$.
	\item A set $C \subseteq E(M_1)\cup E(M_2)-e$  is a circuit if
	$C\cap E(M_i) \cup\{e\}$ is a circuit for both $i\in\{1,2\}$.
\end{enumerate}

In general, neither lattice path matroids nor bicircular matroids are closed
under $2$-sums. Nonetheless, there are some particular instances
where the $2$-sum of a pair of bicircular (resp.\ lattice path) matroids
is a bicircular (resp.\ lattice path) matroid.

In particular, consider a pair of lattice path matroids $L_1$ and $L_2$
with interval presentations $(E_1,\mathcal{N})$ and $(E_2,\mathcal{N}')$.
Let $\mathcal{N} = (N_1,\dots, N_k)$,
$\mathcal{N}' = (N_1',\dots, N_l')$, and $e_1 \le \cdots \le e_n$ and
$e'_1 \le \cdots \le e'_m$  be the respective interval orderings of $\mathcal{N}$
and of $\mathcal{N}'$. Suppose that 
$E_1\cap E_2 =\{e\} = \{e_n\} = \{e_1'\}$.
In this case, the $2$-sum $L_1\oplus_e L_2$ is the lattice path matroid
with interval presentation
$(N_1,\dots N_{k-1}, (N_k\cup N'_1)-\{e\}, N_2',\dots,  N_m')$, and interval
ordering
\[
e_1 \le \cdots \le e_{n-1} \le e'_2 \le \cdots e_m'.
\]
The previous assertion is a straightforward observation. It also follows from
Lemma 3.2 in  \cite{boninJCTB100}, and by noticing that the $2$-sum
$M_1\oplus_e M_2$ of a pair of matroids
can be obtained by deleting $e$ from the parallel connection of
$M_1$ and $M_2$ over $e$.

Consider now a pair of graphs $G_1$ and $G_2$ 
such that $E(G_1) \cap E(G_2) = \{e\}$, where $e$ is a loop
incident with $v_1\in V(G_1)$, and incident with $v_2\in V(G_2)$.
If $G$ is the graph obtained by gluing $G_1$ and $G_2$ over the
vertices $v_1$ and $v_2$, and removing both copies of $e$, then
the bicircular matroid $B(G)$ is the $2$-sum $B(G_1)\oplus_e B(G_2)$.
In this case, we say that $G$ is a \textit{loop sum} of $G_1$ and $G_2$. 
Whenever any two loop sums of $G_1$ and $G_2$ are isomorphic
we denote the unique (up to isomorphism) loop sum of $G_1$ and $G_2$
by $G_1\oplus_l G_2$.  The following claim is proved by  the arguments in this
brief subsection.

\begin{lemma}\label{lem:loopsums}
Consider a pair of graphs $G_1$ and $G_2$ such that 
$E(G_1)\cap E(G_2)  = \{e\}$ where $e$ is a loop in $G_1$ and in $G_2$.
Let $G$ be the loop sum of $G_1$ and $G_2$ over $e$ and
suppose that  $B(G_1)$ and $B(G_2)$ are lattice path matroids. If there are
interval orderings of $E(G_1)$ and of $E(G_2)$ such that $e$ is a terminal
element in both, then $B(G)$ is a lattice path matroid. \hfill $\square$
\end{lemma}

\subsection*{\normalfont{\textsc{Clones}}}

To conclude this section we introduce one more notion that will be
useful in Section~\ref{sec:GF}. A pair of elements
$e$ and $f$ in a matroid $M$ are called \textit{clones} if the permutation
of $e$ and $f$ is an automorphism of $M$. For instance, any
pair of elements in a uniform matroid are a pair of clones. 
We say  that a pair  of edges $e$ and $f$ of a graph $G$ are
\textit{clones} if $e$ and $f$ are clones in $B(G)$. 
In particular, any pair of parallel edges are clone edges. 

Consider a matroid $M$ and a set $S\subseteq E$ of clones in $M$.
A \textit{clone extension} of $M$ over $S$, is a matroid $M+f$ such that
$f$ is $S\cup\{f\}$ is a set of clones in $M+f$. In general, 
there might be more than one non-isomorphic clone extensions. 
For instance, if $S = \{e\}$ the parallel and the coparallel extensions
of $e$ are clone extensions of $M$ over $S$.

\begin{lemma}\label{lem:cloneEX}
Let $M$ be a matroid and $S$ a set of clones of $M$.
If $S$ is a circuit of $M$, then there is a unique (up to isomorphism)
clone extension of $M$ over $S$.
\end{lemma}
\begin{proof}
Consider a pair of clone extensions $M+e$ and $M+f$
of $M$ over $S$. We verify that $\varphi\colon M+e\to M+f$
defined as the identity on $E$ and $\varphi(e) =f$, is an isomorphism. 
We do so by observing that this function maps circuits to circuits,
and the claim follows by symmetry.  Clearly, every circuit of $M+e$ to
which $e$ does not belong, is mapped to a circuit of $M+f$ to which $f$ does
not belong. Similarly, if $C$ is a circuit of $M+e$ such that $e\in C$
but $s\not\in C$ for some $s\in S$, then $C-e+s$ is a circuit of $M+e$
(because $s$ and $e$ are clones). Since $C-e+s\subseteq E$, then
$C-e+s$ is a circuit of $M+f$, and so, $(C-e+s)-s+f$ is a circuit of
$M+f$, i.e., $\varphi[C] = C-e+f$ is a circuit of $M+f$.
To conclude the proof, we argue that for every circuit $C$ of $M+e$
to which $e$ belongs, there is some $s\in S$ such that $s\not\in C$
(and thus, the claim will follow by the previously considered cases).
Anticipating a contradiction, suppose
that there is some circuit $C$ of $M+e$ such that $S\cup\{e\}\subseteq C$. 
Since $M+e$ is an extension of $M$, then $S$ is a circuit of $M+e$.
Hence, $C$ properly contains another circuit, which contradicts
the fact that circuits are minimal dependent sets.
\end{proof}

Clearly, any bicircular matroid $B(G)$ together with any set of three
parallel edges of $G$, satisfy the hypothesis of
Lemma~\ref{lem:cloneEX}. We use this observation 
to prove the following result.

\begin{lemma}\label{lem:triangles}
Let $G$ be graph and let $f_1$, $f_2$ and $f_3$ be three different parallel
edges. If $B(G)$ is a lattice path matroid, then $B(G+f)$ is a lattice path matroid
whenever $f$ is added to the parallel class of $f_1$, $f_2$ and $f_3$. Moreover,
if $B(G)$ has an interval ordering where $f_1$ is its minimum (resp.\ maximum),
then $B(G+f)$ has an interval ordering where $f_1$ is its minimum
(resp.\ maximum).
\end{lemma}
\begin{proof}
We begin with the following observation about transversal matroids.
Consider a bipartite graph $H(E,Y)$, and let $T$ be the transversal
matroid  presented by $H$ (with ground set $E$). Let $e_1e_2e_3$ be
a circuit  of $T$. It is evident that
$|N_H(\{e_1,e_2,e_3\})| = 2$. Moreover, at least one of these elements
has degree two in $H$; otherwise, by the pigeonhole principle, two 
of them would be leaves of $H$ with a common support, i.e., 
parallel elements in $T$. Also notice that any two elements
$e,f\in E$ with the same neighborhood in $H$ are clones
in $T$. In particular, any pair of parallel edges in a graph 
are clone elements in its bicircular matroid.

Let $G$, $f_1$, $f_2$, $f_3$ and $f$ be as in the hypothesis.
In particular, $B(G+f)$ is a clone extension of $G$ over
$\{f_1,f_2,f_3\}$.
Let $(E,\mathcal{N})$ be an
interval presentation of $B(G)$ with interval ordering
$e_1\le \cdots \le e_{k-2} \le f_1 \le e_k \le \cdots \le e_n$. By the choice of 
$f_1,f_2$ and $f_3$,
and by the arguments in the previous paragraph, we assume that
$f_1$ belongs to exactly two sets of $\mathcal{N}$. Consider the
presentation $(E\cup\{f\},\mathcal{N}')$, where $f$ belongs to the same
two sets as $f_1$. Then, this is an interval presentation of a clone extension
$B(G)+f$ over $\{f_1,f_2,f_3\}$ where any of the following is an interval ordering
of $B(G)+f$
\[
e_1 \le \cdots \le f_1 \le f  \le e_k \le \cdots \le e_n, \text{ and }
e_1 \le \cdots \le e_{k-2} \le f \le f_1 \le \cdots \le e_n.
\] 
Since $B(G+f)$ and $B(G)+f$ are clone extensions of $B(G)$ over $\{f_1,f_2,f_3\}$,
and $f_1f_2f_3$ is a circuit of $B(G)$, by Lemma~\ref{lem:cloneEX}
we conclude that $B(G+f) \cong B(G)+f$. Thus, $B(G+f)$ is a lattice path matroid.
The moreover statement follows from the previously defined interval orderings.
\end{proof}


\section{Excluded minors}
\label{sec:EM}

In \cite{boninJCTB100}, Bonin exhibits an infinite family of excluded minors for
lattice path matroids.  In this section we introduce eight bicircular excluded
minors for lattice path matroids. Some of these will have structural implications
about those graphs with lattice path bicircular matroids.
We begin by exhibiting the prism $C^{2,4}$ and the four graphs whose bicircular
matroid is isomorphic to this prism.


\begin{figure}[ht!]
\begin{center}

\begin{tikzpicture}[scale=0.8]

\begin{scope}[xshift=-7.5cm, yshift = 2cm, scale=0.8]
\node [vertex] (0) at (-1,0){};
\node [vertex] (1) at (1,0){};
\node [vertex] (c) at (0,0.9){};
\node [vertex] (3) at (0,2.24){};
\node[] at (0,-1){$P_1$};

\foreach \from/\to in {c/0,c/1, c/3}
\draw [edge] (\from) to (\to);

\path[-]         (1) edge [min distance=1cm] (1);
\path[-]         (0) edge [min distance=1cm] (0);
\path[-]         (3) edge [min distance=1cm] (3);

\end{scope}

\begin{scope}[xshift=-2.5cm, yshift = 2cm, scale=0.8]
\node [vertex] (0) at (-1,0){};
\node [vertex] (1) at (1,0){};
\node [vertex] (c) at (0,0.9){};
\node [vertex] (3) at (0,2.24){};
\node[] at (0,-1){$P_2$};

\foreach \from/\to in {c/1,c/0, c/3, 0/c, 1/c, 3/c}
\path[-]          (\from)  edge   [bend left]   (\to);

\end{scope}

\begin{scope}[xshift=-7.5cm, yshift = -2cm, scale=0.8]
\node [vertex] (0) at (-1,0){};
\node [vertex] (1) at (1,0){};
\node [vertex] (c) at (0,0.9){};
\node [vertex] (3) at (0,2.24){};
\node[] at (0,-1){$P_3$};

\foreach \from/\to in {c/0}
\draw [edge] (\from) to (\to);

\path[-]         (0) edge [min distance=1cm] (0);

\foreach \from/\to in {c/1, c/3, 1/c, 3/c}
\path[-]          (\from)  edge   [bend left]   (\to);

\end{scope}

\begin{scope}[xshift=-2.5cm, yshift = -2cm, scale=0.8]
\node [vertex] (0) at (-1,0){};
\node [vertex] (1) at (1,0){};
\node [vertex] (c) at (0,0.9){};
\node [vertex] (3) at (0,2.24){};
\node[] at (0,-1){$P_4$};

\foreach \from/\to in {c/1,c/0}
\draw [edge] (\from) to (\to);

\path[-]         (1) edge [min distance=1cm] (1);
\path[-]         (0) edge [min distance=1cm] (0);

\path[-]          (c)  edge   [bend left]   (3);
\path[-]          (3)  edge   [bend left]   (c);

\end{scope}

\begin{scope}[xshift=5cm, yshift=0cm, scale=0.8]
\node [vertex] (0) at (-1.3,0){};
\node [vertex] (1) at (1.3,0){};
\node [vertex] (3) at (-1.3,3){};
\node [vertex] (4) at (1.3,3){};
\node [vertex] (5) at (0,3.7){};
\node [vertex] (6) at (0,0.7){};
\node[] at (0,-1){$C^{2,4}$};

\foreach \from/\to in {0/1, 1/4, 3/4, 3/0, 5/3, 5/4}
\path[-] (\from) edge (\to);

\foreach \from/\to in {6/0, 6/1, 6/5}
\draw [edge, dotted] (\from) to (\to);

\end{scope}

\end{tikzpicture}

\caption{To the right, an affine representation of $C^{2,4}$. To the left, 
four graphs whose bicircular matroid is isomorphic to $C^{2,4}$.}
\label{fig:C24}
\end{center}
\end{figure}
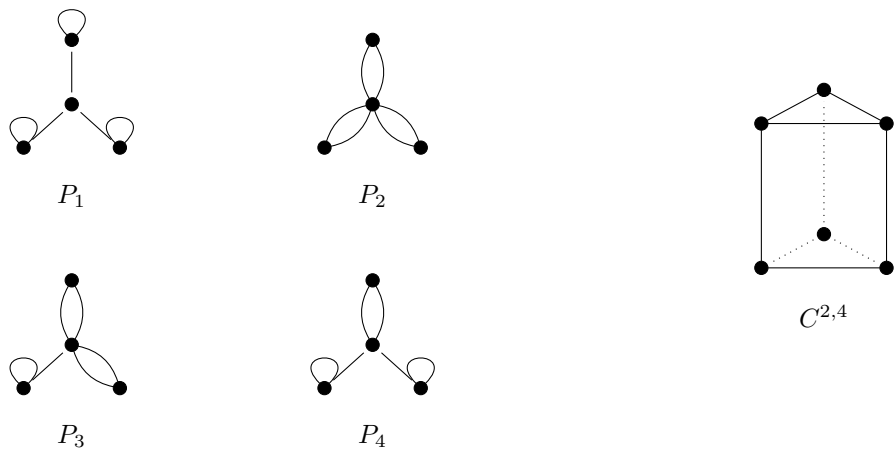

Recall that a \textit{block} of a graph $G$ is a maximal 
$2$-connected subgraph. An \textit{end block} is a block of $G$
that is a leaf in the block tree of $G$; a \textit{middle
block} is a block of $G$ that is not an end block. In particular,
if $B(G)$ is a connected matroid, then every end block of $G$
contains at least one cycle (i.e., $G$ has no loopless leaves). Furthermore, 
suppose that $G$ has three different end blocks. By contracting
all of $G$ except these three end blocks, we can notice
that $G$ contains $P_i$ as a minor for some $i\in\{1,2,3,4\}$
(these graphs are depicted in Figure~\ref{fig:C24}). Therefore, if
$G$ is a graph such that $B(G)$ is a connected matroid with no $C^{2,4}$-minor,
then the block tree of $G$ is a path.

\begin{proposition}\label{prop:blockpath}
Let $G$ be a graph such that $B(G)$ is a connected matroid. 
If $B(G)$ is does not contain $C^{2,4}$ as a minor, then the block
tree of $G$ is a path. Moreover, every middle block contains exactly
two vertices. 
\end{proposition}
\begin{proof}
In the paragraph preceding this proposition, we showed
that the block tree of $G$ is a path. This also implies that 
every middle block contains exactly two cut vertices of $G$.
To prove the moreover statement, and anticipating a
contradiction, suppose that there is a middle block
$MB$ of $G$ with at least three vertices. With out loss of generality, 
assume that $MB$ is the unique middle block of $G$; otherwise
contract every middle block that is not $MB$. So, let 
$EB_1$, $MB$ and $EB_2$ be the blocks of $G$. Let $x\in MB\cap EB_1$
and $y\in  MB\cap EB_2$ be the cut vertices of $G$ in $MB$. Since
$B(G)$ is connected, $G$ has no loopless leaves. Thus,
there are two possible scenarios for each end block $EB_i$: Either
there is a cycle in $EB_i$, or $EB_i$ consist of an edge together
with a loop incident with the end vertex. 
We consider the case when $EB_1$ and $EB_2$ contain a cycle.
Since $MB$ is two connected and it is not an edge, there are two internally
disjoint $xy$-paths $Q_1$ and $Q_2$. Furthermore, since $MB$ has at least
three vertices, then we choose $Q_1$ such that it is not an edge. Finally, let
$C_1$ and $C_2$ be a pair of cycles of the end blocks. 
Consider the following contractions of $G$: Contract $Q_1$ to a path of
length $2$; contract $Q_2$; contract $C_1$ to a pair of parallel edges;
and contract $C_2$ to a pair of parallel edges. The previous minor of $G$
is isomorphic to the graph $P_2$ (Figure~\ref{fig:C24}). Thus, $B(G)$ contains
a $C^{2,4}$-minor. When one or both end blocks consist of an edge together
with a loop, we find either of $P_3$ or $P_4$ as a minor of $G$, and the claim
follows by similar arguments.
\end{proof}

Consider the four point line with ground set $\{1,2,3,4\}$, and add
three elements $1'$, $2'$, and $3'$, coparallel to $1$, to $2$ and to $3$,
respectively.  Then, by contracting $4$ we obtain $C^{2,4}$ as a contraction
minor.  Now, consider a uniform matroid $U_{r,n}$ where $2\le r \le n-1$. 
It is not hard to notice, that any four points of $U_{r,n}$ belong to 
a four point line minor of $U_{r,n}$. Therefore, if we add
elements in series to at least three different points of $U_{r,n}$,
we obtain a matroid $M$ with a $C^{2,4}$-minor. This arguments
prove the following statement.

\begin{lemma}\label{lem:uniformextensions}
Let $r$ and $n$ be a pair of non-negative integers and consider the uniform
matroid $U_{r,n}$. If $2\le r \le n-1$, and
$M$ is a series extension of at least three different elements 
of $U_{r,n}$, then $M$ contains  a $C^{2,4}$-minor.
\hfill $\square$
\end{lemma}


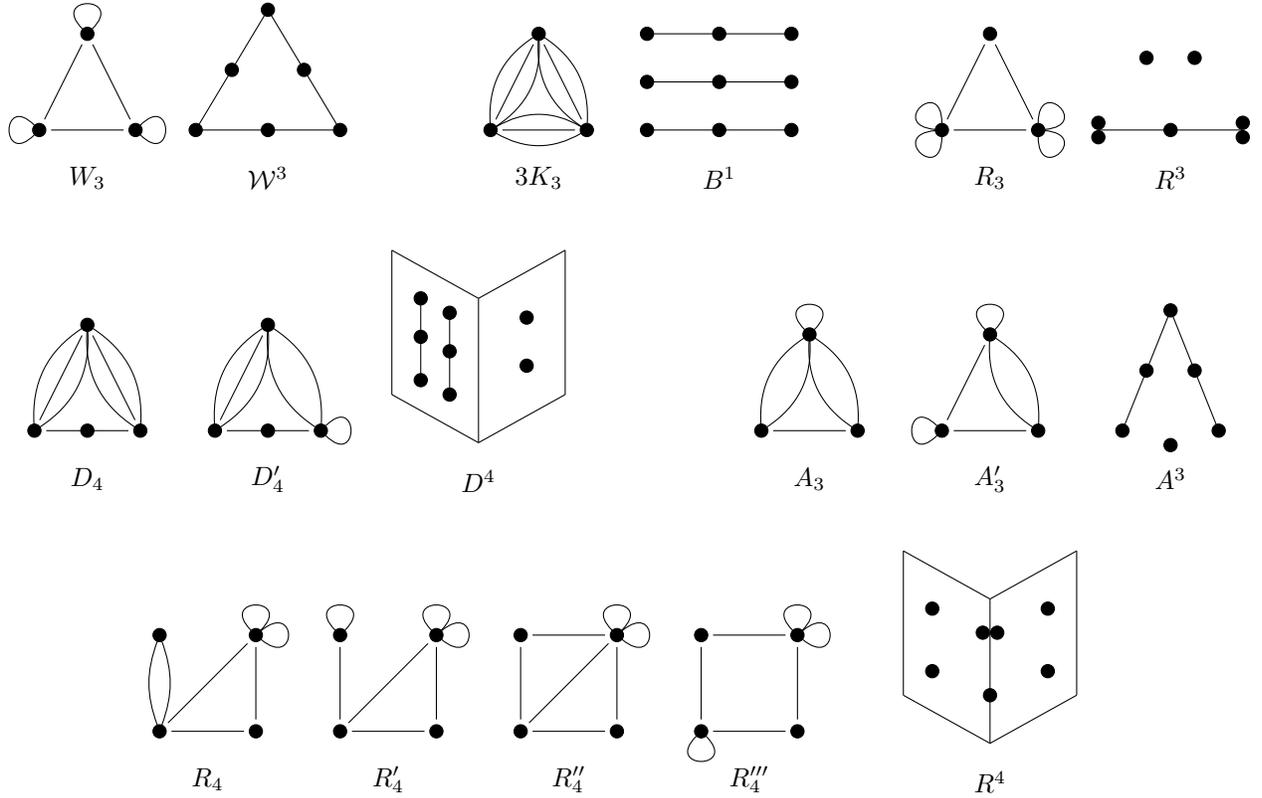
\begin{figure}[ht!]
\begin{center}

\begin{tikzpicture}[scale=0.8]

\begin{scope}[xshift=-1.5cm, yshift=2cm, scale=0.8]
\node [vertex] (0) at (-1,0){};
\node [vertex] (1) at (1,0){};
\node [vertex] (2) at (0,2){};
\node[] at (0,-1){$3K_3$};

\path[-]          (0)  edge   [bend left]   (2);
\path[-]          (0)  edge   [bend right]   (2);
\path[-]          (1)  edge   [bend right]   (2);
\path[-]          (1)  edge   [bend left]   (2);
\path[-]          (1)  edge   [bend right]   (0);
\path[-]          (1)  edge   [bend left]   (0);

\foreach \from/\to in {0/1, 0/2, 1/2}
\draw [edge] (\from) to (\to);
\end{scope}

\begin{scope}[xshift=1.5cm, yshift=2cm, scale=0.8]
\coordinate (c0) at (-1.5,0){};
\coordinate (c1) at (1.5,0){};
\coordinate (d0) at (-1.5,1){};
\coordinate (d1) at (1.5,1){};
\coordinate (i0) at (-1.5,2){};
\coordinate (i1) at (1.5,2){};
\node [vertex] (0) at (-1.5,0){};
\node [vertex] (1) at (1.5,0){};
\node [vertex] (2) at (0,0){};
\node [vertex] (0) at (-1.5,1){};
\node [vertex] (1) at (1.5,1){};
\node [vertex] (2) at (0,1){};
\node [vertex] (0) at (-1.5,2){};
\node [vertex] (1) at (1.5,2){};
\node [vertex] (2) at (0,2){};
\node[] at (0,-1){$B^1$};

\foreach \from/\to in {c0/c1, d0/d1, i0/i1}
\path[-] (\from) edge (\to);
\end{scope}

\begin{scope}[xshift=-9cm, yshift = 2cm, scale=0.8]
\node [vertex] (0) at (-1,0){};
\node [vertex] (1) at (1,0){};
\node [vertex] (3) at (0,2){};
\node[] at (0,-1){$W_3$};

\foreach \from/\to in {1/0, 0/3, 3/1}
\draw [edge] (\from) to (\to);

\path[-]         (1) edge [min distance=1cm, out=45, in=-45] (1);
\path[-]         (0) edge [min distance=1cm, out=225, in=135] (0);
\path[-]         (3) edge [min distance=1cm] (3);

\end{scope}

\begin{scope}[xshift=-6cm, yshift = 2cm, scale=0.8]
\node [vertex] (0) at (-1.5,0){};
\node [vertex] (1) at (1.5,0){};
\node [vertex] (3) at (0,2.5){};
\node [vertex] (01) at (-0.75,1.25){};
\node [vertex] (13) at (0.75,1.25){};
\node [vertex] (30) at (0,0){};
\node[] at (0,-1){$\mathcal{W}^3$};

\foreach \from/\to in {1/0, 0/3, 3/1}
\path[-] (\from) edge (\to);

\end{scope}

\begin{scope}[xshift=9cm, yshift=2cm, scale=0.8]
\coordinate (0) at (-1.5,0){};
\node [vertex] (00) at (-1.5,0.15){};
\node [vertex] (01) at (-1.5,-0.15){};
\coordinate  (1) at (1.5,0){};
\node [vertex] (10) at (1.5,-0.15){};
\node [vertex] (11) at (1.5,0.15){};
\node [vertex] (30) at (-0.5,1.5){};
\node [vertex] (31) at (0.5,1.5){};
\node [vertex] (c) at (0,0){};
\node[] at (0,-1){$R^3$};

\path[-] (0) edge (1);

\end{scope}

\begin{scope}[xshift=6cm, yshift=2cm, scale=0.8]
\node [vertex] (0) at (-1,0){};
\node [vertex] (1) at (1,0){};
\node [vertex] (3) at (0,2){};
\node[] at (0,-1){$R_3$};

\foreach \from/\to in {1/0, 0/3, 3/1}
\draw [edge] (\from) to (\to);

\path[-]         (0) edge [min distance=1cm, out=90, in=175] (0);
\path[-]         (0) edge [min distance=1cm, out=185, in=270] (0);
\path[-]         (1) edge [min distance=1cm, out=90, in=5] (1);
\path[-]         (1) edge [min distance=1cm, out=355, in=270] (1);

\end{scope}

\begin{scope}[xshift=3cm, yshift = -3cm, scale=0.8]
\node [vertex] (0) at (-1,0){};
\node [vertex] (1) at (1,0){};
\node [vertex] (3) at (0,2){};
\node[] at (0,-1){$A_3$};

\foreach \from/\to in {0/1}
\draw [edge] (\from) to (\to);

\path[-]         (3) edge [min distance=1cm] (3);

\foreach \from/\to in {3/1, 1/3, 3/0, 0/3}
\path[-]          (\from)  edge   [bend left]   (\to);

\end{scope}

\begin{scope}[xshift= 6cm, yshift = -3cm, scale=0.8]
\node [vertex] (0) at (-1,0){};
\node [vertex] (1) at (1,0){};
\node [vertex] (3) at (0,2){};
\node[] at (0,-1){$A_3'$};

\foreach \from/\to in {0/1, 0/3}
\draw [edge] (\from) to (\to);

\path[-]         (3) edge [min distance=1cm] (3);
\path[-]         (0) edge [min distance=1cm, out=225, in=135] (0);

\foreach \from/\to in {3/1, 1/3}
\path[-]          (\from)  edge   [bend left]   (\to);
\end{scope}

\begin{scope}[xshift=9cm, yshift=-3cm, scale=0.8]
\node [vertex] (0) at (-1,0){};
\node [vertex] (1) at (1,0){};
\node [vertex] (3) at (0,2.5){};
\node [vertex] (01) at (-0.5,1.25){};
\node [vertex] (13) at (0.5,1.25){};
\node [vertex] (30) at (0,-0.3){};
\node[] at (0,-1){$A^3$};

\foreach \from/\to in {0/3, 3/1}
\path[-] (\from) edge (\to);

\end{scope}

\begin{scope}[xshift=-9cm, yshift = -3cm, scale=0.8]
\node [vertex] (0) at (-1.1,0){};
\node [vertex] (1) at (1.1,0){};
\node [vertex] (3) at (0,2.2){};
\node [vertex] at (0,0){};
\node[] at (0,-1){$D_4$};

\foreach \from/\to in {0/1, 0/3, 1/3}
\draw [edge] (\from) to (\to);

\foreach \from/\to in {3/1, 1/3, 3/0, 0/3}
\path[-]          (\from)  edge   [bend left]   (\to);

\end{scope}

\begin{scope}[xshift=-6cm, yshift = -3cm, scale=0.8]
\node [vertex] (0) at (-1.1,0){};
\node [vertex] (1) at (1.1,0){};
\node [vertex] (3) at (0,2.2){};
\node [vertex] at (0,0){};
\node[] at (0,-1){$D_4'$};

\foreach \from/\to in {0/1, 0/3}
\draw [edge] (\from) to (\to);

\path[-]         (1) edge [min distance=1cm, in = -45, out = 45] (1);

\foreach \from/\to in {3/1, 1/3, 3/0, 0/3}
\path[-]          (\from)  edge   [bend left]   (\to);
\end{scope}

%

\begin{scope}[xshift=-2.5cm, yshift=-3.2cm, scale=0.8]
\coordinate (c0) at (0,0){};
\coordinate (c1) at (0,3){};
\coordinate (d0) at (-1.8,1){};
\coordinate (d1) at (-1.8,4){};
\coordinate (i0) at (1.8,1){};
\coordinate (i1) at (1.8,4){};
\node [vertex] (1) at (-1.2,1.3){};
\node [vertex] at (-1.2,2.2){};
\node [vertex] (2) at (-1.2,3){};
\node [vertex] (11) at (-0.6,1){};
\node [vertex] at (-0.6,1.9){};
\node [vertex] (22) at (-0.6,2.7){};

\node [vertex]  at (1,2.6){};
\node [vertex] at (1,1.6){};
\node[] at (0,-0.8){$D^4$};

\foreach \from/\to in {c0/c1, d0/d1, i0/i1, c0/d0, c0/i0, c1/i1, c1/d1, 1/2, 11/22}
\path[-] (\from) edge (\to);
\end{scope}

\begin{scope}[xshift=-7cm, yshift=-8cm, scale=0.8]
\node [vertex] (0) at (-1,0){};
\node [vertex] (1) at (1,0){};
\node [vertex] (2) at (-1,2){};
\node [vertex] (3) at (1,2){};
\node[] at (0,-1){$R_4$};

\foreach \from/\to in {0/1, 0/3, 1/3}
\draw [edge] (\from) to (\to);

\path[-]          (0)  edge   [bend left =20]   (2);
\path[-]          (0)  edge   [bend right =20]   (2);

\path[-]         (3) edge [min distance=1cm] (3);
\path[-]         (3) edge [min distance=1cm, out=40, in=320] (3);

\end{scope}


\begin{scope}[xshift=-4cm, yshift=-8cm, scale=0.8]
\node [vertex] (0) at (-1,0){};
\node [vertex] (1) at (1,0){};
\node [vertex] (2) at (-1,2){};
\node [vertex] (3) at (1,2){};
\node[] at (0,-1){$R_4'$};

\foreach \from/\to in {0/1, 0/2, 0/3, 1/3}
\draw [edge] (\from) to (\to);

\path[-]         (3) edge [min distance=1cm] (3);
\path[-]         (3) edge [min distance=1cm, out=40, in=320] (3);

\path[-]         (2) edge [min distance=1cm] (2);
\end{scope}


\begin{scope}[xshift=-1cm, yshift=-8cm, scale=0.8]
\node [vertex] (0) at (-1,0){};
\node [vertex] (1) at (1,0){};
\node [vertex] (2) at (-1,2){};
\node [vertex] (3) at (1,2){};
\node[] at (0,-1){$R''_4$};

\foreach \from/\to in {0/1, 0/3, 1/3, 0/2, 2/3}
\draw [edge] (\from) to (\to);

\path[-]         (3) edge [min distance=1cm] (3);
\path[-]         (3) edge [min distance=1cm, out=40, in=320] (3);

\end{scope}

\begin{scope}[xshift=2cm, yshift=-8cm, scale=0.8]
\node [vertex] (0) at (-1,0){};
\node [vertex] (1) at (1,0){};
\node [vertex] (2) at (-1,2){};
\node [vertex] (3) at (1,2){};
\node[] at (0,-1){$R'''_4$};

\foreach \from/\to in {0/1, 1/3, 0/2, 2/3}
\draw [edge] (\from) to (\to);

\path[-]         (3) edge [min distance=1cm] (3);
\path[-]         (3) edge [min distance=1cm, out=40, in=320] (3);
\path[-]         (0) edge [min distance=1cm, out=225, in=315] (0);

\end{scope}


\begin{scope}[xshift=6cm, yshift=-8.2cm, scale=0.8]
\coordinate (c0) at (0,0){};
\coordinate (c1) at (0,3){};
\coordinate (d0) at (-1.8,1){};
\coordinate (d1) at (-1.8,4){};
\coordinate (i0) at (1.8,1){};
\coordinate (i1) at (1.8,4){};
\node [vertex]  at (-0.15,2.3){};
\node [vertex]  at (0.15,2.3){};
\node [vertex] at (0,1){};
\node [vertex] at (-1.2,2.8){};
\node [vertex]  at (1.2,2.8){};
\node [vertex] at (-1.2,1.5){};
\node [vertex] at (1.2,1.5){};
\node[] at (0,-0.8){$R^4$};

\foreach \from/\to in {c0/c1, d0/d1, i0/i1, c0/d0, c0/i0, c1/i1, c1/d1}
\path[-] (\from) edge (\to);
\end{scope}

\end{tikzpicture}

\caption{Six excluded minors to the class of lattice path matroids: $\mathcal{W}^3$,
$B^1$, $R^3$, $D^4$, $A^3$,  and $R^4$. In each case there is an affine
representation, together with all corresponding bicircular presentations.}
\label{fig:bicircularEX}
\end{center}
\end{figure}

The matroids listed in Theorem~\ref{thm:excludedminors} are presented
in Figures~\ref{fig:C24}, \ref{fig:bicircularEX}, and~\ref{fig:BB}. In each
case, we depict a bicircular and an affine representation of these matroids.
With a brief look at Figure 2 in \cite{boninJCTB100}, the reader can realize
that indeed, $C^{2,4}$, $\mathcal{W}^3$, $A^3$, $R^3$, $R^4$, $D^4$ and
$(S^1)^\ast$ are excluded minors for the class of lattice path matroids. 
It is not immediate to conclude from
the description in \cite{boninJCTB100} that $B^1$ is not a lattice path
matroid. Nonetheless,  this can easily be observed using
Proposition~\ref{prop:boninff}.

\begin{proposition}\label{prop:exminors}
Each of the following matroids is a bicircular matroid that is not a lattice
path matroid:
\[
C^{2,4},~\mathcal{W}^3,~A^3,~R^3,~R^4,~D^4,~B^1, \text{ and } S^1.
\]
\end{proposition}
\begin{proof}
From the corresponding illustration, it is clear that  these matroids are
bicircular matroids. As we already mentioned, with a brief look at
\cite{boninJCTB100}, the reader can realize that indeed, $C^{2,4}$,
$\mathcal{W}^3$, $A^3$, $R^3$, $R^4$, $D^4$ and $(S^1)^\ast$
are excluded minors for the class of lattice path matroids. 
Since lattice path matroids are closed under duality \cite{boninEJC27}, then
$S^1$ is an excluded minor for the class of lattice path matroids.
To see that $B^1$ is not a lattice path matroid
it suffices to notice that their corresponding  fundamental flats
do not form two chains ordered by inclusion (Proposition~\ref{prop:boninff}).
For instance, any four points of $B^1$ such that no three of them belong to a
common line,  are a spanning circuit. Thus, each three point line of $B^1$ is spanned
by some spanning circuit, which implies that each three point
line is a fundamental flat of $B^1$. So, there are three
incomparable fundamental flats of $B^1$, which implies that
the fundamental flats of $B^1$ do not form two disjoint chains
ordered by inclusion. Thus, we conclude that $B^1$ is not a lattice path matroid.
\end{proof}


\begin{figure}[ht!]
\begin{center}

\begin{tikzpicture}[scale=0.8]


\begin{scope}[xshift=-6cm, yshift=-3.5cm, scale=0.9]
\node [vertex] (0) at (-1,0){};
\node [vertex] (1) at (1,0){};
\node [vertex] (2) at (-1,2){};
\node [vertex] (3) at (1,2){};
\node [vertex] (s) at (0,1){};
\node[] at (0,-1){$S_1$};

\foreach \from/\to in {0/1, 3/2, 0/s, s/3}
\draw [edge] (\from) to (\to);

\path[-]          (0)  edge   [bend left]   (2);
\path[-]          (0)  edge   [bend right]   (2);
\path[-]          (1)  edge   [bend right]   (3);
\path[-]          (1)  edge   [bend left]   (3);
\end{scope}

\begin{scope}[xshift=-2.5cm, yshift=-3.5cm, scale=0.9]
\node [vertex] (0) at (-1,0){};
\node [vertex] (1) at (1,0){};
\node [vertex] (2) at (-1,2){};
\node [vertex] (3) at (1,2){};
\node [vertex] (s) at (0,1){};
\node[] at (0,-1){$S_1'$};

\foreach \from/\to in {0/1, 3/2, 0/s, s/3, 0/2}
\draw [edge] (\from) to (\to);

\path[-]          (2)  edge   [min distance = 1cm]   (2);
\path[-]          (1)  edge   [bend right]   (3);
\path[-]          (1)  edge   [bend left]   (3);
\end{scope}

\begin{scope}[xshift=1cm, yshift=-3.5cm, scale=0.9]
\node [vertex] (0) at (-1,0){};
\node [vertex] (1) at (1,0){};
\node [vertex] (2) at (-1,2){};
\node [vertex] (3) at (1,2){};
\node [vertex] (s) at (0,1){};
\node[] at (0,-1){$S_1''$};

\foreach \from/\to in {0/1, 3/2, 0/s, s/3, 0/2, 1/3}
\draw [edge] (\from) to (\to);

\path[-]          (2)  edge   [min distance = 1cm]   (2);
\path[-]          (1)  edge   [min distance = 1cm, out = 225, in = 315]   (1);
\end{scope}


\begin{scope}[xshift=5cm, yshift=-3.3cm, scale=0.9]
\coordinate (c0) at (-1.5,0){};
\coordinate (c1) at (1.5,0){};
\coordinate (i0) at (-1.5,1.5){};
\coordinate (i1) at (1.5,1.5){};
\node [vertex] (0) at (-1.5,0){};
\node [vertex] (1) at (1.5,0){};
\node [vertex] (2) at (0,0){};
\node [vertex] (1) at (0.1,0.75){};
\node [vertex] (2) at (-0.1,0.75){};
\node [vertex] (0) at (-1.5,1.5){};
\node [vertex] (1) at (1.5,1.5){};
\node [vertex] (2) at (0,1.5){};
\node[] at (0,-1.2){$(S^1)^\ast$};

\foreach \from/\to in {c0/c1, i0/i1}
\path[-] (\from) edge (\to);
\end{scope}

\end{tikzpicture}

\caption{To the left, three bicircular presentations of $S^1$. Since
this is a rank $5$ matroid, we  choose to present an affine representation of its
dual $(S^1)^\ast$. $S^1$ is an excluded bicircular minor to lattice path matroids.}
\label{fig:BB}
\end{center}
\end{figure}
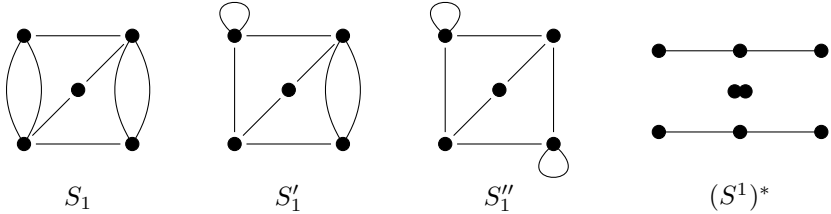


\section{Graph families}
\label{sec:GF}

In this section, we define five graph families $\mathcal{F}_i$ with
$i\in\{0,\dots,4\}$ and we show that the graphs in these
families have lattice path bicircular matroids. 
Throughout this section, we will heavily rely on Lemma~\ref{lem:Gsubdivision}
to argue that subdivisions of certain graphs have  lattice path bicircular matroids. 
In particular, we will use that any series extension
of two elements of a uniform matroid is a lattice path matroid.
We will also rely on Lemma~\ref{lem:triangles}. This lemma asserts that if a
graph $G$ has a  lattice path bicircular matroid,  then $G+f$ has a lattice path
bicircular  matroid whenever $f$ is added to a parallel class that contains three
different edges.

\subsection*{\normalfont{\textsc{Family $\mathcal{F}_0$}}}

We begin with the simplest construction. Consider the graphs
$K'_{2,3}$, $K''_{2,3}$ and $K^\ast_{2,3}$ obtained from $K_{2,3}$ by
adding certain parallel edges or a loop  as follows.

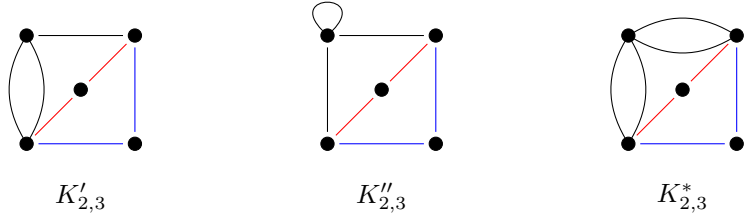
\begin{center}
\begin{tikzpicture}[scale=0.8]

\begin{scope}[xshift=-5cm, yshift=-3.5cm, scale=0.9]
\node [vertex] (0) at (-1,0){};
\node [vertex] (1) at (1,0){};
\node [vertex] (2) at (-1,2){};
\node [vertex] (3) at (1,2){};
\node [vertex] (s) at (0,1){};
\node[] at (0,-1){$K'_{2,3}$};

\draw [edge] (2) to (3);

\foreach \from/\to in {0/1, 1/3}
\draw [edge, blue] (\from) to (\to);

\foreach \from/\to in { 0/s, s/3}
\draw [edge, red] (\from) to (\to);

\path[-]          (0)  edge   [bend left]   (2);
\path[-]          (0)  edge   [bend right]   (2);
\end{scope}

\begin{scope}[xshift=0cm, yshift=-3.5cm, scale=0.9]
\node [vertex] (0) at (-1,0){};
\node [vertex] (1) at (1,0){};
\node [vertex] (2) at (-1,2){};
\node [vertex] (3) at (1,2){};
\node [vertex] (s) at (0,1){};
\node[] at (0,-1){$K''_{2,3}$};

\foreach \from/\to in {0/2, 2/3}
\draw [edge] (\from) to (\to);

\foreach \from/\to in {0/1, 1/3}
\draw [edge, blue] (\from) to (\to);

\foreach \from/\to in { 0/s, s/3}
\draw [edge, red] (\from) to (\to);

\path[-]          (2)  edge   [min distance = 1cm]   (2);
\end{scope}

\begin{scope}[xshift=5cm, yshift=-3.5cm, scale=0.9]
\node [vertex] (0) at (-1,0){};
\node [vertex] (1) at (1,0){};
\node [vertex] (2) at (-1,2){};
\node [vertex] (3) at (1,2){};
\node [vertex] (s) at (0,1){};
\node[] at (0,-1){$K^\ast_{2,3}$};

\foreach \from/\to in {0/1, 1/3}
\draw [edge, blue] (\from) to (\to);

\foreach \from/\to in {0/s, s/3}
\draw [edge, red] (\from) to (\to);

\path[-]          (0)  edge   [bend left]   (2);
\path[-]          (0)  edge   [bend right]   (2);
\path[-]          (3)  edge   [bend left]   (2);
\path[-]          (3)  edge   [bend right]   (2);
\end{scope}
\end{tikzpicture}
\end{center}
For the sake of clarity in the upcoming arguments, we distinguish
some edges (red and blue) of $K'_{2,3}$, of $K''_{2,3}$ and of $K^\ast_{2,3}$
as depicted above (we will use this technique through the rest of this work).
A graph $G$ belongs to $\mathcal{F}_0$ if either of the following hold:
\begin{enumerate}
	\item $G$ is a subdivision of $K_{2,3}$,
	\item $G$ is a subdivision of blue and/or red edges of $K'_{2,3}$ of
	$K''_{2,3}$ or of $K^\ast_{2,3}$, or
	\item $G$ is a subdivision of at most two edges of $K_4$.
\end{enumerate}

In particular, notice that any subdivision of $K_{2,3}$ is a theta graph. 
Thus, $B(G)$ is the uniform matroid of rank $|E(G)|-1$ and $|E(G)|$ elements. 
Clearly then, $B(G)$ is a lattice path matroid. Also, it is not hard to notice
that the bicircular matroid of $K_4$ is $U_{4,6}$. Thus, by
Lemma~\ref{lem:Gsubdivision}, the bicircular matroid of any subdivision $G'$
of any two edges of $K_4$  is a lattice path matroid.

\begin{proposition}\label{prop:familyF0}
The bicircular matroid of each graph in $\mathcal{F}_0$ is a lattice path matroid.
\end{proposition}
\begin{proof}
Given the arguments above, we only need to show that
when $G$ is a subdivision of blue and red edges of $K'_{2,3}$,
of $K''_{2,3}$, or of $K^\ast_{2,3}$
then $B(G)$ is a lattice path matroid. The last follows by noticing that
the bicircular matroid $B(K^\ast_{2,3})$ is a series extension
of $U_{3,6}$. Indeed, if we contract each pair of red and blue edges
to a single red and blue edge, respectively, we obtain a graph
$2K_3$ whose bicircular matroid is $U_{3,6}$. Thus, the bicircular matroid
of any subdivision $G$ of red and blue edges of $K^\ast_{2,3}$ is a series
extension of two elements of $U_{3,6}$. By Lemma~\ref{lem:Gsubdivision},
we conclude that $B(G)$ is a lattice path matroid.

It is not hard to notice that there is a color preserving isomorphism from
$B(K'_{2,3})$ to $B(K''_{2,3})$. So, to conclude this proof, it
suffices to consider the case when $G$ is a subdivision of $K'_{2,3}$.
Now, we propose the following lattice path presentation of the
bicircular matroid of $K'_{2,3}$.
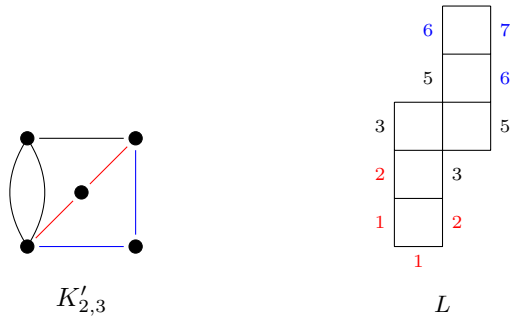
\begin{center}
\begin{tikzpicture}[scale=0.8]

\begin{scope}[xshift=-3cm, scale=0.9]
\node [vertex] (0) at (-1,0){};
\node [vertex] (1) at (1,0){};
\node [vertex] (2) at (-1,2){};
\node [vertex] (3) at (1,2){};
\node [vertex] (s) at (0,1){};
\node[] at (0,-1){$K'_{2,3}$};

\draw [edge] (2) to (3);

\foreach \from/\to in {0/1, 1/3}
\draw [edge, blue] (\from) to (\to);

\foreach \from/\to in { 0/s, s/3}
\draw [edge, red] (\from) to (\to);

\path[-]          (0)  edge   [bend left]   (2);
\path[-]          (0)  edge   [bend right]   (2);
\end{scope}
\begin{scope}[xshift=3cm, scale = 0.8]
\draw[step=1cm] (-1,0) grid (0,3);
\draw[step=1cm] (0,2) grid (1,5);
\node[] at (0,-1.2){\small{$L$}};

\color{red}
\node[] at (-1.3,0.5){\scriptsize{$1$}};
\node[] at (-1.3,1.5){\scriptsize{$2$}};

\node[] at (-0.5,-0.3){\scriptsize{$1$}};
\node[] at (0.3,0.5){\scriptsize{$2$}};

\color{black}
\node[] at (-1.3,2.5){\scriptsize{$3$}};
\node[] at (-0.3, 3.5){\scriptsize{$5$}};

\node[] at (0.3,1.5){\scriptsize{$3$}};
\node[] at (1.3,2.5){\scriptsize{$5$}};

\color{blue}
\node[] at (1.3,3.5){\scriptsize{$6$}};
\node[] at (1.3,4.5){\scriptsize{$7$}};

\node[] at (-0.3,4.5){\scriptsize{$6$}};

\color{black}
\end{scope}

\end{tikzpicture}
\end{center}
It is not hard to notice that any color preserving labeling of the edges
of $K'_{2,3}$ with $\{1,\dots, 7\}$, yields an isomorphism
from $B(K'_{2,3})$ to the lattice path matroid $L$. Moreover, this isomorphism
defines an interval ordering of $E(K'_{2,3})$, where the minimum
is a red edge, and the maximum is a blue edge. Thus, by
Lemma~\ref{lem:Gsubdivision}, if $G$ is obtained from $K'_{2,3}$ by 
subdividing any blue and any red edge, then $B(G)$ is a lattice path matroid. 
Therefore, the statement of this proposition holds true.
\end{proof}

\subsection*{\normalfont{\textsc{Family $\mathcal{F}_1$}}}

Our most basic building block for the family $\mathcal{F}_1$
is the graph $G_1$, obtained from the $4$ cycle
by duplicating two non consecutive edges. 
We distinguish the parallel edges by thinking of one class as blue
edges, and the other as red edges.
Given three non-negative integers $r$, $b$ and $d$,
we denote the graph $G_1(r,d,b)$ obtained from $G_1$ by adding
$r$ red parallel edges (to the red edges), $b$ blue parallel edges
(to the blue edges), and $d$ parallel diagonal edges.
From left to right,  the following is a depiction of $G_1 = G_1(0,0,0)$, of $G_1(1,3,1)$,
and $G(r,d,b)$, where the dashed edges represent $r$, $d$ and $b$ parallel copies.
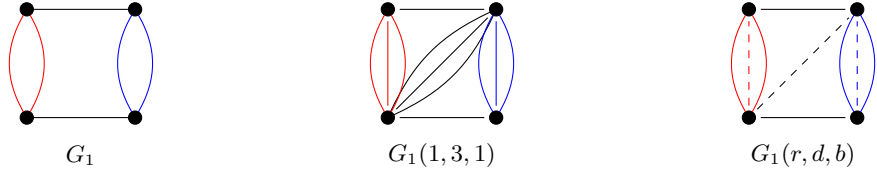
\begin{center}
\begin{tikzpicture}[scale=0.8]

\begin{scope}[xshift=-6cm, yshift=0cm, scale=0.9]
\node [vertex] (0) at (-1,0){};
\node [vertex] (1) at (1,0){};
\node [vertex] (2) at (-1,2){};
\node [vertex] (3) at (1,2){};
\node[] at (0,-0.7){\small{$G_1$}};

\path[-]          (3)  edge    (2);
\path[-]          (0)  edge     (1);
\path[-, red]          (0)  edge   [bend left]   (2);
\path[-, red]          (0)  edge   [bend right]   (2);
\path[-, blue]          (1)  edge   [bend right]   (3);
\path[-, blue]          (1)  edge   [bend left]   (3);
\end{scope}
\begin{scope}[xshift=0cm, yshift=0cm, scale=0.9]
\node [vertex] (0) at (-1,0){};
\node [vertex] (1) at (1,0){};
\node [vertex] (2) at (-1,2){};
\node [vertex] (3) at (1,2){};
\node[] at (0,-0.7){\small{$G_1(1,3,1)$}};

\foreach \from/\to in {3/2, 0/1}
\draw [edge] (\from) to (\to);

\path[-]          (0)  edge   [bend left = 20]   (3);
\path[-]          (0)  edge   [bend right = 20]   (3);
\draw [edge] (0) to (3);

\path[-, red]          (0)  edge   [bend left]   (2);
\path[-, red]          (0)  edge   [bend right]   (2);
\draw [edge, red] (0) to (2);

\path[-, blue]          (1)  edge   [bend right]   (3);
\path[-, blue]          (1)  edge   [bend left]   (3);
\draw [edge, blue] (1) to (3);
\end{scope}

\begin{scope}[xshift=6cm, yshift=0cm, scale=0.9]
\node [vertex] (0) at (-1,0){};
\node [vertex] (1) at (1,0){};
\node [vertex] (2) at (-1,2){};
\node [vertex] (3) at (1,2){};
\node[] at (0,-0.7){\small{$G_1(r,d,b)$}};

\foreach \from/\to in {3/2, 0/1}
\draw [edge] (\from) to (\to);

\draw [edge, dashed] (0) to (3);

\path[-, red]          (0)  edge   [bend left]   (2);
\path[-, red]          (0)  edge   [bend right]   (2);
\draw [edge, dashed, red] (0) to (2);

\path[-, blue]          (1)  edge   [bend right]   (3);
\path[-, blue]          (1)  edge   [bend left]   (3);
\draw [edge, dashed, blue] (1) to (3);
\end{scope}

\end{tikzpicture}
\end{center}

The colors on the edges of $G_1$ are useful to define the family $G_1(r,d,b)$. 
This colors are also convenient to define the family $\mathcal{F}_1$,
but in this case, we extend this coloring when either of $r$, $d$, or $b$
equal $0$.  In Figure~\ref{fig:H1}, we illustrate these edge colorings.
With these figures in mind, we define the family $\mathcal{F}_1$.
A graph $G$ belongs to $\mathcal{F}_1$ if the following holds:
\begin{enumerate}
	\item $G$ is a subdivision of
		at most one red edge and at most one blue edge of
		$G(r,d,b)$ for some non-negative integers $r$, $d$ and $b$
		(see Figure~\ref{fig:H1}).
\end{enumerate}

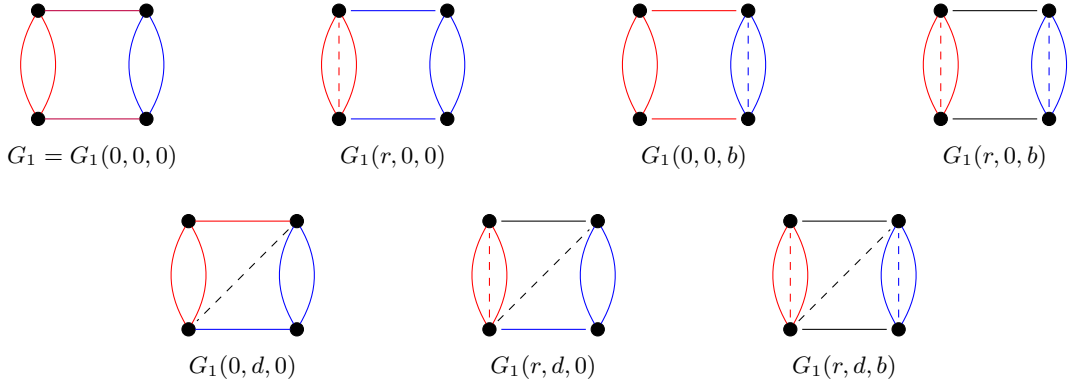
\begin{figure}[ht!]
\begin{center}

\begin{tikzpicture}[scale=0.8]

\begin{scope}[xshift=-7.5cm, yshift=3.5cm, scale=0.9]
\node [vertex] (0) at (-1,0){};
\node [vertex] (1) at (1,0){};
\node [vertex] (2) at (-1,2){};
\node [vertex] (3) at (1,2){};
\node[] at (0,-0.7){\small{$G_1 = G_1(0,0,0)$}};

\path[-, purple]       (3)  edge    (2);
\path[-, purple]         (0)  edge     (1);
\path[-, red]         (0)  edge   [bend left]   (2);
\path[-, red]          (0)  edge   [bend right]   (2);
\path[-, blue]          (1)  edge   [bend right]   (3);
\path[-, blue]         (1)  edge   [bend left]   (3);
\end{scope}

\begin{scope}[xshift=-5cm, yshift=0cm, scale=0.9]
\node [vertex] (0) at (-1,0){};
\node [vertex] (1) at (1,0){};
\node [vertex] (2) at (-1,2){};
\node [vertex] (3) at (1,2){};
\node[] at (0,-0.7){\small{$G_1(0,d,0)$}};

\foreach \from/\to in {3/0}
\draw [edge, dashed] (\from) to (\to);

\path[-, red]          (3)  edge    (2);
\path[-, blue]          (0)  edge     (1);
\path[-, red]          (0)  edge   [bend left]   (2);
\path[-, red]          (0)  edge   [bend right]   (2);
\path[-, blue]          (1)  edge   [bend right]   (3);
\path[-, blue]          (1)  edge   [bend left]   (3);
\end{scope}

\begin{scope}[xshift=-2.5cm, yshift=3.5cm, scale=0.9]

\node [vertex] (0) at (-1,0){};
\node [vertex] (1) at (1,0){};
\node [vertex] (2) at (-1,2){};
\node [vertex] (3) at (1,2){};
\node[] at (0,-0.7){\small{$G_1(r,0,0)$}};

\foreach \from/\to in {0/1,3/2}
\draw [edge, blue] (\from) to (\to);

\path[-, red]          (0)  edge   [bend left]   (2);
\path[-, red]          (0)  edge   [bend right]   (2);
\draw [edge, red, dashed] (0) to (2);

\path[-, blue]          (1)  edge   [bend right]     (3);
\path[-, blue]          (1)  edge   [bend left]    (3);
\end{scope}

\begin{scope}[xshift=0cm, yshift=0cm, scale=0.9]
\node [vertex] (0) at (-1,0){};
\node [vertex] (1) at (1,0){};
\node [vertex] (2) at (-1,2){};
\node [vertex] (3) at (1,2){};
\node[] at (0,-0.7){\small{$G_1(r,d,0)$}};

\draw [edge] (3) to (2);
\draw [edge, dashed] (0) to (3);

\foreach \from/\to in {0/1}
\draw [edge, blue] (\from) to (\to);

\path[-, red]          (0)  edge   [bend left]   (2);
\path[-, red]          (0)  edge   [bend right]   (2);
\draw [edge, red, dashed] (0) to (2);

\path[-, blue]          (1)  edge   [bend right]     (3);
\path[-, blue]          (1)  edge   [bend left]    (3);
\end{scope}

\begin{scope}[xshift=2.5cm, yshift=3.5cm, scale=0.9]
\node [vertex] (0) at (-1,0){};
\node [vertex] (1) at (1,0){};
\node [vertex] (2) at (-1,2){};
\node [vertex] (3) at (1,2){};
\node[] at (0,-0.7){\small{$G_1(0,0,b)$}};

\foreach \from/\to in {3/2, 0/1}
\draw [edge, red] (\from) to (\to);

\path[-, red]          (0)  edge   [bend left]   (2);
\path[-, red]          (0)  edge   [bend right]   (2);

\path[-, blue]          (1)  edge   [bend right]   (3);
\path[-, blue]          (1)  edge   [bend left]   (3);
\draw [edge, dashed, blue] (1) to (3);
\end{scope}

\begin{scope}[xshift=7.5cm, yshift=3.5cm, scale=0.9]
\node [vertex] (0) at (-1,0){};
\node [vertex] (1) at (1,0){};
\node [vertex] (2) at (-1,2){};
\node [vertex] (3) at (1,2){};
\node[] at (0,-0.7){\small{$G_1(r,0,b)$}};

\foreach \from/\to in {3/2, 0/1}
\draw [edge] (\from) to (\to);

\path[-, red]          (0)  edge   [bend left]   (2);
\path[-, red]          (0)  edge   [bend right]   (2);
\draw [edge, dashed, red] (0) to (2);

\path[-, blue]          (1)  edge   [bend right]   (3);
\path[-, blue]          (1)  edge   [bend left]   (3);
\draw [edge, dashed, blue] (1) to (3);
\end{scope}

%
\begin{scope}[xshift=5cm, yshift=0cm, scale=0.9]
\node [vertex] (0) at (-1,0){};
\node [vertex] (1) at (1,0){};
\node [vertex] (2) at (-1,2){};
\node [vertex] (3) at (1,2){};

\node[] at (0,-0.7){\small{$G_1(r,d,b)$}};

\foreach \from/\to in {3/2, 0/1}
\draw [edge] (\from) to (\to);
\draw [edge, dashed] (0) to (3);

\path[-, red]          (0)  edge   [bend left]   (2);
\path[-, red]          (0)  edge   [bend right]   (2);
\draw [edge, dashed, red] (0) to (2);

\path[-, blue]          (1)  edge   [bend right]   (3);
\path[-, blue]          (1)  edge   [bend left]   (3);
\draw [edge, dashed, blue] (1) to (3);
\end{scope}

\end{tikzpicture}
\caption{The generating set of graphs for $\mathcal{F}_1$.
Each dashed edge represents $r$, $b$ and $d$ red, blue  and diagonal
parallel edges,  respectively. Non-bended edges of $G_1$ are both blue and red.}
\label{fig:H1}
\end{center}
\end{figure}

Notice that $G_1(r,d,b)$ is obtained by adding an edge to a parallel
class to either of $G_1(r-1,d,b)$, $G_1(r,d-1,b)$ or $G_1(r,d,b-1)$. 
Thus, by Lemma~\ref{lem:triangles}, we conclude that if $G_1(1,3,1)$ has a
lattice path bicircular matroid, then  $G_1(r,d,b)$  has a lattice path bicircular
matroid for every $r,b\ge 1$ and $d\ge 3$. Furthermore, it is obvious that
if $G_1(r,d,b)$ has a lattice path  bicircular matroid, then $G_1(r',d',b')$
has a lattice path bicircular matroid for any $r'\le r$, $d'\le d$ and $b'\le b$. 
Both of these observations together, show that if $G_1(1,3,1)$ has a
lattice path bicircular matroid, then $G_1(r,d,b)$ has a lattice path bicircular matroid
for all non-negative integers $r$, $d$ and $b$.

\begin{lemma}\label{lem:Grdb}
For all non-negative integers $r$, $d$ and $b$,  the graph $G_1(r,d,b)$  has a
lattice path bicircular matroid, with an interval ordering where the minimum
(resp.\ maximum) element is a red (resp.\ blue) edge.
\end{lemma}
\begin{proof}
Suppose that $B(G_1(1,3,1))$  is a lattice path matroid. By the arguments preceding
this paragraph,  we know that $B(G_1(r,d,b))$ is a lattice path matroid
for all non-negative integers $r$, $d$ and $b$. Furthermore, 
with the same arguments and the ``moreover'' statement of Lemma~\ref{lem:triangles},
we conclude that if $B(G_1(1,3,1))$ has an interval ordering where
the minimum (resp.\ maximum) element is a red (resp.\ blue) edge, 
then so does $B(G_1(r,d,b))$. Thus, it suffices to show that this lemma
holds for $r = 1$, $d = 3$ and $b = 1$. This follows from the lattice path
presentation of $B(G_1(1,3,1))$ in Figure~\ref{fig:LatG131}.
\end{proof}

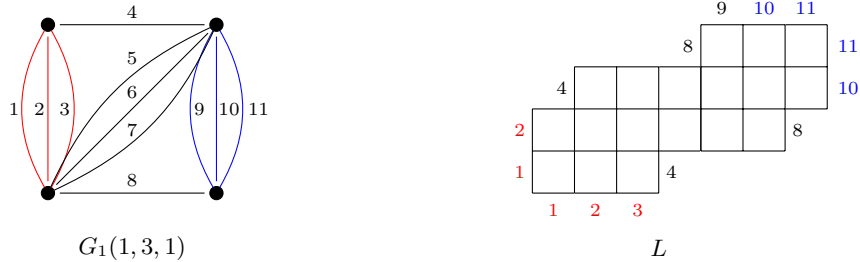
\begin{figure}[ht!]
\begin{center}

\begin{tikzpicture}[scale=0.7]

\begin{scope}[xshift=-5cm, yshift=6cm, scale=0.8]
\node [vertex] (0) at (-2,0){};
\node [vertex] (1) at (2,0){};
\node [vertex] (2) at (-2,4){};
\node [vertex] (3) at (2,4){};
\node[] at (-2.8,2){\scriptsize{$1$}};
\node[] at (-2.2,2){\scriptsize{$2$}};
\node[] at (-1.6,2){\scriptsize{$3$}};
\node[] at (0,4.3){\scriptsize{$4$}};
\node[] at (0,3.2){\scriptsize{$5$}};
\node[] at (0,2.4){\scriptsize{$6$}};
\node[] at (0,1.5){\scriptsize{$7$}};
\node[] at (0,0.3){\scriptsize{$8$}};
\node[] at (1.6,2){\scriptsize{$9$}};
\node[] at (2.3,2){\scriptsize{$10$}};
\node[] at (3,2){\scriptsize{$11$}};
\node[] at (0,-1.3){\small{$G_1(1,3,1)$}};

\foreach \from/\to in {0/1, 2/3, 0/3}
\draw [edge] (\from) to (\to);

\path[-, red]          (0)  edge   [bend left]   (2);
\path[-, red]          (0)  edge   [bend right]   (2);
\draw[edge, red]          (0)  to   (2);
\path[-, blue]          (1)  edge   [bend left]   (3);
\path[-, blue]          (1)  edge   [bend right]   (3);
\draw[edge, blue]          (1)  to   (3);
\path[-]          (0)  edge   [bend right =20]   (3);
\path[-]          (0)  edge   [bend left =20]   (3);

\end{scope}

\begin{scope}[xshift=5cm, yshift = 6cm, scale=0.8]

\draw[step=1cm] (-2,0) grid (0,3);
\draw[step=1cm] (0.01,1) grid (0.98,3);
\draw[step=1cm] (1,1) grid (3,4);
\draw[step=1cm] (-3,0) grid (-2.01,2);
\draw[step=1cm] (3.01,2) grid (4,4);

\color{red}
\node[] at (-3.3,0.5){\scriptsize{$1$}};
\node[] at (-3.3,1.5){\scriptsize{$2$}};
\node[] at (-2.5,-0.4){\scriptsize{$1$}};
\node[] at (-1.5,-0.4){\scriptsize{$2$}};
\node[] at (-0.5,-0.4){\scriptsize{$3$}};
\color{black}
\node[] at (-2.3,2.5){\scriptsize{$4$}};
\node[] at (0.3,0.5){\scriptsize{$4$}};
\node[] at (0.7,3.5){\scriptsize{$8$}};
\node[] at (3.3,1.5){\scriptsize{$8$}};
\node[] at (1.5,4.4){\scriptsize{$9$}};
\color{blue}
\node[] at (4.5,2.5){\scriptsize{$10$}};
\node[] at (4.5,3.5){\scriptsize{$11$}};

\node[] at (2.5,4.4){\scriptsize{$10$}};
\node[] at (3.5,4.4){\scriptsize{$11$}};
\color{black}

\node[] at (0,-1.3){\small{$L$}};

\end{scope}
\end{tikzpicture}
\caption{The graph $G_1(1,3,1)$ together with a lattice path
matroid $L$, where the labels describe an isomorphism $\varphi\colon
B(G_1(1,3,1))\to L$. Indeed, it is not hard to observe that $\varphi$ defines
a bijection between the families of circuits of rank at most $3$, and  that it
defines a bijection on the collections of bases. Since
$B(G_1(1,3,1))$ and $L$ are rank $4$ matroids, we conclude
that $\varphi$ defines a bijection on the family of all circuits.}
\label{fig:LatG131}
\end{center}
\end{figure}


Recall that a pair of elements in a matroid $M$ are clones, if permuting
these elements yields an automorphism of $M$. 
It is not hard to see for each graph depicted in Figure~\ref{fig:H1},
the blue and the red edges correspond to a pair of classes of clone elements
in the corresponding bicircular matroid. Thus, for any of these
graphs, any two subdivisions, $G'$ and $G''$, of at most
one red edge and at most one blue edge (into the same number of
blue and red edges, respectively) it holds that $B(G')\cong B(G'')$.

\begin{proposition}\label{prop:familyF1}
The bicircular matroid of each graph in $\mathcal{F}_1$ is a lattice path matroid.
\end{proposition}
\begin{proof}
By Lemma~\ref{lem:Grdb}, the bicircular matroid of each graph
$G_1(r,d,b)$ has an interval ordering where the minimum and maximum
elements are a red and a blue edge, respectively. By the arguments
in the paragraph above this lemma, any two subdivisions, $G'$ and $G''$,
of at most one red edge and at most one blue edge (into the same number
of blue and red edges, respectively) it holds that $B(G')\cong B(G'')$. 
Since we can choose an interval ordering of $B(G_1(r,d,b))$
where the minimum (resp.\ maximum) element is a red (resp.\ blue)
edge, we conclude by Lemma~\ref{lem:Gsubdivision}, that
if $G'$ is a subdivision of at most one red edge and at most one blue
edge of $G_1(r,d,b)$, then $B(G')$ is a lattice path matroid. 
Recall that each graph in
$\mathcal{F}_1$ is obtained by subdividing at most one red edge
and at  most and one blue edge of some $G_1(r,d,b)$. The claim
follows. 
\end{proof}

\subsection*{\normalfont{\textsc{Family  $\mathcal{F}_2$}}}

This family is constructed in a similar fashion to the construction of
$\mathcal{F}_1$.
Denote by $2K_3$ the graph obtained from $K_3$
by duplicating each edge. Again, we distinguish a pair of parallel
classes, one by color red and the other one by blue. 
Similar to how we did in the previous subsection, for a pair of non-negative
integers $r$ and $b$, we denote by $2K_3(r,b)$ the graph obtained from 
$2K_3$ by adding $r$ red parallel edges (to red edges) and
$b$ blue parallel edge (to blue edges).  From left to right, 
the following is a depiction of $2K_3 = 2K_3(0,0)$, of $2K_3(1,1)$,
and of $2K_3(r,b)$, where the dashed edges represent $r$ and $b$
parallel copies.

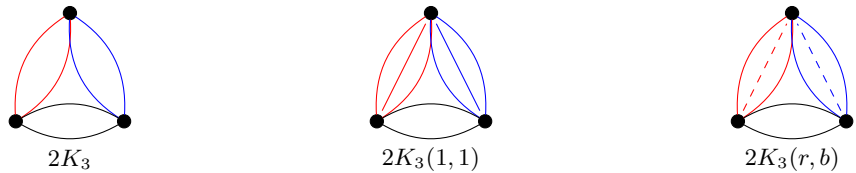
\begin{center}
\begin{tikzpicture}[scale=0.8]

\begin{scope}[xshift=-6cm, yshift=0cm, scale=0.9]
\node [vertex] (0) at (-1,0){};
\node [vertex] (1) at (1,0){};
\node [vertex] (2) at (0,2){};
\node[] at (0,-0.7){\small{$2K_3$}};

\path[-, red]          (0)  edge   [bend left]   (2);
\path[-, red]          (0)  edge   [bend right]   (2);
\path[-, blue]          (1)  edge   [bend right]   (2);
\path[-, blue]          (1)  edge   [bend left]   (2);
\path[-]          (1)  edge   [bend right]   (0);
\path[-]          (1)  edge   [bend left]   (0);

\end{scope}

\begin{scope}[xshift=0cm, yshift=0cm, scale=0.9]
\node [vertex] (0) at (-1,0){};
\node [vertex] (1) at (1,0){};
\node [vertex] (2) at (0,2){};
\node[] at (0,-0.7){\small{$2K_3(1,1)$}};

\path[-, red]          (0)  edge   [bend left]   (2);
\path[-, red]          (0)  edge   [bend right]   (2);
\path[-, blue]          (1)  edge   [bend right]   (2);
\path[-, blue]          (1)  edge   [bend left]   (2);
\path[-]          (1)  edge   [bend right]   (0);
\path[-]          (1)  edge   [bend left]   (0);
\draw[edge, blue] (1) to (2);
\draw[edge, red] (0) to (2);

\end{scope}

\begin{scope}[xshift=6cm, yshift=0cm, scale=0.9]
\node [vertex] (0) at (-1,0){};
\node [vertex] (1) at (1,0){};
\node [vertex] (2) at (0,2){};
\node[] at (0,-0.7){\small{$2K_3(r,b)$}};

\path[-, red]          (0)  edge   [bend left]   (2);
\path[-, red]          (0)  edge   [bend right]   (2);
\path[-, blue]          (1)  edge   [bend right]   (2);
\path[-, blue]          (1)  edge   [bend left]   (2);
\path[-]          (1)  edge   [bend right]   (0);
\path[-]          (1)  edge   [bend left]   (0);

\draw[edge, blue, dashed] (1) to (2);
\draw[edge, red, dashed] (0) to (2);

\end{scope}
\end{tikzpicture}
\end{center}
As previously mentioned, we construct $\mathcal{F}_2$ in a similar
manner to how we constructed $\mathcal{F}_1$. A graph $G$
belongs to $\mathcal{F}_2$ if the following holds:
\begin{enumerate}
	\item $G$ is a subdivision of at most one red edge and at most one
	blue edge of $2K_3(r,b)$ for some non-negative integers $r$ and $b$.
\end{enumerate}

\begin{proposition}\label{prop:familyF2}
The bicircular matroid of each graph in $\mathcal{F}_2$ is a lattice path matroid.
\end{proposition}
\begin{proof}
We begin by proposing the following lattice path presentation of
$B(2K_3(1,1))$.

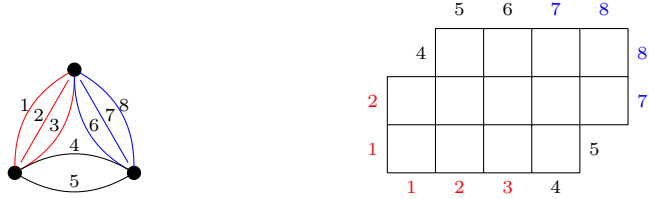
\begin{center}

\begin{tikzpicture}[scale=0.8]

\begin{scope}[xshift=-3cm, yshift=5cm, scale=0.66]
\node [vertex] (0) at (-1.5,0){};
\node [vertex] (1) at (1.5,0){};
\node [vertex] (2) at (0,2.6){};

\node[] at (-1.25,1.7){\scriptsize{$1$}};
\node[] at (-0.9,1.45){\scriptsize{$2$}};
\node[] at (-0.5,1.2){\scriptsize{$3$}};
\node[] at (0,0.7){\scriptsize{$4$}};
\node[] at (0,-0.2){\scriptsize{$5$}};
\node[] at (0.5,1.2){\scriptsize{$6$}};
\node[] at (0.9,1.45){\scriptsize{$7$}};
\node[] at (1.25,1.7){\scriptsize{$8$}};

\path[-, red]          (0)  edge   [bend left]   (2);
\path[-, red]          (0)  edge   [bend right]   (2);
\draw[edge, red] (0) to (2);

\path[-, blue]          (1)  edge   [bend right]   (2);
\path[-, blue]          (1)  edge   [bend left]   (2);
\draw[edge, blue] (1) to (2);

\path[-]          (1)  edge   [bend right]   (0);
\path[-]          (1)  edge   [bend left]   (0);
\end{scope}

\begin{scope}[xshift=3cm, yshift = 5cm, scale=0.8]

\draw[step=1cm] (0,0) grid (3,3);
\draw[step=1cm] (-1,0) grid (-0.01,2);
\draw[step=1cm] (3.01,1) grid (4,3);

\color{red}
\node[] at (-1.3,0.5){\scriptsize{$1$}};
\node[] at (-1.3,1.5){\scriptsize{$2$}};

\node[] at (-0.5,-0.3){\scriptsize{$1$}};
\node[] at (0.5,-0.3){\scriptsize{$2$}};
\node[] at (1.5,-0.3){\scriptsize{$3$}};

\color{black}
\node[] at (-0.3,2.5){\scriptsize{$4$}};
\node[] at (2.5,-0.3){\scriptsize{$4$}};

\node[] at (3.3,0.5){\scriptsize{$5$}};

\node[] at (0.5,3.4){\scriptsize{$5$}};
\node[] at (1.5,3.4){\scriptsize{$6$}};

\color{blue}
\node[] at (4.3,1.5){\scriptsize{$7$}};
\node[] at (4.3,2.5){\scriptsize{$8$}};
\node[] at (2.5,3.4){\scriptsize{$7$}};
\node[] at (3.5,3.4){\scriptsize{$8$}};
\color{black}
\end{scope}
\end{tikzpicture}

\label{fig:2K3}
\end{center}
Denote by $J$ the lattice path matroid on the right. One can notice that
$B(2K_3(1,1))$ and $J$ are isomorphic matroids because,  the only
non-spanning circuits in both matroids are $123$ and $678$.
This lattice path presentation of $B(2K_3(1,1))$ has an interval ordering
where the minimum (resp.\ maximum)
element  is a red (resp.\ blue) edge. 
Clearly,  $2K_3(r,b)$ is obtained by adding an edge to a parallel class of
$2K_3(r-1,b)$ or of $2K_3(r,b-1)$. Thus, by Lemma~\ref{lem:triangles}, we
conclude that  $B(2K_3(r,b))$ is a lattice path matroid, and has an interval
ordering where the minimum (resp.\ maximum) is a red (resp.\ blue) edge.

As mentioned before, any class of parallel edges is a set of clone edges. Thus, 
any two subdivisions of $2K_3(r,b)$ of at most
one red edge and at most one blue edge (into the same number of
blue and red edges, respectively) have isomorphic bicircular matroids. 
By the arguments in the first paragraph, we can choose an interval ordering
of $B(2K_3(r,b))$ where the minimum (resp.\ maximum) element is a red
(resp.\ blue) edge. Thus, we conclude by Lemma~\ref{lem:Gsubdivision}, that
if $G'$ is a subdivision of at most one red edge and at most one blue
edge of $2K_3(r,b)$, then $B(G')$ is a lattice path matroid. 
The claim follows.
\end{proof}

\subsection*{\normalfont{\textsc{Family $\mathcal{F}_3$}}}

Similar to previous cases, we begin by introducing the building blocks
of $\mathcal{F}_3$. Consider three non-negative
integers $r$, $j$ and $l$. We denote by $K_3(r,j,l)$ the graph
obtained from $K_3$ by adding $l$ loops incident with the same vertex $v$; by 
adding $j$ parallel edges to some edge of $K_3$ incident with $v$;
and by adding $r$ parallel edges to the (unique) edge of $K_3$ not incident
with $v$. In Figure~\ref{fig:H3}, we depict $K_3(r,j,l)$ 
together with a couple of particular instances. Again, we use
blue and red color to distinguish some edges, and this is useful for the
definition of $\mathcal{F}_3$.

\begin{figure}[ht!]
\begin{center}

\begin{tikzpicture}[scale=0.8]

\begin{scope}[xshift=-5cm, yshift=3.5cm, scale=0.9]
\node [vertex] (0) at (-1,0){};
\node [vertex] (1) at (1,0){};
\node [vertex] (2) at (0,2){};
\node[] at (0,-1){\small{$K_3(1,1,0)$}};

\draw [edge, purple] (0) to (1);

\foreach \from/\to in {0/2, 2/0}
\path[-, red]          (\from)  edge   [bend left]   (\to);
\foreach \from/\to in {1/2, 2/1}
\path[-, blue]          (\from)  edge   [bend left]   (\to);
\end{scope}

\begin{scope}[xshift=0cm, yshift=3.5cm, scale=0.9]
\node [vertex] (0) at (-1,0){};
\node [vertex] (1) at (1,0){};
\node [vertex] (2) at (0,2){};
\node[] at (0,-1){\small{$K_3(r,1,0)$}};

\foreach \from/\to in {1/2, 2/1}
\path[-, blue]          (\from)  edge   [bend left]   (\to);

\path[-, red, dashed]          (0)  edge   [bend left]   (2);
\path[-, red]          (2)  edge   [bend left]   (0);

\draw [edge, blue] (0) to (1);
\end{scope}
\begin{scope}[xshift=5cm, yshift=3.5cm, scale=0.9]
\node [vertex] (0) at (-1,0){};
\node [vertex] (1) at (1,0){};
\node [vertex] (2) at (0,2){};
\node[] at (0,-1){\small{$K_3(r,j,0)$}};

\foreach \from/\to in {1/2, 2/1}
\path[-, blue]          (1)  edge   [bend left]   (2);
\path[-, blue, dashed]          (2)  edge   [bend left]   (1);

\path[-, red, dashed]          (0)  edge   [bend left]   (2);
\path[-, red]          (2)  edge   [bend left]   (0);

\draw [edge] (0) to (1);
\end{scope}

\begin{scope}[xshift=-7.5cm, yshift=0cm, scale=0.9]

\node [vertex] (0) at (-1,0){};
\node [vertex] (1) at (1,0){};
\node [vertex] (2) at (0,2){};
\node[] at (0,-1){\small{$K_3(1,0,1)$}};

\foreach \from/\to in {0/1, 1/2}
\draw [edge, purple] (\from) to (\to);

\foreach \from/\to in {0/2, 2/0}
\path[-, red]          (\from)  edge   [bend left]   (\to);
\path[-]         (1) edge [min distance=1cm, in = 45, out = -45] (1);

\end{scope}

\begin{scope}[xshift=-2.5cm, yshift=0cm, scale=0.9]
\node [vertex] (0) at (-1,0){};
\node [vertex] (1) at (1,0){};
\node [vertex] (2) at (0,2){};
\node[] at (0,-1){\small{$K_3(r,0,1)$}};

\foreach \from/\to in {0/1, 1/2}
\draw [edge, blue] (\from) to (\to);

\path[-, red, dashed]          (0)  edge   [bend left]   (2);
\path[-, red]          (2)  edge   [bend left]   (0);
\path[-]         (1) edge [min distance=1cm, in = 45, out = -45] (1);
\end{scope}

\begin{scope}[xshift=2.5cm, yshift=0cm, scale=0.9]%
\node [vertex] (0) at (-1,0){};
\node [vertex] (1) at (1,0){};
\node [vertex] (2) at (0,2){};
\node[] at (0,-1){\small{$K_3(1,j,1)$}};

\foreach \from/\to in {0/2, 2/0}
\path[-, red]          (\from)  edge   [bend left]   (\to);
\draw [edge, red] (0) to (1);
\foreach \from/\to in {1/2, 2/1}
\path[-, blue]          (1)  edge   [bend left]   (2);
\path[-, blue, dashed]          (2)  edge   [bend left]   (1);

\path[-]         (1) edge [min distance=1cm, in = 45, out = -45] (1);
\end{scope}

\begin{scope}[xshift=7.5cm, yshift=0cm, scale=0.9]
\node [vertex] (0) at (-1,0){};
\node [vertex] (1) at (1,0){};
\node [vertex] (2) at (0,2){};
\node[] at (0,-1){\small{$K_3(r,j,1)$}};

\draw [edge] (0) to (1);

\foreach \from/\to in {1/2, 2/1}
\path[-, blue]          (1)  edge   [bend left]   (2);
\path[-, blue, dashed]          (2)  edge   [bend left]   (1);

\path[-, red, dashed]          (0)  edge   [bend left]   (2);
\path[-, red]          (2)  edge   [bend left]   (0);

\path[-]         (1) edge [min distance=1cm, in = 45, out = -45] (1);
\end{scope}

\begin{scope}[xshift=-5cm, yshift=-3.5cm, scale=0.9]
\node [vertex] (0) at (-1,0){};
\node [vertex] (1) at (1,0){};
\node [vertex] (2) at (0,2){};
\node[] at (0,-1){\small{$K_3(1,0,l)$}};

\foreach \from/\to in {0/2, 2/0}
\path[-, red]          (\from)  edge   [bend left]   (\to);

\foreach \from/\to in {0/1, 1/2}
\draw [edge, red] (\from) to (\to);
\path[-, dashed]         (1) edge [min distance=1cm, in = 45, out = -45] (1);
\end{scope}

%

\begin{scope}[xshift=0cm, yshift=-3.5cm, scale=0.9]
\node [vertex] (0) at (-1,0){};
\node [vertex] (1) at (1,0){};
\node [vertex] (2) at (0,2){};
\node[] at (0,-1){\small{$K_3(1,j,l)$}};

\foreach \from/\to in {0/2, 2/0}
\path[-, red]          (\from)  edge   [bend left]   (\to);
\draw [edge, red] (0) to (1);

\path[-]          (1)  edge   [bend left]   (2);
\path[-, dashed]          (2)  edge   [bend left]   (1);

\path[-, dashed]         (1) edge [min distance=1cm, in = 45, out = -45] (1);
\end{scope}

\begin{scope}[xshift=5cm, yshift=-3.5cm, scale=0.9]
\node [vertex] (0) at (-1,0){};
\node [vertex] (1) at (1,0){};
\node [vertex] (2) at (0,2){};
\node[] at (0,-1){\small{$K_3(r,j,l)$}};

\path[-, dashed, red]          (0)  edge   [bend left]   (2);
\path[-,  red]          (2)  edge   [bend left]   (0);

\draw [edge] (0) to (1);

\path[-]          (1)  edge   [bend left]   (2);
\path[-, dashed]          (2)  edge   [bend left]   (1);

\path[-, dashed]         (1) edge [min distance=1cm, in = 45, out = -45] (1);
\end{scope}

\end{tikzpicture}
\caption{The generating set of graphs for $\mathcal{F}_3$.
Each dashed edge represents $r$ and $j$ parallel edges, 
and the dashed loop represents $l$ loops incident in the same vertex.
The non-bend edges of $K_3(1,1,0)$  and of $K_3(1,0,1)$ are both blue and red.}
\label{fig:H3}
\end{center}
\end{figure}
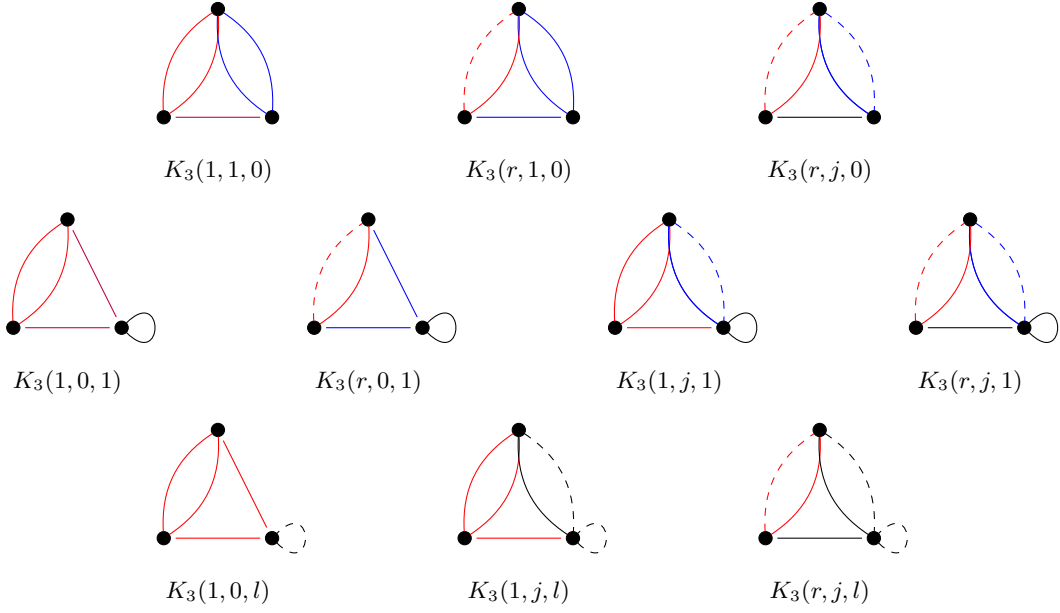


In particular, notice that the bicircular matroids of $K_3(1,1,0)$ and
$K_3(1,0,1)$ are isomorphic to the uniform matroid $U_{3,5}$.  Also
notice that for any pair of non-negative integers, $r$ and $j$, 
the graph $K_3(r,j,0)$ is a subgraph of $2K_3(r,j)$.
We construct $\mathcal{F}_3$ in a similar fashion to how we
constructed $\mathcal{F}_1$. This is done by following the 
edge colorings depicted in Figure~\ref{fig:H3}. A graph $G$ belongs to 
$\mathcal{F}_3$ if either of the followings statements hold:
\begin{enumerate}
	\item $G$ is a subdivision of at most one red edge and at most
	one blue edge of $K_3(r,j,l)$ for some non-negative integers $r$, $j$ and $l$,
	with $r\ge 1$, 
	\item $G$ is a cycle with one chord $e$, and possibly some edges parallel to
	$e$, and at most one loop in each end point of $e$,
	\item $G$ is a chordless cycle with arbitrarily many loops in at most
	one vertex $v$, and at most one loop in one neighbor of $v$, or
	\item $G$ is a connected graph on two vertices.
\end{enumerate}

\begin{lemma}\label{lem:familyF3}
For every graph $G$ in $\mathcal{F}_3$ there is a graph $H\in \mathcal{F}_1\cup
\mathcal{F}_2$ such that $B(G)$ is a minor of $B(H)$.
\end{lemma}
\begin{proof}
We consider the four possible scenarios for a graph in $\mathcal{F}_3$,
and we begin with the last one.  Suppose that $G$ is a connected
graph on two vertices, $x$ and $y$. Let $l_x$ denote the number of loops
on $x$, $l_y$ the number of loops on $y$, and $m$ the number of $xy$-edges.
Consider the graph $H$ obtained from $G_1(l_x,m,l_y)$ by contracting one
blue and one red edge.  By removing one loop from each vertex of $H$ and
two of the non-loop edges, we see that $G$ is a subgraph of $H$, and
thus a minor of $G_1(l_x,m,l_y)\in \mathcal{F}_1$. Below, we depict
these graphs when $l_x = l_y = 1$ and $m = 2$.
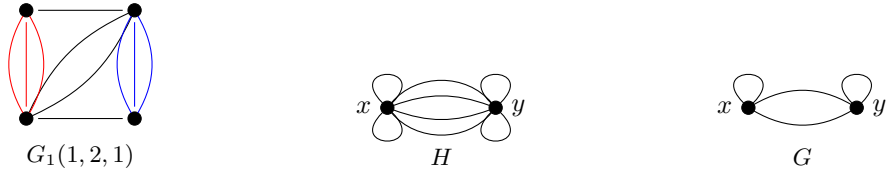
\begin{center}
\begin{tikzpicture}[scale=0.8]
\begin{scope}[xshift=-6cm, yshift=0cm, scale=0.9]
\node [vertex] (0) at (-1,0){};
\node [vertex] (1) at (1,0){};
\node [vertex] (2) at (-1,2){};
\node [vertex] (3) at (1,2){};
\node[] at (0,-0.7){\small{$G_1(1,2,1)$}};

\foreach \from/\to in {3/2, 0/1}
\draw [edge] (\from) to (\to);

\path[-]          (0)  edge   [bend left = 20]   (3);
\path[-]          (0)  edge   [bend right = 20]   (3);

\path[-, red]          (0)  edge   [bend left]   (2);
\path[-, red]          (0)  edge   [bend right]   (2);
\draw [edge, red] (0) to (2);

\path[-, blue]          (1)  edge   [bend right]   (3);
\path[-, blue]          (1)  edge   [bend left]   (3);
\draw [edge, blue] (1) to (3);
\end{scope}
\begin{scope}[xshift=0cm, yshift=0cm, scale=0.9]
\node [vertex, label=180:{$x$}] (0) at (-1,0.2){};
\node [vertex, label = 0:{$y$}] (1) at (1,0.2){};
\node[] at (0,-0.7){\small{$H$}};

\path[-]          (0)  edge   [bend left= 50]   (1);
\path[-]          (0)  edge   [bend right= 50]   (1);
\path[-]          (0)  edge   [bend left = 20]   (1);
\path[-]          (0)  edge   [bend right = 20]   (1);

\path[-]         (1) edge [min distance=1cm] (1);
\path[-]         (1) edge [min distance=1cm, in = 225, out = 315] (1);
\path[-]         (0) edge [min distance=1cm] (0);
\path[-]         (0) edge [min distance=1cm, in = 225, out = 315] (0);
\end{scope}

\begin{scope}[xshift=6cm, yshift=0cm, scale=0.9]
\node [vertex, label=180:{$x$}] (0) at (-1,0.2){};
\node [vertex, label = 0:{$y$}] (1) at (1,0.2){};
\node[] at (0,-0.7){\small{$G$}};

\path[-]          (0)  edge   [bend left]   (1);
\path[-]          (0)  edge   [bend right]   (1);

\path[-]         (1) edge [min distance=1cm] (1);
\path[-]         (0) edge [min distance=1cm] (0);

\end{scope}
\end{tikzpicture}
\end{center}
Now suppose that $G$ satisfies the third statement of the definition
of $\mathcal{F}_3$. Let $m$ be the length of the cycle, $l_v$ denote
the number of loops on $v$, and suppose that a neighbor $u$ of $v$ is
incident with one loop; the case when no neighbor of $v$ has a loop,
follows from the case we are now considering. In this case, consider the
graph $H$ obtained by subdividing the top edge of $G(0,0,l_v-1)$ into
$m-3$ new edges ($m-3 \ge 0$ since $m\ge 3$). Since the top 
edge of $G(0,0,l_v-1)$ is a red edge (see Figure~\ref{fig:H1}),
then $H\in\mathcal{F}_1$. Finally, by contracting one of the red bend edges and
one of the blue edges of $H$ we see that $G$ is a minor of $H$. We illustrate
the case when $l_v = 3$ below, where the dotted edges represent a path of length 
$m-2$.
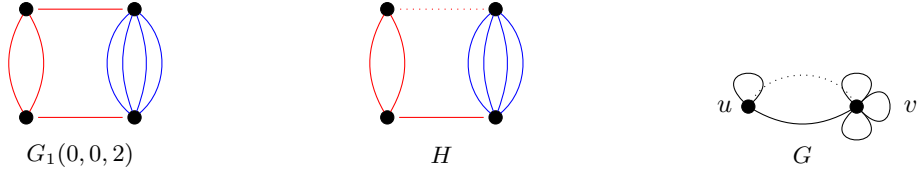
\begin{center}
\begin{tikzpicture}[scale=0.8]
\begin{scope}[xshift=-6cm, yshift=0cm, scale=0.9]
\node [vertex] (0) at (-1,0){};
\node [vertex] (1) at (1,0){};
\node [vertex] (2) at (-1,2){};
\node [vertex] (3) at (1,2){};
\node[] at (0,-0.7){\small{$G_1(0,0,2)$}};

\foreach \from/\to in {3/2, 0/1}
\draw [edge, red] (\from) to (\to);

\path[-, red]          (0)  edge   [bend left]   (2);
\path[-, red]          (0)  edge   [bend right]   (2);

\path[-, blue]          (1)  edge   [bend right=50]   (3);
\path[-, blue]          (1)  edge   [bend left = 50]   (3);
\path[-, blue]          (1)  edge   [bend right = 20]   (3);
\path[-, blue]          (1)  edge   [bend left= 20]   (3);
\end{scope}
\begin{scope}[xshift=0cm, yshift=0cm, scale=0.9]
\node [vertex] (0) at (-1,0){};
\node [vertex] (1) at (1,0){};
\node [vertex] (2) at (-1,2){};
\node [vertex] (3) at (1,2){};
\node[] at (0,-0.7){\small{$H$}};

\draw [edge, red] (0) to (1);
\draw [edge, red, dotted] (3) to (2);

\path[-, red]          (0)  edge   [bend left]   (2);
\path[-, red]          (0)  edge   [bend right]   (2);

\path[-, blue]          (1)  edge   [bend right=50]   (3);
\path[-, blue]          (1)  edge   [bend left = 50]   (3);
\path[-, blue]          (1)  edge   [bend right = 20]   (3);
\path[-, blue]          (1)  edge   [bend left= 20]   (3);
\end{scope}

\begin{scope}[xshift=6cm, yshift=0cm, scale=0.9]
\node [vertex, label=180:{$u$}] (0) at (-1,0.2){};
\node [vertex] (1) at (1,0.2){};
\node[] at (0,-0.7){\small{$G$}};
\node[] at (2,0.2){$v$};

\path[-, dotted]          (0)  edge   [bend left = 60]   (1);
\path[-]          (0)  edge   [bend right]   (1);

\path[-]         (1) edge [min distance=1cm] (1);
\path[-]         (1) edge [min distance=1cm, in = 225, out = 315] (1);
\path[-]         (1) edge [min distance=1cm, in = 45, out = 315] (1);
\path[-]         (0) edge [min distance=1cm] (0);

\end{scope}
\end{tikzpicture}
\end{center}
Our next case is when $G$ is a cycle with two distinguished
vertices $u$ and $v$, such that $u$ and $v$ have at most
one loop each, and any chord of the cycle is a $uv$-edge. 
Analogous to the previous case, we only consider the case
when $u$ and $v$ are incident with a loop each. Let $m$ be
the number of $uv$-edges, $a_1$ the length of one of the
$uv$-arcs of the cycle, and $a_2$ the length of the other.
Let $H$ be the graph obtained form $G_1(0,m,0)$ by subdividing
the top and bottom edges into $a_1-1$ and $a_2-1$ edges, 
respectively. Since the top edge and bottom are red and blue
edges of $G_1(0,m,0)$ (see Figure~\ref{fig:H1}), we conclude
that $H\in \mathcal{F}_1$. In this case, $G$ is the minor of $H$ obtained
by contracting of edge of each parallel class. We illustrate the
case when $m = 2$ as follows, the dotted edges represent 
the corresponding subdivisions.
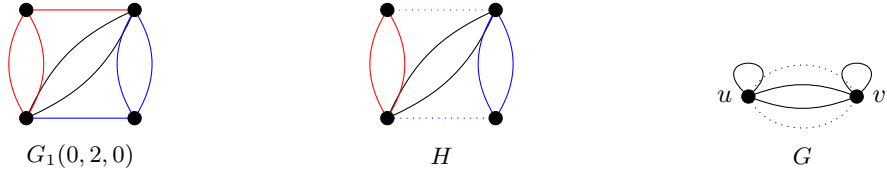
\begin{center}
\begin{tikzpicture}[scale=0.8]
\begin{scope}[xshift=-6cm, yshift=0cm, scale=0.9]
\node [vertex] (0) at (-1,0){};
\node [vertex] (1) at (1,0){};
\node [vertex] (2) at (-1,2){};
\node [vertex] (3) at (1,2){};
\node[] at (0,-0.7){\small{$G_1(0,2,0)$}};

\path[-]          (0)  edge   [bend left = 20]   (3);
\path[-]          (0)  edge   [bend right = 20]   (3);

\path[-, red]          (3)  edge    (2);
\path[-, blue]          (0)  edge     (1);
\path[-, red]          (0)  edge   [bend left]   (2);
\path[-, red]          (0)  edge   [bend right]   (2);
\path[-, blue]          (1)  edge   [bend right]   (3);
\path[-, blue]          (1)  edge   [bend left]   (3);
\end{scope}
\begin{scope}[xshift=0cm, yshift=0cm, scale=0.9]
\node [vertex] (0) at (-1,0){};
\node [vertex] (1) at (1,0){};
\node [vertex] (2) at (-1,2){};
\node [vertex] (3) at (1,2){};
\node[] at (0,-0.7){\small{$H$}};

\path[-]          (0)  edge   [bend left = 20]   (3);
\path[-]          (0)  edge   [bend right = 20]   (3);

\path[-, red, dotted]          (3)  edge    (2);
\path[-, blue, dotted]          (0)  edge     (1);
\path[-, red]          (0)  edge   [bend left]   (2);
\path[-, red]          (0)  edge   [bend right]   (2);
\path[-, blue]          (1)  edge   [bend right]   (3);
\path[-, blue]          (1)  edge   [bend left]   (3);
\end{scope}

\begin{scope}[xshift=6cm, yshift=0cm, scale=0.9]
\node [vertex, label=180:{$u$}] (0) at (-1,0.4){};
\node [vertex, label =0:{$v$}] (1) at (1,0.4){};
\node[] at (0,-0.7){\small{$G$}};

\path[-, dotted]          (0)  edge   [bend left = 60]   (1);
\path[-, dotted]          (0)  edge   [bend right = 60]   (1);
\path[-]          (0)  edge   [bend left = 20]   (1);
\path[-]          (0)  edge   [bend right = 20]   (1);

\path[-]         (1) edge [min distance=1cm] (1);

\path[-]         (0) edge [min distance=1cm] (0);

\end{scope}
\end{tikzpicture}
\end{center}
The final case is when $G$ is a subdivision of at most one red edge
and at most one blue edge of $K_3(r,j,l)$. As we already observed,
$K_3(r,j,0)$ is a subgraph of $2K_3(r,l)$.  Thus, $G$ is a subgraph
of a subdivision $H$ of at most one red and at most
one blue edge of $2K_3(r,j)$. The claim follows since $H\in \mathcal{F}_2$.
Now suppose that $l\ge 1$, and let $H$ be the graph obtained from
$G_1(r+1,j+1,l)$ by contracting one blue edge.  In this case,
$K_3(r,j,l)$ is a subgraph of $H$. We illustrate the case when
$r = j  = 1$ and $l = 2$ below. 
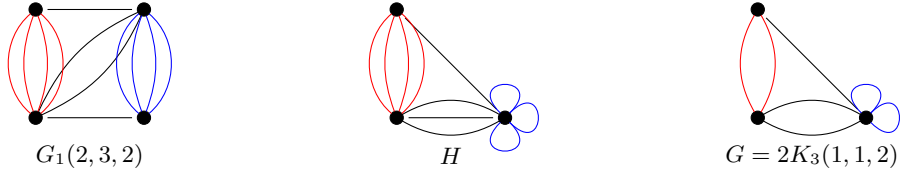
\begin{center}
\begin{tikzpicture}[scale=0.8]
\begin{scope}[xshift=-6cm, yshift=0cm, scale=0.9]
\node [vertex] (0) at (-1,0){};
\node [vertex] (1) at (1,0){};
\node [vertex] (2) at (-1,2){};
\node [vertex] (3) at (1,2){};
\node[] at (0,-0.7){\small{$G_1(2,3,2)$}};

\draw[edge] (2) to (3);
\draw[edge] (0) to (1);

\path[-]          (0)  edge   [bend left = 20]   (3);
\path[-]          (0)  edge   [bend right = 20]   (3);

\path[-, red]          (0)  edge   [bend left = 20]   (2);
\path[-, red]          (0)  edge   [bend right = 20]   (2);
\path[-, red]          (0)  edge   [bend left = 50]   (2);
\path[-, red]          (0)  edge   [bend right = 50]   (2);
\path[-, blue]          (1)  edge   [bend right = 20]   (3);
\path[-, blue]          (1)  edge   [bend left= 20]   (3);
\path[-, blue]          (1)  edge   [bend right = 50]   (3);
\path[-, blue]          (1)  edge   [bend left= 50]   (3);
\end{scope}
\begin{scope}[xshift=0cm, yshift=0cm, scale=0.9]
\node [vertex] (0) at (-1,0){};
\node [vertex] (1) at (1,0){};
\node [vertex] (2) at (-1,2){};
\node[] at (0,-0.7){\small{$H$}};

\draw[edge] (2) to (1);
\draw[edge] (0) to (1);

\path[-]          (0)  edge   [bend left]   (1);
\path[-]          (0)  edge   [bend right]   (1);

\path[-, red]          (0)  edge   [bend left = 20]   (2);
\path[-, red]          (0)  edge   [bend right = 20]   (2);
\path[-, red]          (0)  edge   [bend left = 50]   (2);
\path[-, red]          (0)  edge   [bend right = 50]   (2);
\path[-, blue]         (1) edge [min distance=1cm, in = 315, out = 45] (1);
\path[-, blue]         (1) edge [min distance=1cm, out = 315, in = 225] (1);
\path[-, blue]         (1) edge [min distance=1cm] (1);
\end{scope}

\begin{scope}[xshift=6cm, yshift=0cm, scale=0.9]
\node [vertex] (0) at (-1,0){};
\node [vertex] (1) at (1,0){};
\node [vertex] (2) at (-1,2){};
\node[] at (0,-0.7){\small{$G = 2K_3(1,1,2)$}};

\draw[edge] (2) to (1);

\path[-]          (0)  edge   [bend left]   (1);
\path[-]          (0)  edge   [bend right]   (1);

\path[-, red]          (0)  edge   [bend left]   (2);
\path[-, red]          (0)  edge   [bend right]   (2);
\path[-, blue]         (1) edge [min distance=1cm, in = 315, out = 45] (1);
\path[-, blue]         (1) edge [min distance=1cm] (1);
\end{scope}
\end{tikzpicture}
\end{center}
By preserving the edge colors in the operations previously described, 
we can easily see that the red edges coincide with the colors
described in Figure~\ref{fig:H3}. So, if $G$ is a subdivision of at
most one red edge of $K_3(r+1,j+1,l)$, then $G$ is a minor of a subdivision $H$
of at most one red edge of $G_1(r+1, j+1, l)$. Finally,
notice that the only case when the set of blue edges of $K_3(r,j,l)$ is not
empty, is  when $l = 1$  (see Figure~\ref{fig:H3}). Also notice, that when $l = 1$
then the blue edges and the unique loop in $K_3(r,j,l)$ are clone
edges. Thus, subdividing a blue edge yields the same bicircular matroid
as subdividing the loop. Following the operations previously defined,
we can notice that blue
edges of $G_1(r+1,j+1,l)$ are contracted to loop edges in 
$H$. Thus, for any subdivision $G$ of at most one red edge and at most
one blue edge of $K_3(r,j,l)$, there is a subdivision $H$ of at most
one red edge and at most one blue edge of some $G_1(r+1,j+1,l)$
such that $B(G)$ is a minor of $B(H)$. Since this covers the last
case of the definition of $\mathcal{F}_3$, the claim is proved.
\end{proof}

\begin{proposition}\label{prop:familyF3}
The bicircular matroid of each graph in $\mathcal{F}_3$ is a lattice path matroid.
\end{proposition}
\begin{proof}
By Lemma~\ref{lem:familyF3}, we know that the bicircular matroid
of each graph in $\mathcal{F}_3$ is a minor of the bicircular matroid
of some graph in $\mathcal{F}_1\cup \mathcal{F}_2$. Since the
class of lattice path matroid is a minor closed class \cite{boninEJC27},
the claim follows by Propositions~\ref{prop:familyF1} and \ref{prop:familyF2}.
\end{proof}

\subsection*{\normalfont{\textsc{Family $\mathcal{F}_4$}}}

Our final family consists of certain graphs with cut vertices. 
The restrictions upon the middle blocks is that each middle
block contains exactly two vertices (no restriction upon the edge set). 
To define the constraints over the end blocks of these graphs, we
depict three possible end blocks. Again, we distinguish some red edges, 
and in this case we distinguish one vertex which corresponds
to the cut vertex of the end block. 
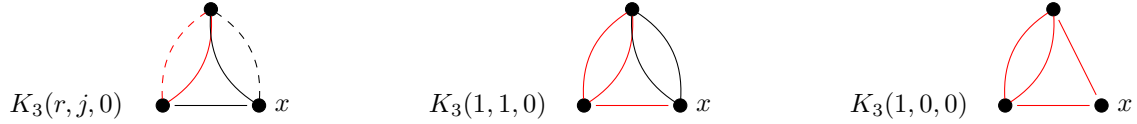
\begin{center}
\begin{tikzpicture}[scale=0.8]

\begin{scope}[xshift=-7cm, yshift=0cm, scale=0.8]
\node [vertex] (0) at (-1,0){};
\node [vertex, label = 0:{$x$}] (1) at (1,0){};
\node [vertex] (2) at (0,2){};
\node[] at (-3,0){$K_3(r,j,0)$};

\draw [edge] (0) to (1);

\path[-, red, dashed]          (0)  edge   [bend left]   (2);
\path[-, red]          (0)  edge   [bend right]   (2);
\path[-, dashed]          (1)  edge   [bend right]   (2);
\path[-]          (1)  edge   [bend left]   (2);

\end{scope}

\begin{scope}[xshift=0cm, yshift=0cm, scale=0.8]
\node [vertex] (0) at (-1,0){};
\node [vertex, label = 0:{$x$}] (1) at (1,0){};
\node [vertex] (2) at (0,2){};
\node[] at (-3,0){$K_3(1,1,0)$};

\draw [edge, red] (0) to (1);

\path[-, red]          (0)  edge   [bend left]   (2);
\path[-, red]          (0)  edge   [bend right]   (2);
\path[-]          (1)  edge   [bend right]   (2);
\path[-]          (1)  edge   [bend left]   (2);

\end{scope}

\begin{scope}[xshift=7cm, yshift=0cm, scale=0.8]
\node [vertex] (0) at (-1,0){};
\node [vertex, label = 0:{$x$}] (1) at (1,0){};
\node [vertex] (2) at (0,2){};
\node[] at (-3,0){$K_3(1,0,0)$};

\foreach \from/\to in {0/1, 1/2}
\draw [edge, red] (\from) to (\to);

\path[-, red]          (0)  edge   [bend left]   (2);
\path[-, red]          (0)  edge   [bend right]   (2);

\end{scope}
\end{tikzpicture}
\end{center}
We first define an auxiliary family $\mathcal{F}_4'$. A graph
$G$ belongs to $\mathcal{F}_4'$ if it is a loopless graph with the following
properties:
\begin{enumerate}
	\item the block graph of $G$ is a path, 
	\item every middle block of $G$ has exactly two vertices, and
	\item each end block of $G$ is a subdivision of at most one red
	edge of either of the graphs above, where $x$ is a cut vertex.
\end{enumerate}

An example of a graph in $\mathcal{F}_4'$ looks as follows, where
one blue and one red edge might be subdivided (dashed edges
represent multiple parallel copies, and dots represent arbitrarily many
blocks of two vertices).

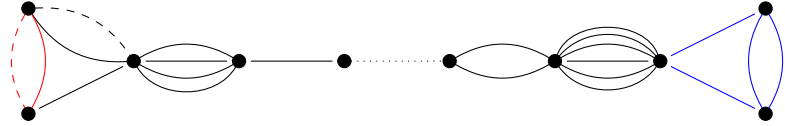
\begin{center}
\begin{tikzpicture}[scale=0.7]

\node [vertex] (0) at (-5,0){};
\node [vertex] (1) at (-3,0){};
\node [vertex] (2) at (-1,0){};
\node [vertex] (3) at (1,0){};
\node [vertex] (4) at (3,0){};
\node [vertex] (5) at (5,0){};

\node [vertex] (a0) at (-7,1){};
\node [vertex] (b0) at (-7,-1){};
\foreach \from/\to in {0/b0}
\draw [edge] (\from) to (\to);


\path[-, dashed]          (a0)  edge   [bend left]   (0);
\path[-]          (a0)  edge   [bend right]   (0);

\path[-, red]          (a0)  edge   [bend left]   (b0);
\path[-, red, dashed]          (a0)  edge   [bend right]   (b0);
\node [vertex] (a6) at (7,1){};
\node [vertex] (b6) at (7,-1){};
\foreach \from/\to in {5/b6, a6/5}
\draw [edge, blue] (\from) to (\to);

\path[-, blue]          (a6)  edge   [bend left]   (b6);
\path[-, blue]          (a6)  edge   [bend right]   (b6);

\foreach \from/\to in {4/5, 1/2, 0/1}
\draw [edge] (\from) to (\to);

\path[-]          (0)  edge   [bend left]   (1);
\path[-]          (0)  edge   [bend right]   (1);
\path[-]          (0)  edge   [bend right=60]   (1);

\path[-]          (2)  edge   [dotted]   (3);
\path[-]          (4)  edge   [bend right]   (3);
\path[-]          (4)  edge   [bend left]   (3);
\path[-]          (4)  edge   [bend right]   (5);
\path[-]          (4)  edge   [bend right=55]   (5);
\path[-]          (4)  edge   [bend left]   (5);
\path[-]          (4)  edge   [bend left=50]   (5);
\path[-]          (4)  edge   [bend left=70]   (5);

\end{tikzpicture}
\end{center}

A graph belongs to $\mathcal{F}_4$ if it can be obtained from
a graph in $\mathcal{F}'_4$ by adding loops to cut vertices and
removing edges.  In particular, any graph whose  block tree is a path,
and its blocks have exactly two vertices, belongs to $\mathcal{F}_4$. 
Also, any graph that satisfies items 1 and 2 above, and its end blocks
are loopless cycles, also belongs to $\mathcal{F}_4$.
Consider a graph $G$ obtained from some $H\in\mathcal{F}_4'$ by adding
loops to one cut vertex. It is not hard to notice that $G$ can also be obtained
from some  $H'\in\mathcal{F}_4'$ by contracting one edge in some middle block.
Thus, we inductively conclude that  every graph in $\mathcal{F}_4$
is a minor of some graph in $\mathcal{F}_4'$. 

\begin{proposition}\label{prop:familyF4}
The bicircular matroid of every graph in $\mathcal{F}_4$
is a lattice path matroid.
\end{proposition}
\begin{proof}
Recall that lattice path matroids are closed under minors~\cite{boninEJC27}.
Thus, it suffices to show that all graphs in $\mathcal{F}_4'$
have  lattice path bicircular matroids.  To begin with, consider the following graphs.
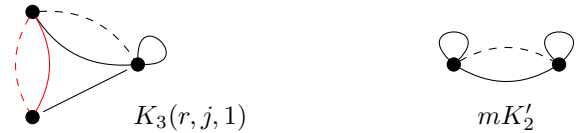
\begin{center}
\begin{tikzpicture}[scale=0.7]

\node [vertex] (0) at (-5,3){};
\node [vertex] (a) at (-5,1){};
\node [vertex] (1) at (-3,2){};
\node[] at (-2,1){$K_3(r,j,1)$};

\path[-, dashed]          (0)  edge   [bend left]   (1);
\path[-]          (0)  edge   [bend right]   (1);

\path[-, red]          (0)  edge   [bend left]   (a);
\path[-, red, dashed]          (0)  edge   [bend right]   (a);
%
\node [vertex] (n) at (3,2){};
\node [vertex] (m) at (5,2){};
\node[] at (4,1){$mK_2'$};

\path[-]          (n)  edge   [bend right]   (m);
\path[-, dashed]          (n)  edge   [bend left]   (m);

\foreach \from/\to in {1/a}
\draw [edge] (\from) to (\to);

\path[-]	     (n) edge [min distance=1cm]  (n);
\path[-]	     (m) edge [min distance=1cm]  (m);
\path[-]	     (1) edge [min distance=1cm, out= 90, in= 0]  (1);

\end{tikzpicture}
\end{center}
On the one hand, the bicircular matroid of a subdivision $G$ of at most one red edge
of $K_3(r,j,1)$ is a lattice path matroid (Proposition~\ref{prop:familyF3}).
Clearly, all black parallel edges together with the loop, are a fundamental flat
of $B(G)$. So, by  Proposition~\ref{prop:boninff}, we know
that there is an interval ordering where these edges (and the loop) are 
a final segment. Moreover, since this is a set of clone edges, we can choose
an interval ordering of $E(G)$ where the loop is the maximum element.
On the other hand,  the bicircular matroid of $mK_2'$ is the uniform matroid
$U_{2,m+2}$, so it is a lattice path matroid. 
Since any pair of elements in a uniform matroid are clones, 
then there is an interval ordering of the edge set of $mK_2'$ where
the loops are the minimum and maximum elements. 
Recall that the loop sum of a pair of graphs is obtained
by gluing two graphs over a loop, and then deleting this loop. Now,
notice that  any graph in $\mathcal{F}_4'$ is a loop sum of the form
\[
EB_1\oplus_l MB_1\oplus_l \cdots \oplus_l MB_k \oplus_l EB_2,
\]
where $EB_1$ and $EB_2$ are a subdivision of at most
one red edge of either $K_3(r,j,1)$, $K_3(1,1,1)$ or $K_3(1,0,1)$, 
and the graphs $MB_i$ are a graph on two vertices with exactly one loop
in each vertex. Thus, by the arguments above, and by 
Lemma~\ref{lem:loopsums} the bicircular matroid of the loop sum above
is a lattice path matroid. Therefore, the bicircular matroid
of each graph in $\mathcal{F}_4'$ is a lattice path matroid. 
So, the claim follows because every graph in $\mathcal{F}_4$ is a
minor of some graph in $\mathcal{F}_4'$.
\end{proof}

\subsection*{\normalfont{\textsc{Efficient recognition}}}

To conclude this section, we observe that it takes linear time to recognize
the graph families $\mathcal{F}_i$ with $i\in\{0,1,2,3,4\}$.
Consider a connected graph $G$ and let $H$
be the graph obtained from $G$ after contracting all subdivided edges.
By definition of  each $\mathcal{F}_i$, if $G$  belongs to some
$\mathcal{F}_i$ with $i\in\{0,1,2,3,4\}$, then $H$ must be either of the following
graphs:
\begin{enumerate}
		\item $K_{2,3},~K'_{2,3},~K''_{2,3},~K^\ast_{2,3}$ or $K_4$,
		\item $G(r,d,b)$ for some non-negative integers $r$ and $b$,
		\item $2K_3(r,b)$ for some non-negative integers $r$ and $b$,
		\item $K_3(r,j,l)$ for some non-negative
		integers $r$, $j$ and $l$, or a graph on two vertices, or
		\item a graph whose block tree is a path, every middle block has
		exactly two vertices, and each block has two vertices or it is
		some  graph $K_3(r,j,l)$ for some non-negative integers $K_3(r,j,l)$.
\end{enumerate}
Each of these classes can be recognized in linear time
with respect to the size of the edge set of the input graph. Moreover, 
after contracting subdivided edges of graph $G$, and  keeping track of which
edges correspond to subdivision classes, we can determine if $G$
belongs to $\bigcup_{i=0}^4\mathcal{F}_i$
in linear time with respect to $|E(G)|$. By these arguments we observe
the following.

\begin{observation}\label{obs:recognition}
Given an input graph $G$, there is a linear time algorithm (in $|E(G)|$)
to decide if $G$ belongs to the union
\[
\mathcal{F}_0\cup \mathcal{F}_1\cup \mathcal{F}_2\cup
\mathcal{F}_3\cup \mathcal{F}_4.
\]
\end{observation}


\section{Excluded minors and graph families}
\label{sec:EMGF}

We have arrived to the most technical section of this work. Here,
we show if a connected bicircular matroid $B(G)$ is
$ex_B(\mathcal{L})$-minor free, then $G$ belongs to a family
$\mathcal{F}_i$ for some $i\in\{0,1,2,3,4\}$.
The set $ex_B(\mathcal{L})$ consists of the matroids introduced in
Section~\ref{sec:EM}; namely,  the matroids
\[
\mathcal{W}^3,~C^{2,4},~R^3,~R^4,~D^4,~A^3,~B^1, \text{ and } S^1.
\]

\begin{proposition}\label{prop:F0-F4}
Let $G$ be a graph. If $B(G)$ is an $ex_B(\mathcal{L})$-minor free
connected matroid, then
\[
G\in \mathcal{F}_0\cup\mathcal{F}_1\cup
\mathcal{F}_2\cup\mathcal{F}_3\cup \mathcal{F}_4.
\]
\end{proposition}
\begin{proof}
This proof is divided in three main cases:
Non-outerplanar graphs (Proposition~\ref{prop:nonouterplanar}),
$2$-connected graphs (Proposition~\ref{prop:2connected}), and
outerplanar graphs with a cut vertex (Proposition~\ref{prop:excludedcase}).
The claim follows from these three propositions.
\end{proof}

The proofs of  Propositions~\ref{prop:nonouterplanar},
\ref{prop:2connected}, and~\ref{prop:excludedcase} are subdivided 
in several cases. Regardless of the particular subcase in turn, the proof
idea is very simple:
Given a graph $G$ we verify that if $G$ does not belong to some family
$\mathcal{F}_i$, then  there is a graph $G'$ with a minor $H$
such that $B(G)\cong B(G')$ and $B(H)\in ex_B(\mathcal{L})$ (in most
cases $G' = G$).  In Figures~\ref{fig:C24}, \ref{fig:bicircularEX}
and~\ref{fig:BB}, we depict bicircular and affine presentations of 
matroids in $ex_B(\mathcal{L})$.

\subsection*{\normalfont{\textsc{Non-outerplanar graphs}}}

A well-known characterization of outerplanar graphs states that a graph
is an outerplanar graph if it does not contain a subdivision of
$K_4$ nor a subdivision of $K_{2,3}$ \cite{chartrandIHP}. We first consider
the case when
$G$ contains a subdivision of $K_4$, and later we consider the general case
of non-outerplanar graphs.

Recall that the bicircular matroid of $K_4$ is $U_{4,6}$, and that the
operation of  subdividing an edge $e$ in a graph $G$, translates to series
extension of $e$ in $B(G)$. Lemma~\ref{lem:uniformextensions},
asserts that if $M$ is a series extension of at least  three different
elements of a uniform matroid (of rank at least two and positive corank), 
then $M$ contains a $C^{2,4}$-minor. So,
we conclude that if $G$ contains a subdivision of at least three
different edges of $K_4$, then $B(G)$ contains
a $C^{2,4}$-minor.

\begin{lemma}\label{lem:K4subdivision}
Let $G$ be a graph such that $B(G)$ is an $ex_B(\mathcal{L})$-minor free
connected matroid. If $G$ contains a subdivision of $K_4$, then
$G$ equals a subdivision of at most two edges of $K_4$.
In particular, this implies that $G\in\mathcal{F}_0$.
\end{lemma}
\begin{proof}
By the arguments preceding this lemma, we know that any subdivision $G$
of at least three edges of $K_4$, satisfies that $B(G)$ has a $C^{2,4}$-minor. 
Thus, if $B(G)$ is  an $ex_B(\mathcal{L})$-minor free matroid, and $G$
is a subdivision of $K_4$, then $G$ equals a subdivision of at most two edges
of $K_4$. We proceed to show that if $G$ contains a subdivision of $K_4$, 
the $G$ is a subdivision of $K_4$. First, we observe that the bicircular matroids
of the following graphs contain some minor in $ex_B(\mathcal{L})$.
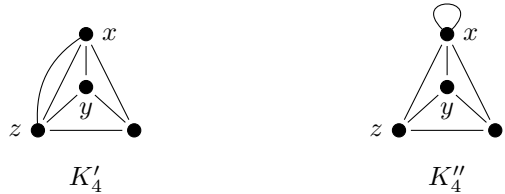
\begin{center}

\begin{tikzpicture}[scale=0.8]

\begin{scope}[xshift=-3cm, yshift = 0cm, scale=0.8]
\node [vertex, label = 180:{$z$}] (0) at (-1,0){};
\node [vertex] (1) at (1,0){};
\node [vertex, label = 270:{$y$}] (2) at (0,0.9){};
\node [vertex, label = 0:{$x$}] (3) at (0,2){};
\node[] at (0,-1){$K_4'$};

\foreach \from/\to in {0/1,0/2, 0/3, 1/2, 1/3, 2/3}
\draw [edge] (\from) to (\to);
\path[-]          (0)  edge   [bend left]   (3);

\end{scope}
\begin{scope}[xshift=3cm, yshift = 0cm, scale=0.8]
\node [vertex, label = 180:{$z$}] (0) at (-1,0){};
\node [vertex] (1) at (1,0){};
\node [vertex, label = 270:{$y$}] (2) at (0,0.9){};
\node [vertex, label = 0:{$x$}] (3) at (0,2){};
\node[] at (0,-1){$K_4''$};

\foreach \from/\to in {0/1,0/2, 0/3, 1/2, 1/3, 2/3}
\draw [edge] (\from) to (\to);
\path[-]         (3) edge [min distance=1cm] (3);

\end{scope}
\end{tikzpicture}
\end{center}
Consider the minor of $K_4''$ obtained by contracting the edge $xy$.
This minor is isomorphic to $A_3$ (Figure~\ref{fig:bicircularEX}). Similarly,
by contracting one of the $xz$-edges
of $K_4'$ we also obtain $A_3$ as a minor of $K_4'$. Thus, $B(K_4')$ and
$B(K_4'')$ contain $A^3$ as a minor. 
To conclude the proof, suppose that $G$ contains a subdivision
$H$ of $K_4$, and $G\neq H$. So, there is an edge in $E(G)\setminus E(H)$,
and by contracting $H$ to $K_4$, we find either $K_4'$ or $K_4''$ as a minor of $G$.
This contradicts the fact that $B(G)$ is $ex_B(\mathcal{L})$-minor free, 
and so the claim follows.
\end{proof}

To conclude this first scenario, we consider the general case of
non-outerplanar graphs. We proceed by doing some case checking:
We list all possible minimal supergraphs of $K_{2,3}$ 
that do not belong to $\mathcal{F}_0$, and we see that the bicircular matroids
of these graphs have an $ex_B(\mathcal{L})$-minor.
This (almost) exhaustive case checking is the same technique used
in the rest of this section. 

\begin{proposition}\label{prop:nonouterplanar}
Let $G$ be a non-outerplanar graph. If $B(G)$ is an $ex_B(\mathcal{L})$-minor
free connected matroid, then $G$ belongs to $\mathcal{F}_0$.
\end{proposition}
\begin{proof}
Throughout this proof we assume that $G$ and $B(G)$ are as in the
hypothesis. Also, if $G$ contains no subdivision 
of $K_{2,3}$ then it contains a subdivision of $K_4$, and the
claim follows by Lemma~\ref{lem:K4subdivision}. So, we also assume
that $G$ contains a subdivision of $K_{2,3}$. Since any subdivision of
$K_{2,3}$ belongs to $\mathcal{F}_0$, we suppose that 
$G$ contains a subdivision $H$ of $K_{2,3}$ plus some edge not in $E(H)$. 
In other words, $G$ contains a subdivision of some supergraph of $K_{2,3}$.
It is not hard to notice that the minimal supergraphs of $K_{2,3}$ are $K_4$,
$K'_{2,3}$, $K''_{2,3}$, and the graphs $T_1$ and $T_2$ depicted below.
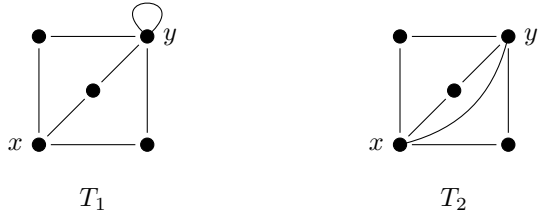
\begin{center}

\begin{tikzpicture}[scale=0.8]

\begin{scope}[xshift=-3cm, yshift=0cm, scale=0.9]
\node [vertex, label = 180:{$x$}] (0) at (-1,0){};
\node [vertex] (1) at (1,0){};
\node [vertex] (2) at (-1,2){};
\node [vertex, label = 0:{$y$}] (3) at (1,2){};
\node [vertex] (s) at (0,1){};
\node[] at (0,-1){$T_1$};

\foreach \from/\to in {0/1, 3/2, 0/s, s/3, 0/2, 1/3}
\draw [edge] (\from) to (\to);

\path[-]          (3)  edge   [min distance = 1cm]   (3);
\end{scope}

\begin{scope}[xshift=3cm, yshift = 0cm, scale=0.9]
\node [vertex, label = 180:{$x$}] (0) at (-1,0){};
\node [vertex] (1) at (1,0){};
\node [vertex] (2) at (-1,2){};
\node [vertex, label = 0:{$y$}] (3) at (1,2){};
\node [vertex] (s) at (0,1){};
\node[] at (0,-1){$T_2$};

\foreach \from/\to in {0/1, 3/2, 0/s, s/3, 0/2, 1/3}
\draw [edge] (\from) to (\to);

\path[-]          (0)  edge   [bend right]   (3);

\end{scope}

\end{tikzpicture}
\end{center}
In particular,  if $G$ contains a subdivision of $K_4$, then $G\in \mathcal{F}_0$
(Lemma~\ref{lem:K4subdivision}). Also, observe that by contracting the edge $xy$
of $T_2$, we recover $P_2$ as a minor of $T_2$. Thus $B(T_2)$ contains
a $C^{2,4}$-minor. Moreover, it is not hard to notice that  $B(T_1)\cong B(T_2)$ so,
$B(T_1)$ contains a $C^{2,4}$-minor as well. So, if $G$ is 
$ex_B(\mathcal{L})$-free then $G$ contains no subdivision of $T_1$ nor
of $T_2$. 

We proceed to consider the case when $G$ contains a subdivision of
$K'_{2,3}$. Again,  we show that if $G$ is obtained from $K'_{2,3}$ be adding
one edge, then $B(G)$ contains an $ex_B(\mathcal{L})$-minor. Recall that
$T_1$ has a bicircular matroid with an $ex_B(\mathcal{L})$-minor. Thus, 
if $G'$ is obtained by adding  a loop to $x$ (or $y$) then $B(G')$ has an
$ex_B(\mathcal{L})$-minor. By considering the following cases we show
that $G'$ plus any loop yields a bicircular matroid with an $ex_B(\mathcal{L})$-minor.
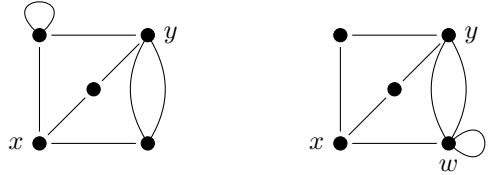
\begin{center}

\begin{tikzpicture}[scale=0.8]

\begin{scope}[xshift=-2.5cm, yshift=0cm, scale=0.9]
\node [vertex, label = 180:{$x$}] (0) at (-1,0){};
\node [vertex] (1) at (1,0){};
\node [vertex] (2) at (-1,2){};
\node [vertex, label = 0:{$y$}] (3) at (1,2){};
\node [vertex] (s) at (0,1){};

\foreach \from/\to in {0/1, 3/2, 0/s, s/3, 0/2}
\draw [edge] (\from) to (\to);

\path[-]          (2)  edge   [min distance = 1cm]   (2);
\path[-]          (1)  edge   [bend right]   (3);
\path[-]          (1)  edge   [bend left]   (3);

\end{scope}

\begin{scope}[xshift=2.5cm, yshift = 0cm, scale=0.9]
\node [vertex, label = 180:{$x$}] (0) at (-1,0){};
\node [vertex, label = 270:{$w$}] (1) at (1,0){};
\node [vertex] (2) at (-1,2){};
\node [vertex, label = 0:{$y$}] (3) at (1,2){};
\node [vertex] (s) at (0,1){};

\foreach \from/\to in {0/1, 3/2, 0/s, s/3, 0/2}
\draw [edge] (\from) to (\to);

\path[-]          (1)  edge   [min distance = 1cm, out = 40, in = -40]   (1);
\path[-]          (1)  edge   [bend right]   (3);
\path[-]          (1)  edge   [bend left]   (3);
\end{scope}
\end{tikzpicture}
\end{center}
The graph to the left is $S_1'$, so $B(S_1') = S^1 \in ex_B(\mathcal{L})$; and
if we contract one of the $yw$-edges in the graph to the right, we recover
$R_4''$. If $G$ is a spanning supergraph of $K'_{2,3}$ and does not contain
any subgraph of the above, nor $K_4$, $T_1$ and $T_2$, then it contains one
of the following subgraphs.
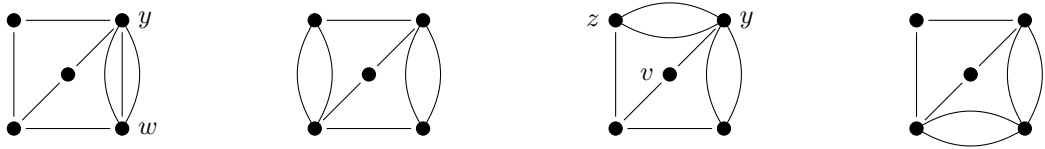
\begin{center}

\begin{tikzpicture}[scale=0.8]

\begin{scope}[xshift=-7.5cm, yshift=0cm, scale=0.9]
\node [vertex] (0) at (-1,0){};
\node [vertex, label = 0:{$w$}] (1) at (1,0){};
\node [vertex] (2) at (-1,2){};
\node [vertex, label = 0:{$y$}] (3) at (1,2){};
\node [vertex] (s) at (0,1){};

\foreach \from/\to in {1/0, 0/s, s/3, 3/2, 1/3, 2/0}
\draw [edge] (\from) to (\to);

\path[-]          (1)  edge   [bend right]   (3);
\path[-]          (1)  edge   [bend left]   (3);
\end{scope}

\begin{scope}[xshift=-2.5cm, yshift=0cm, scale=0.9]
\node [vertex] (0) at (-1,0){};
\node [vertex] (1) at (1,0){};
\node [vertex] (2) at (-1,2){};
\node [vertex] (3) at (1,2){};
\node [vertex] (s) at (0,1){};

\foreach \from/\to in {1/0, 0/s, s/3, 3/2}
\draw [edge] (\from) to (\to);

\path[-]          (1)  edge   [bend right]   (3);
\path[-]          (1)  edge   [bend left]   (3);
\path[-]          (2)  edge   [bend right]   (0);
\path[-]          (2)  edge   [bend left]   (0);

\end{scope}

\begin{scope}[xshift=2.5cm, yshift = 0cm, scale=0.9]
\node [vertex] (0) at (-1,0){};
\node [vertex] (1) at (1,0){};
\node [vertex, label = 180:{$z$}] (2) at (-1,2){};
\node [vertex, label = 0:{$y$}] (3) at (1,2){};
\node [vertex, label=180:{$v$}] (s) at (0,1){};

\foreach \from/\to in {1/0, 0/s, s/3, 0/2}
\draw [edge] (\from) to (\to);

\path[-]          (1)  edge   [bend right]   (3);
\path[-]          (1)  edge   [bend left]   (3);
\path[-]          (2)  edge   [bend right]   (3);
\path[-]          (2)  edge   [bend left]   (3);
\end{scope}

\begin{scope}[xshift=7.5cm, yshift = 0cm, scale=0.9]
\node [vertex] (0) at (-1,0){};
\node [vertex] (1) at (1,0){};
\node [vertex] (2) at (-1,2){};
\node [vertex] (3) at (1,2){};
\node [vertex] (s) at (0,1){};

\foreach \from/\to in { 3/2, 0/s, s/3, 0/2}
\draw [edge] (\from) to (\to);

\path[-]          (1)  edge   [bend right]   (3);
\path[-]          (1)  edge   [bend left]   (3);
\path[-]          (1)  edge   [bend right]   (0);
\path[-]          (1)  edge   [bend left]   (0);

\end{scope}

\end{tikzpicture}
\end{center}
We see that $R''_4$ is a minor of the left most graph above,  by contracting
one of $wy$-edges.  The middle left
graph above is $S_1$ so, its bicircular matroid belongs to $ex_B(\mathcal{L})$.
We can see that $A_3$ is a minor of the middle right graph by contracting 
one of the $yz$-edges and the edge $yv$. Finally, 
notice that the rightmost graph is $K^\ast_{2,3}$. 
We already argued that if we add a loop to $K'_{2,3}$, then its bicircular
matroid has an $ex_B(\mathcal{L})$-minor. Also, $T_2$ together with the
previous four graphs
show that if we add a non-loop edge to $K'_{2,3}$ we obtain either
a copy of $K^\ast_{2,3}$, or a graph whose bicircular matroid
has an  $ex_B(\mathcal{L})$-minor. This implies that  if $G$ contains a
subdivision $H$ of $K'_{2,3}$, but no subdivision of $K^\ast_{2,3}$, then
$H = G$ (otherwise, there is some edge in $G$ that is not in $H$, and
with adequate contractions, we obtain one of the graphs above, which
contradicts the fact that $B(G)$ is $ex_B(\mathcal{L})$-minor free).

Recall that we distinguished some blue and red edges in $K'_{2,3}$.
In order to conclude the proof, we must show that if $H$ is a subdivision
of some black edges of $K'_{2,3}$, then $B(H)$ contains
an $ex_B(\mathcal{L})$-minor. In other words, the bicircular matroids
of either of the following graph contain an $ex_B(\mathcal{L})$-minor.
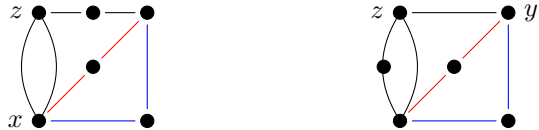
\begin{center}
\begin{tikzpicture}[scale=0.8]

\begin{scope}[xshift=-3cm, yshift=-3.5cm, scale=0.9]
\node [vertex, label = 180:{$x$}] (0) at (-1,0){};
\node [vertex] (1) at (1,0){};
\node [vertex, label = 180:{$z$}] (2) at (-1,2){};
\node [vertex] (3) at (1,2){};
\node [vertex] (s) at (0,1){};
\node [vertex] (s1) at (0,2){};

\foreach \from/\to in {2/s1, s1/3}
\draw [edge] (\from) to (\to);

\foreach \from/\to in  {0/s, s/3}
\draw [edge, red] (\from) to (\to);

\foreach \from/\to in {0/1, 1/3}
\draw [edge, blue] (\from) to (\to);

\path[-]          (0)  edge   [bend left]   (2);
\path[-]          (0)  edge   [bend right]   (2);
\end{scope}

\begin{scope}[xshift=3cm, yshift=-3.5cm, scale=0.9]
\node [vertex] (0) at (-1,0){};
\node [vertex] (1) at (1,0){};
\node [vertex, label = 180:{$z$}] (2) at (-1,2){};
\node [vertex, label = 0:{$y$}] (3) at (1,2){};
\node [vertex] (s) at (0,1){};
\node [vertex] (s1) at (-1.3,1){};

\draw [edge] (2) to (3);

\foreach \from/\to in  {0/s, s/3}
\draw [edge, red] (\from) to (\to);

\foreach \from/\to in {0/1, 1/3}
\draw [edge, blue] (\from) to (\to);

\path[-]          (0)  edge   [bend left]   (2);
\path[-]          (0)  edge   [bend right]   (2);
\end{scope}

\end{tikzpicture}
\end{center}
Notice that if we contract one of the $xz$-edges in the graph to the left, 
we recover $T_1$ as a minor; on the graph to the right we find $T_2$
as a minor by contracting the $yz$-edge. We already argued that neither
$B(T_1)$ nor $B(T_2)$ are $ex_B(\mathcal{L})$-minor free. Thus, 
if $B(G)$ is an $ex_B(\mathcal{L})$-minor free connected matroid, 
and $G$ is a graph that contains subdivision of $K'_{2,3}$
but not of $K^\ast_{2,3}$, then $G\in\mathcal{F}_0$.

Suppose that $G$ contains a subdivision of $K^\ast_{2,3}$. Since
$K^\ast_{2,3}$ is a supergraph of $K'_{2,3}$, then every subdivision
of $K^\ast_{2,3}$ contains a subdivision of $K'_{2,3}$. So, by the arguments
above, if $G$ is a proper supergraph of $K^\ast_{2,3}$, then $B(G)$
has an $ex_B(\mathcal{L})$-minor. Moreover, if $H$ is a subdivision of
$K^\ast_{2,3}$, then $H$ is a subdivision of the blue and red edges
of $K^\ast_{2,3}$. Therefore, the statement of this proposition is settled when
$G$ contains a subdivision of $K_4$, of $K'_{2,3}$ or
a subdivision of $K^\ast_{2,3}$. When
$G$ contains a subdivision of  $K''_{2,3}$ the proof follows analogous
arguments to the ones above. Finally, if $G$ contains no
subdivision of $K_4$, of $K'_{2,3}$, of $K''_{2,3}$ nor of $K^\ast_{2,3}$,
then $G$ contains a subdivision of $T_1$, of $T_2$ or $G$ equals
a subdivision of $K_{2,3}$. Since $B(G)$ is $ex_B(\mathcal{L})$-minor free,
we conclude that $G$ is a subdivision of $K_{2,3}$, so $G\in \mathcal{F}_0$.
All possible cases for non-outerplanar graphs have now been considered.
\end{proof}

\subsection*{\normalfont{\textsc{Subdivisions of $G_1$}}}

We move on to the second case. In this case we consider
$2$-connected graphs. We begin by the subcase when $G$
contains a subdivision of $G_1$. Even though this seems like an arbitrary
assumption, the following lemma shows that this case settles
a broad scenario.

\begin{lemma}\label{lem:2connected}
Let $G$ be a $2$-connected graph of minimum degree three. 
If $G$ contains a pair of non-adjacent vertices and $B(G)$
is $ex_B(\mathcal{L})$-minor free, then $G$ contains a subdivision
of $G_1$. 
\end{lemma}
\begin{proof}
We may assume that $G$ has exactly four vertices.
Indeed, suppose that $G$ is a $2$-connected graph,
and let $x$ and $y$ be a pair of non-adjacent vertices. Since $G$ is
$2$-connected, then there are a pair of disjoint $xy$-paths $P$ and $Q$
of length at least $2$. Now, consider the subgraph of $G$ induced by the
vertices of $P$ and $Q$, and then  contract $P$ and $Q$ to paths
of length exactly $2$.  Thus, with out loss of generality, we assume that
$G$ is a graph on four vertices, and that  $x$ and $y$ are a pair of non-adjacent
vertices of $G$. The following illustrates all  edge
minimal $2$-connected graphs of minimum degree three, where $xy\not\in E$
(up to isomorphism).

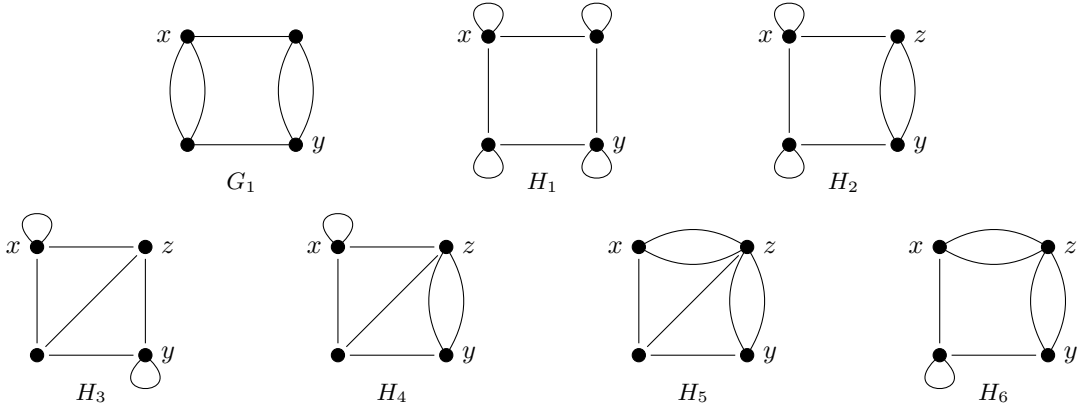
\begin{center}

\begin{tikzpicture}[scale=0.8]

\begin{scope}[xshift=-5cm, yshift=0.5cm, scale=0.9]
\node [vertex] (0) at (-1,0){};
\node [vertex, label=right:{$y$}] (1) at (1,0){};
\node [vertex, label=left:{$x$}]  (2) at (-1,2){};
\node [vertex] (3) at (1,2){};
\node[] at (0,-0.7){\small{$G_1$}};

\path[-]       (3)  edge    (2);
\path[-]         (0)  edge     (1);
\path[-]         (0)  edge   [bend left]   (2);
\path[-]          (0)  edge   [bend right]   (2);
\path[-]          (1)  edge   [bend right]   (3);
\path[-]         (1)  edge   [bend left]   (3);
\end{scope}

\begin{scope}[xshift=0cm, yshift=0.5cm, scale=0.9]
\node [vertex] (0) at (-1,0){};
\node [vertex, label=right:{$y$}] (1) at (1,0){};
\node [vertex, label=left:{$x$}]  (2) at (-1,2){};
\node [vertex] (3) at (1,2){};
\node[] at (0,-0.7){\small{$H_1$}};

\foreach \from/\to in {0/1, 1/3, 2/3, 2/0}
\draw [edge] (\from) to (\to);

\path[-]         (0) edge [min distance=1cm, out=315, in=225] (0);
\path[-]         (1) edge [min distance=1cm, out=315, in=225]  (1);
\path[-]         (2) edge [min distance=1cm] (2);
\path[-]         (3) edge [min distance=1cm] (3);
\end{scope}

\begin{scope}[xshift=5cm, yshift=0.5cm, scale=0.9]
\node [vertex] (0) at (-1,0){};
\node [vertex, label=right:{$y$}] (1) at (1,0){};
\node [vertex, label=left:{$x$}]  (2) at (-1,2){};
\node [vertex, label=right:{$z$}] (3) at (1,2){};
\node[] at (0,-0.7){\small{$H_2$}};

\foreach \from/\to in {0/1, 2/3, 2/0}
\draw [edge] (\from) to (\to);

\path[-]         (0) edge [min distance=1cm, out=315, in=225] (0);
\path[-]         (2) edge [min distance=1cm] (2);

\path[-]          (1)  edge   [bend right]   (3);
\path[-]         (1)  edge   [bend left]   (3);
\end{scope}

\begin{scope}[xshift=-7.5cm, yshift=-3cm, scale=0.9]
\node [vertex] (0) at (-1,0){};
\node [vertex, label=right:{$y$}] (1) at (1,0){};
\node [vertex, label=left:{$x$}]  (2) at (-1,2){};
\node [vertex, label=right:{$z$}] (3) at (1,2){};
\node[] at (0,-0.7){\small{$H_3$}};

\foreach \from/\to in {0/1, 1/3, 2/3, 2/0, 3/0}
\draw [edge] (\from) to (\to);

\path[-]         (1) edge [min distance=1cm, out=315, in=225]  (1);
\path[-]         (2) edge [min distance=1cm] (2);

\end{scope}

\begin{scope}[xshift=-2.5cm, yshift=-3cm, scale=0.9]
\node [vertex] (0) at (-1,0){};
\node [vertex, label=right:{$y$}] (1) at (1,0){};
\node [vertex, label=left:{$x$}]  (2) at (-1,2){};
\node [vertex, label=right:{$z$}] (3) at (1,2){};
\node[] at (0,-0.7){\small{$H_4$}};

\foreach \from/\to in {0/1, 2/3, 2/0, 3/0}
\draw [edge] (\from) to (\to);

\path[-]         (2) edge [min distance=1cm] (2);
\path[-]          (1)  edge   [bend right]   (3);
\path[-]         (1)  edge   [bend left]   (3);

\end{scope}

\begin{scope}[xshift=2.5cm, yshift=-3cm, scale=0.9]
\node [vertex] (0) at (-1,0){};
\node [vertex, label=right:{$y$}] (1) at (1,0){};
\node [vertex, label=left:{$x$}]  (2) at (-1,2){};
\node [vertex, label=right:{$z$}] (3) at (1,2){};
\node[] at (0,-0.7){\small{$H_5$}};

\foreach \from/\to in {0/1, 2/0, 3/0}
\draw [edge] (\from) to (\to);

\path[-]          (2)  edge   [bend right]   (3);
\path[-]         (2)  edge   [bend left]   (3);
\path[-]          (1)  edge   [bend right]   (3);
\path[-]         (1)  edge   [bend left]   (3);

\end{scope}

\begin{scope}[xshift=7.5cm, yshift=-3cm, scale=0.9]
\node [vertex] (0) at (-1,0){};
\node [vertex, label=right:{$y$}] (1) at (1,0){};
\node [vertex, label=left:{$x$}]  (2) at (-1,2){};
\node [vertex, label=right:{$z$}] (3) at (1,2){};
\node[] at (0,-0.7){\small{$H_6$}};

\foreach \from/\to in {0/1, 2/0}
\draw [edge] (\from) to (\to);

\path[-]         (0) edge [min distance=1cm, out=315, in=225] (0);
\path[-]          (2)  edge   [bend right]   (3);
\path[-]         (2)  edge   [bend left]   (3);
\path[-]          (1)  edge   [bend right]   (3);
\path[-]         (1)  edge   [bend left]   (3);

\end{scope}
\end{tikzpicture}
\end{center}

To conclude the proof, we show that for each $i\in\{1,\dots, 6\}$, 
the bicircular matroid $B(H_i)$ has an $ex_B(\mathcal{L})$-minor.
For instance, by removing any loop of $H_1$, we obtain a subdivision of
$W_3$ (Figure~\ref{fig:bicircularEX}), thus
$B(H_1)$ contains a $\mathcal{W}^3$-minor. Similarly, by contracting
one of the $yz$-edges of $H_2$, we recover $W_3$ as a minor of $H_2$. 
Actually, by contracting one $yz$-edge in either of $H_3$, $H_4$ or $H_6$,
 we find $A_3'$ as a minor of these graphs. 
 Finally, by contracting one $yz$-edge of $H_5$ we recover $A_3$.
 Therefore, if $B(G)$ is $ex_B(\mathcal{L})$-minor
free, then $G$ contains a subdivision of $G_1$.
\end{proof}

At this point it is convenient to have a clear image of the structure of the graphs
$G_1(r,d,b)$ (see, for instance, Figure~\ref{fig:H1}). With the same
procedure as in previous occasions, 
we begin by studying which supergraphs of $G_1$ have a bicircular
matroid with an $ex_B(\mathcal{L})$-minor. In this case, 
these graphs are the following ones.

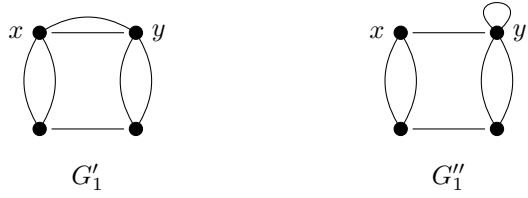
\begin{center}
\begin{tikzpicture}[scale=0.8]

\begin{scope}[xshift=-3cm, yshift=0cm, scale=0.8]
\node [vertex] (0) at (-1,0){};
\node [vertex] (1) at (1,0){};
\node [vertex, label = left:{$x$}] (2) at (-1,2){};
\node [vertex, label = right:{$y$}] (3) at (1,2){};
\node[] at (0,-1){$G_1'$};

\foreach \from/\to in {0/1, 2/3}
\draw [edge] (\from) to (\to);

\path[-]          (0)  edge   [bend left]   (2);
\path[-]          (0)  edge   [bend right]   (2);
\path[-]          (1)  edge   [bend right]   (3);
\path[-]          (1)  edge   [bend left]   (3);
\path[-]          (2)  edge   [bend left]   (3);
\end{scope}

\begin{scope}[xshift=3cm, yshift=0cm, scale=0.8]
\node [vertex] (0) at (-1,0){};
\node [vertex] (1) at (1,0){};
\node [vertex, label = left:{$x$}] (2) at (-1,2){};
\node [vertex, label = right:{$y$}] (3) at (1,2){};
\node[] at (0,-1){$G_1''$};

\foreach \from/\to in {0/1, 2/3}
\draw [edge] (\from) to (\to);

\path[-]          (0)  edge   [bend left]   (2);
\path[-]          (0)  edge   [bend right]   (2);
\path[-]          (1)  edge   [bend right]   (3);
\path[-]          (1)  edge   [bend left]   (3);
\path[-]         (3) edge [min distance=1cm] (3);
\end{scope}

\end{tikzpicture}
\end{center}
By contracting one $xy$-edge in $G_1'$ or in $G_1''$ we
see that $A_3$ is a minor of these graphs, thus $B(G_1')$ and 
$B(G_1'')$ have an $A^3$-minor. This implies that if $G$ is a supergraph
of $G_1$, and $B(G)$ is a connected
$ex_B(\mathcal{L})$-minor free matroid, then $G$ equals 
$G_1(r,d,b)$ for some non-negative integers $r$, $d$ and $b$. Furthermore,
if $G$ contains a subdivision $H$ of one of the previous graphs, then
$G = H$; otherwise, we could find a subdivision of  either $G'_1$ or
of $G''_1$ in $G$. 

\begin{lemma}\label{lem:G1subdivision}
Let $G$ be a graph such that $B(G)$ is an $ex_B(\mathcal{L})$-minor free
connected matroid. If $G$ contains a subdivision of $G_1$, then
$G\in \mathcal{F}_0\cup \mathcal{F}_1$.
\end{lemma}
\begin{proof}
Let $G$ be as in the hypothesis. Above this statement, we argued that if
$G$ contains a subdivision of $G_1$, then $G$ is a subdivision
of a graph $G_1(r,d,b)$ for some non-negative integers $r$, $d$ and $b$.

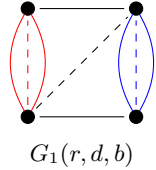
\begin{center}
\begin{tikzpicture}[scale=0.8]

\begin{scope}[xshift=6cm, yshift=0cm, scale=0.9]
\node [vertex] (0) at (-1,0){};
\node [vertex] (1) at (1,0){};
\node [vertex] (2) at (-1,2){};
\node [vertex] (3) at (1,2){};
\node[] at (0,-0.7){\small{$G_1(r,d,b)$}};

\foreach \from/\to in {3/2, 0/1}
\draw [edge] (\from) to (\to);

\draw [edge, dashed] (0) to (3);

\path[-, red]          (0)  edge   [bend left]   (2);
\path[-, red]          (0)  edge   [bend right]   (2);
\draw [edge, dashed, red] (0) to (2);

\path[-, blue]          (1)  edge   [bend right]   (3);
\path[-, blue]          (1)  edge   [bend left]   (3);
\draw [edge, dashed, blue] (1) to (3);
\end{scope}
\end{tikzpicture}
\end{center}

To begin with, recall that $G_1(0,0,0) = G_1$ and $B(G_1)\cong U_{4,6}$. So, if 
$B(G)$ is an $ex_B(\mathcal{L})$-minor free matroid, then
$G$ is a subdivision of at most two edges of $G_1(r,d,b)$;
otherwise, $B(G)$ contains a $C^{2,4}$-minor (Lemma~\ref{lem:uniformextensions}).
We consider the
minimal outerplanar subdivision of two edges of the graphs
$G_1(r,d,b)$ that are not subdivisions of one red and one blue edge. 
To this end, notice that subdividing any two parallel edges or any diagonal
edge, creates a subdivision of $K_{2,3}$, i.e., a non-outerplanar subdivision. 
Thus, any non-outerplanar subdivision of at most two edges of $G_1$ 
belongs to $\mathcal{F}_1$. (This settles the case when $r = d = b = 0$).

 Consider the case when $r = b = 0$ and $d\ge 1$.
 In this case, one (not black) color class corresponds
to the edges on the top left side of the diagonal, and the other one
to the bottom right edges. Subdividing a diagonal edge, 
creates a supergraph of $K_{2,3}$, so this is taken care
of by Proposition~\ref{prop:nonouterplanar}; this also happens if we
subdivide a pair of parallel edges. Now we show that
subdividing two non parallel edges on the same side of the diagonal 
yields a graph whose bicircular matroid has an $ex_B(\mathcal{L})$-minor.
Up to symmetry, the following graph is the unique possible minimal case.
To see that its bicircular matroid has an $ex_B(\mathcal{L})$-minor,
notice that this graph contains $P_2$ as a minor: Remove
the edge $yz$, and contract the edges $zw$ and $zx$.
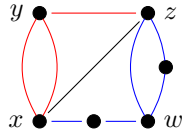
\begin{center}
\begin{tikzpicture}[scale=0.8]
\begin{scope}[xshift=-3cm, yshift=4cm, scale=0.9]
\node [vertex, label = left:{$x$}] (0) at (-1,0){};
\node [vertex, label = right:{$w$}] (1) at (1,0){};
\node [vertex, label = left:{$y$}] (2) at (-1,2){};
\node [vertex, label = right:{$z$}] (3) at (1,2){};
\node [vertex] (s1) at (0,0){};
\node [vertex] (s2) at (1.33,1){};

\draw [edge] (0) to (3);
\draw [edge, red] (2) to (3);

\foreach \from/\to in {0/s1, s1/1}
\draw [edge, blue] (\from) to (\to);

\path[-, red]          (0)  edge   [bend left]   (2);
\path[-, red]          (0)  edge   [bend right]   (2);
\path[-, blue]          (1)  edge   [bend right=10]   (s2);
\path[-, blue]          (s2)  edge   [bend right=10]   (3);
\path[-, blue]          (1)  edge   [bend left]   (3);
\end{scope}
\end{tikzpicture}
\end{center}

The second case we consider is when $r \ge 1$ and $b = 0$. 
Again, the subcases are when $d = 0$ and when $d\ge 1$.
Suppose  that $d = 0$ so,  the red class consists of the $r+2$ parallel edges,
and the remaining edges are blue. By the same arguments as before, we do not
consider subdivisions of a pair of parallel edges (this creates a supergraph
of $K_{2,3}$). The minimal subdivisions of two edges of the same
color class, but not of parallel edges of $G$ are the following ones.
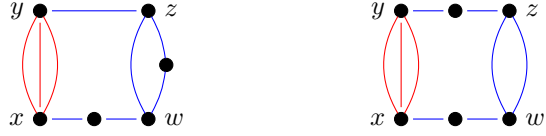
\begin{center}
\begin{tikzpicture}[scale=0.8]

\begin{scope}[xshift=-3cm, yshift=0cm, scale=0.9]
\node [vertex, label = left:{$x$}] (0) at (-1,0){};
\node [vertex, label = right:{$w$}] (1) at (1,0){};
\node [vertex, label = left:{$y$}] (2) at (-1,2){};
\node [vertex, label = right:{$z$}] (3) at (1,2){};
\node [vertex] (s1) at (0,0){};
\node [vertex] (s2) at (1.33,1){};

\draw [edge, red] (2) to (0);

\foreach \from/\to in {0/s1, s1/1, 2/3}
\draw [edge, blue] (\from) to (\to);

\path[-, red]          (0)  edge   [bend left]   (2);
\path[-, red]          (0)  edge   [bend right]   (2);
\path[-, blue]          (1)  edge   [bend right=10]   (s2);
\path[-, blue]          (s2)  edge   [bend right=10]   (3);
\path[-, blue]          (1)  edge   [bend left]   (3);
\end{scope}

\begin{scope}[xshift=3cm, yshift=0cm, scale=0.9]
\node [vertex, label = left:{$x$}] (0) at (-1,0){};
\node [vertex, label = right:{$w$}] (1) at (1,0){};
\node [vertex, label = left:{$y$}] (2) at (-1,2){};
\node [vertex, label = right:{$z$}] (3) at (1,2){};
\node [vertex] (s1) at (0,0){};
\node [vertex] (s2) at (0,2){};

\draw [edge, red] (2) to (0);
\foreach \from/\to in {0/s1, s1/1, 3/s2, s2/2}
\draw [edge, blue ] (\from) to (\to);

\path[-, red]          (0)  edge   [bend left]   (2);
\path[-, red]          (0)  edge   [bend right]   (2);
\path[-, blue]          (1)  edge   [bend right]   (3);
\path[-, blue]          (1)  edge   [bend left]   (3);
\end{scope}
\end{tikzpicture}
\end{center}
By mapping the $yz$-edge (resp.\ the $zw$-path of length $2$) on the left
to the rightmost $zw$-edge (resp.\ the $yz$-path of lentgh $2$) on the right,
we see that the bicircular matroids of these graphs are isomorphic. 
We see that $R_4$ is a minor of the graph on the left after
contracting $zw$ and either of the $xy$-edges.
This implies that $R^4$ is a minor of $B(G)$.  Thus, if $G$
is a subdivision of $G_1(r,0,0)$, then $G$ is a subdivision of at most
one edge of each color class. To conclude this case, we consider
the subcase when $G$ is a subdivision of  $G_1(r,d,0)$ with $d\ge 1$.
In this case, if $G$ is a subdivision of two red or two blue
edges, then $G$ contains a subdivision of either two parallel
edges, or of one of the graphs above. Thus, we only consider the
case when $G$ is a subdivision of some black non-diagonal edge. 
The following is the unique such subdivision.
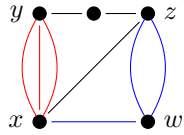
\begin{center}
\begin{tikzpicture}[scale=0.8]
\begin{scope}[scale=0.9]
\node [vertex, label = left:{$x$}] (0) at (-1,0){};
\node [vertex, label = right:{$w$}] (1) at (1,0){};
\node [vertex, label = left:{$y$}] (2) at (-1,2){};
\node [vertex, label = right:{$z$}] (3) at (1,2){};
\node [vertex] (s2) at (0,2){};

\draw [edge, red] (2) to (0);
\path[-, red]          (0)  edge   [bend left]   (2);
\path[-, red]          (0)  edge   [bend right]   (2);

\foreach \from/\to in {3/s2, s2/2, 0/3}
\draw [edge] (\from) to (\to);

\draw [edge, blue] (1) to (0);
\path[-, blue]          (1)  edge   [bend right]   (3);
\path[-, blue]          (1)  edge   [bend left]   (3);
\end{scope}
\end{tikzpicture}
\end{center}
In this case, by contracting one of the $xy$-edges, and removing
one  $zw$-edge, we obtain $R_4''$ as a minor of the graph
above. Therefore, if $G$ is an outerplanar subdivision of either
$G_1(r,d,0)$ for some $d\ge 1$, then $G\in \mathcal{F}_1$. By symmetry,
and together with the previous subcase (when $d = 0$), 
the claim holds when $G$ is an outerplanar subdivision  of $G_1(r,d,0)$
or of $G_1(0,d,b)$ for some non-negative integer $d$.

The final case is when $r$ and $b$ are positive integers. In this case,
the definitions of the red and blue edge do not depend on $d$;
each (not black) color class corresponds to a class of parallel edges.
By the same argument as above, we do not
consider the cases when $G$ is a subdivision of a pair of parallel edges,
nor when is a subdivision of some diagonal edges. Thus, it remains
to show that if $G$ is a subdivision of either the top or the bottom edge, then 
has an $ex_B(\mathcal{L})$-minor. Such a minimal subdivision, looks as
follows, and in this case we see that it contains $D_4$ (Figure~\ref{fig:bicircularEX})
as a minor by contracting the $xw$-edge.
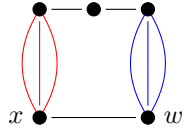
\begin{center}
\begin{tikzpicture}[scale=0.8]
\begin{scope}[scale = 0.9]
\node [vertex, label = left:{$x$}] (0) at (-1,0){};
\node [vertex, label = right:{$w$}] (1) at (1,0){};
\node [vertex] (2) at (-1,2){};
\node [vertex] (3) at (1,2){};
\node [vertex] (s2) at (0,2){};

\foreach \from/\to in {0/1, 3/s2, s2/2}
\draw [edge] (\from) to (\to);

\path[-, red]          (0)  edge   [bend left]   (2);
\path[-, red]          (0)  edge   [bend right]   (2);
\draw [edge, red] (0) to (2);

\path[-, blue]          (1)  edge   [bend right]   (3);
\path[-, blue]          (1)  edge   [bend left]   (3);
\draw [edge, blue] (1) to (3);
\end{scope}

\end{tikzpicture}
\end{center}

After this exhaustive case checking, we conclude that
if $G$ is an outerplanar graph that contains a subdivision of $G_1$,
then $G$ is a subdivision of at most one red edge and at most one blue
edge of  $G_1(r,d,b)$ for some non-negative integers $r$, $d$ and $b$.
Therefore, $G\in\mathcal{F}_1$ (recall that if $G$ is not an outerplanar
graph, then  Proposition~\ref{prop:nonouterplanar} asserts that
$G\in\mathcal{F}_0$).
\end{proof}


\subsection*{\normalfont{\textsc{$2$-connected graphs}}}

The aim of this case is to settle Proposition~\ref{prop:F0-F4} for
$2$-connected graphs. Notice that if $G$ is a $2$-connected graph, then
$G$ is a subdivision of some $2$-connected graph $H$ of minimum
degree $3$ (if $G$ has minimum degree $3$ then $G = H$).
By Lemmas~\ref{lem:2connected} and~\ref{lem:G1subdivision},
if $H$ has a pair of non-adjacent vertices, then $H$ is a spanning
supergraph of $G_1$ so, $|V(H)| = 4$. On the contrary, if every pair of vertices of
$H$ are adjacent, then $H$ contains $K_4$ as a subgraph. Hence, by
Lemma~\ref{lem:K4subdivision}, we conclude that $H = K_4$. 
In both cases, we conclude that if $H$ has at least four vertices, then 
it has exactly four vertices. Thus, we conclude the following statement.

\begin{proposition}\label{prop:boundedsubdivision}
Let $G$ be a $2$-connected graph. If $B(G)$ is an $ex_B(\mathcal{L})$-minor
free matroid, then $G$ is a subdivision of a graph on at most $4$ vertices.
Moreover, if $G$ is a subdivision of a graph on $4$ vertices of degree at least
three, then $G\in \mathcal{F}_0\cup \mathcal{F}_1$.
\end{proposition}
\begin{proof}
By the arguments in the paragraph above, we conclude that
$G$ is a subdivision of a graph on at most $4$ vertices. Moreover, 
with the same arguments we also notice that if $G$ is a subdivision of
a graph on $4$ vertices of minimum degree $3$, then $G$ contains a
subdivision of $K_4$ or of $G_1$. Thus, the claim follows by 
Lemmas~\ref{lem:K4subdivision} and~\ref{lem:G1subdivision}.
\end{proof}

Now we  consider subdivisions of graphs on three vertices.
Recall that the graph $2K_3$ is obtained by duplicating every edge of the triangle,
and its bicircular matroid is the uniform matroid $U_{3,6}$. Notice that
adding a loop to $2K_3$ creates a copy of $A_3$ as a subgraph. 
Also, adding   one more parallel edge to each parallel class yields the graph
$3K_3$. The bicircular matroids of $3K_3$ and of $A_3$ belong to
$ex_B(\mathcal{L})$. Thus,  if $G$ contains a subdivision of $2K_3$, 
and $B(G)$ is a connected $ex_B(\mathcal{L})$-minor free matroid, 
then $G$ is a subdivision of $2K_3(r,b)$ for some non-negative integers
$b$ and $r$. We use these observations to prove the following lemma.

\begin{lemma}\label{lem:2K3subdivision}
Let $G$ be an outerplanar graph such that $B(G)$ is an $ex_B(\mathcal{L})$-minor free
connected matroid. If $G$ contains a subdivision of $2K_3$, then
$G\in \mathcal{F}_2$.
\end{lemma}
\begin{proof}
Before this lemma, we argued that if $G$ contains a subdivision $H$ of
$2K_3$, then $G = H$ and $H = 2K_3(r,b)$ for some non-negative
integers $r$ and $b$. 
Recall that we distinguished some
colored edges in $2K_3(r,b)$, we proceed to show that
$G$ must be a subdivision of at most one red and at most one
blue edge of $2K_3(r,b)$.
As we have done before, we highlight that subdividing a pair
of parallel edges of $2K_3$ creates a subdivision of $K_{2,3}$. 
Since we are considering outerplanar graphs, we do not consider
such subdivisions.

First, consider the case when $r = b = 0$, i.e., when
$G$ is a subdivision of $2K_3$, so $B(G)$ is a series extension
of $U_{3,6}$. By Lemma~\ref{lem:uniformextensions}, we know that if $G$
contains a subdivision of at least three different edges of $2K_3$, then 
$B(G)$ has a $C^{2,4}$-minor. Thus, $G$ is a subdivision
of at most two non-parallel edges of $2K_3$, and so, $G\in\mathcal{F}_2$. 

Suppose now that $r\ge 1$ and $b = 0$. To show that 
$G$ is a subdivision of at most one red edge and at most one blue
edge, we consider the minimal subdivision of $2K_3(1,0)$ that
is not such a subdivision (nor a subdivision of two parallel edges).
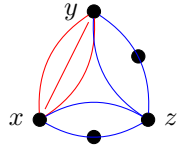
\begin{center}
\begin{tikzpicture}[scale=0.8]

\begin{scope}[xshift=5cm, yshift=0cm, scale=0.9]
\node [vertex, label=left:{$x$}] (0) at (-1,0){};
\node [vertex, label=right:{$z$}] (1) at (1,0){};
\node [vertex, , label=left:{$y$}] (2) at (0,2){};
\node [vertex] (a) at (0.82,1.17){};
\node [vertex] (b) at (0,-0.32){};

\foreach \from/\to in {0/2}
\draw [edge, red] (\from) to (\to);

\path[-, red]          (0)  edge   [bend left]   (2);
\path[-, red]          (0)  edge   [bend right]   (2);
\path[-, blue]          (1)  edge   [bend right]   (2);
\path[-, blue]          (1)  edge   [bend left]   (2);
\path[-, blue]          (1)  edge   [bend right]   (0);
\path[-, blue]          (1)  edge   [bend left]   (0);
\end{scope}
\end{tikzpicture}
\end{center}
By removing the $xz$-edge and contracting one $xy$-edge,
we see that this graph contains $R_4''$ as a minor. Therefore, if $G$ is
an outerplanar subdivision of $2K_3(1,0)$, then it is a subdivision of at
most one red and at most one blue edge of $2K_3(r,0)$. By symmetric
arguments,  the claim also follows when $r = 0$ and $b\ge 1$.
So, we conclude the proof by considering the case when $r$ and 
$b$ are positive integers.
In this case, the red edges correspond to the $(r+2)$-parallel edges, and
the blue edges to the $(b+2)$-parallel edges. Since we are considering
outerplanar subdivisions, we show that the bicircular matroid of the 
following graph has an $ex_B(\mathcal{L})$-minor.
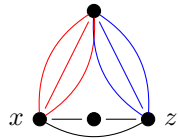
\begin{center}
\begin{tikzpicture}[scale=0.8]
\begin{scope}[scale=0.9]

\node [vertex, label = left:{$x$}] (0) at (-1,0){};
\node [vertex, label = right:{$z$}] (1) at (1,0){};
\node [vertex] (2) at (0,2){};
\node [vertex] (a) at (0,0){};

\foreach \from/\to in {1/a, a/0}
\draw [edge] (\from) to (\to);

\draw [edge, red] (0) to (2);
\path[-, red]          (0)  edge   [bend left]   (2);
\path[-, red]          (0)  edge   [bend right]   (2);
\draw [edge, blue] (1) to (2);
\path[-, blue]          (1)  edge   [bend right]   (2);
\path[-, blue]          (1)  edge   [bend left]   (2);
\path[-]          (1)  edge   [bend left]   (0);
\end{scope}
\end{tikzpicture}
\end{center}
In this case, by removing the $xz$-edge, we see that 
$D_4$ is a subgraph of the graph above. So, its  bicircular matroid
 has a $D^4$-minor. Hence, if $G$ is a subdivision
of $2K_3(r,b)$ for some positive integers $r$ and $b$, then 
$G$ is a subdivision of at most one blue and  at most one red edge. 
After considering all possible cases, we conclude that the statement
of this proposition holds.
\end{proof}

To conclude the case checking for outerplanar subdivisions of 
graphs on three vertices, we recall the definition of the 
graphs $K_3(r,j,l)$. These graphs are obtained from a triangle
by adding $j$ parallel copies to some edge; adding $k$ parallel copies to another;
and  $j$ loops incident in some vertex that is not incident
with both classes of parallel edges. We depict them as follows, together
with the edges colorings used to define the family $\mathcal{F}_3$.
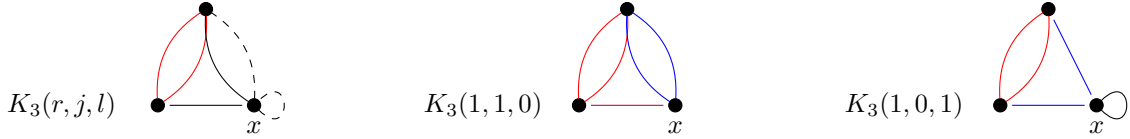
\begin{center}
\begin{tikzpicture}[scale=0.8]

\begin{scope}[xshift=-7cm, yshift=0cm, scale=0.8]
\node [vertex] (0) at (-1,0){};
\node [vertex, label = 270:{$x$}] (1) at (1,0){};
\node [vertex] (2) at (0,2){};
\node[] at (-3,0){$K_3(r,j,l)$};

\draw [edge] (0) to (1);

\path[-, red]          (0)  edge   [bend left]   (2);
\path[-, red]          (0)  edge   [bend right]   (2);
\path[-, dashed]          (1)  edge   [bend right]   (2);
\path[-]          (1)  edge   [bend left]   (2);

\path[-, dashed]         (1) edge [min distance=1cm, in = 45, out = -45] (1);

\end{scope}

\begin{scope}[xshift=0cm, yshift=0cm, scale=0.8]
\node [vertex] (0) at (-1,0){};
\node [vertex, label = 270:{$x$}] (1) at (1,0){};
\node [vertex] (2) at (0,2){};
\node[] at (-3,0){$K_3(1,1,0)$};

\draw [edge, purple] (0) to (1);

\path[-, red]          (0)  edge   [bend left]   (2);
\path[-, red]          (0)  edge   [bend right]   (2);
\path[-, blue]          (1)  edge   [bend right]   (2);
\path[-, blue]          (1)  edge   [bend left]   (2);

\end{scope}

\begin{scope}[xshift=7cm, yshift=0cm, scale=0.8]
\node [vertex] (0) at (-1,0){};
\node [vertex, label = 270:{$x$}] (1) at (1,0){};
\node [vertex] (2) at (0,2){};
\node[] at (-3,0){$K_3(1,0,1)$};

\foreach \from/\to in {0/1, 1/2}
\draw [edge, blue] (\from) to (\to);

\path[-, red]          (0)  edge   [bend left]   (2);
\path[-, red]          (0)  edge   [bend right]   (2);
\path[-]         (1) edge [min distance=1cm, in = 45, out = -45] (1);

\end{scope}
\end{tikzpicture}
\end{center}
\begin{lemma}\label{lem:3vertexsubdivision}
Let $G$ be a $2$-connected outerplanar graph such that $B(G)$ is an
$ex_B(\mathcal{L})$-minor free matroid. If $G$ is a subdivision
of a graph on three vertices of minimum degree three, 
then $G\in\mathcal{F}_2\cup \mathcal{F}_3$.
\end{lemma}
\begin{proof}
Let $H$ be the graph obtained from $G$ by contracting all
subdivided edges. By the hypothesis of this lemma, we assume that
$H$ has minimum degree three and $|V(H)| = 3$.
If $H$ contains  $2K_3$ as a subgraph, then $G$ contains
a subdivision of $2K_3$ and the claim follows by
Lemma~\ref{lem:2K3subdivision}. Now, we show that if $H$ does not
contain a copy of $2K_3$, then $H = K_3(r,j,l)$ for some
non-negative integers $r$, $j$ and $l$. 

Notice that since $H$ is $2$-connected (because is obtained from
contracting subdivided edges of a $2$-connected graph),
then $H$ is a supergraph of $K_3$. So, if every vertex of $H$ has
a loop, then $H$ contains $W_3$ as a subgraph, and so $B(G)$ 
has a $\mathcal{W}^3$-minor. Thus, $H$ has at least one  loopless
vertex $v$. Since $d_H(v) \ge 3$, then $v$ is incident with a pair
of parallel edges. Let $u$ be the other vertex incident in this same
pair of parallel edges, and let $w$ be the remaining vertex. If both, 
$u$ and $w$ are looped vertices, then $H$ contains $A_3'$ as a subgraph,
which implies that $B(G)$ has an $A^3$-minor. Thus, $H$ has at most
one looped vertex. Finally, if this vertex is also incident with
two classes of parallel edges, $H$ contains a copy of $A_3$,
which also implies that $B(G)$ has an $A^3$-minor. 
Since we are assuming that $B(G)$ is $ex_B(\mathcal{L})$-free, 
we conclude that $H$ has at most one looped vertex, and this
vertex is not incident with two classes of parallel edges. Therefore, 
$H = K_3(r,j,l)$ for some non-negative integers $r$, $j$ and $l$,
and so, $G$ is  a subdivision of $K_3(r,j,l)$. Furthermore, 
since $H$ has minimum degree three, then $r\ge 1$
and $j+l\ge 1$. 
 
Recall that the bicircular matroid of  $K_3(1,1,0)$ and of $K_3(1,0,1)$ is the 
uniform matroid $U_{3,5}$. Hence, by Lemma~\ref{lem:uniformextensions},
we know that any subdivision of at least three edges of $K_3(1,1,0)$ or of
$K_3(1,0,1)$ has a bicircular matroid with a $C^{2,4}$-minor.
By the edge symmetries of these graphs,  any subdivision of at least three
edges of $K_3(r,j,l)$ (where $r\ge 1$ and $j+l\ge 1$) has a bicircular
matroid with a $C^{2,4}$-minor.

Let $G$ be a subdivision of at most two non-loop edges%
\footnote{Any subdivision of some loop of $K_3(r,j,l)$ yields a graph with a cut vertex.}
of $K_3(r,j,l)$. We begin by considering all cases when
$l\le 1$ (see the top two rows of Figure~\ref{fig:H3}). In particular,
if $r \ge 1$ and $j+l = 1$, then any subdivision of any pair of non-parallel edges
of $K_3(r,j,l)$  belongs to $\mathcal{F}_3$ (this cases correspond to subdivisions
of $K_3(1,1,0)$, of $K_3(r,1,0)$, of $K_3(1,0,1)$ and of $K_3(r,0,1)$). Also,
if $r = 1 = l$, then any subdivision of a pair of non-parallel edges of $K_3(1,j,1)$
belongs to $\mathcal{F}_3$, so the claim follows for subdivisions
of $K_3(1,1,0)$, of $K_3(r,1,0)$, of $K_3(1,0,1)$, of $K_3(r,0,1)$ and of
$K_3(1,j,1)$. The remaining subcases (of the  case  $l\le 1$) are
subdivisions of $K_3(r,j,l)$ where $r,j\ge 2$ (right most graphs in
the top two columns of Figure~\ref{fig:H3}). In these cases,
we must show that $G$ is a subdivision of at most one edge in each parallel
classes. Since we are working with outerplanar graphs, we assume that
it is not a subdivision of edges in the same parallel class. Thus, it suffices
to show that $G$ does
not contain a subdivision of the bottom edge of $K_3(r,j,l)$. Again, we consider
the minimal such subdivisions, and see that their bicircular matroids
have an $ex_B(\mathcal{L})$-minor. We depict these cases below, 
and clearly, their bicircular matroid is $D^4$ (Figure~\ref{fig:bicircularEX})
which belongs to $ex_B(\mathcal{L})$.
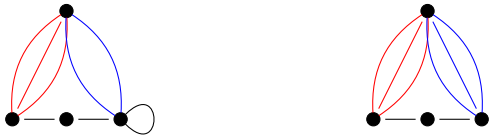
\begin{center}
\begin{tikzpicture}[scale=0.8]
\begin{scope}[xshift = -3cm, scale=0.9]%
\node [vertex] (0) at (-1,0){};
\node [vertex] (1) at (1,0){};
\node [vertex] (2) at (0,2){};
\node [vertex] (s') at (0,0){};

\foreach \from/\to in {0/s', s'/1}
\draw [edge] (\from) to (\to);

\draw [edge, red] (0) to (2);

\foreach \from/\to in {0/2, 2/0}
\path[-, red]          (\from)  edge   [bend left]   (\to);

\foreach \from/\to in {1/2, 2/1}
\path[-, blue]          (\from)  edge   [bend left]   (\to);

\path[-]         (1) edge [min distance=1cm, in = 45, out = -45] (1);

\end{scope}

\begin{scope}[xshift = 3cm, scale=0.9]%
\node [vertex] (0) at (-1,0){};
\node [vertex] (1) at (1,0){};
\node [vertex] (2) at (0,2){};
\node [vertex] (s') at (0,0){};

\foreach \from/\to in {0/s', s'/1}
\draw [edge] (\from) to (\to);

\draw [edge, red] (0) to (2);
\draw [edge, blue] (1) to (2);

\foreach \from/\to in {0/2, 2/0}
\path[-, red]          (\from)  edge   [bend left]   (\to);

\foreach \from/\to in {1/2, 2/1}
\path[-, blue]          (\from)  edge   [bend left]   (\to);
\end{scope}
\end{tikzpicture}
\end{center}
 
In the paragraph above, we settled all cases when $l\le 1$,
now we assume that $l\ge 2$ (see the bottom row of
Figure~\ref{fig:H3}).  We first observe that $G$ is a subdivision of at most
one non-loop edge of $K_3(r,j,l)$. 
Since we are considering only outerplanar subdivisions, then we do not
consider subdivisions of parallel edges. Thus, it suffices to show that the
bicircular matroids of the following two graphs have an $ex_B(\mathcal{L})$-minor. 
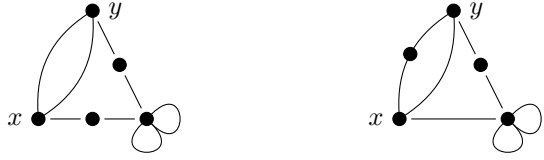
\begin{center}
\begin{tikzpicture}[scale=0.8]
\begin{scope}[xshift = -3cm, scale=0.9]%
\node [vertex, label = left:{$x$}] (0) at (-1,0){};
\node [vertex] (1) at (1,0){};
\node [vertex, label = 0:{$y$}] (2) at (0,2){};
\node [vertex] (s) at (0.5,1){};
\node [vertex] (s') at (0,0){};

\foreach \from/\to in {0/s', s'/1, 1/s, s/2}
\draw [edge] (\from) to (\to);

\path[-]          (0)  edge   [bend left]   (2);
\path[-]          (0)  edge   [bend right]   (2);
\path[-]         (1) edge [min distance=1cm, in = 45, out = -45] (1);
\path[-]         (1) edge [min distance=1cm, in = -45, out = -135] (1);

\end{scope}

\begin{scope}[xshift = 3cm, scale=0.9]%
\node [vertex, label = left:{$x$}] (0) at (-1,0){};
\node [vertex] (1) at (1,0){};
\node [vertex, label = 0:{$y$}] (2) at (0,2){};
\node [vertex] (s) at (0.5,1){};
\node [vertex] (s') at (-0.8,1.2){};

\foreach \from/\to in {0/1, 1/s, s/2}
\draw [edge] (\from) to (\to);

\path[-]          (0)  edge   [bend left]   (2);
\path[-]          (0)  edge   [bend right]   (2);
\path[-]         (1) edge [min distance=1cm, in = 45, out = -45] (1);
\path[-]         (1) edge [min distance=1cm, in = -45, out = -135] (1);

\end{scope}
\end{tikzpicture}
\end{center}
Contracting one $xy$-edge in either of these graphs,
shows that their bicircular matroids  have an $R^4$-minor.
Since $B(G)$ is an $ex_B(\mathcal{L})$-minor free matroid,  then $G$ is a
subdivision of  at most one non-loop edge of $K_3(r,j,l)$ (recall that
we are assuming that $r\ge 1$ and $l\ge 2$). In particular, since
any subdivision of at most one edge of $K_3(1,0,l)$
belongs to  $\mathcal{F}_3$
(see left most graph in the bottom row of Figure~\ref{fig:H3}), then
the claim is settled for $l\ge 2$, $r = 1$ and $j = 0$. 

To conclude the proof we consider the case when $G$ is a subdivision of
$K_3(r,j,l)$, where $r,l\ge 2$. To show that $G\in \mathcal{F}_3$,
we argue that $G$ is a subdivision  of at most one red edge, i.e.,
an edge that is not incident with a loop.  Again,
we assume that $G$ is not a subdivision of loop, nor of a pair of parallel edges. 
Hence, it suffices to show that the bicircular matroid of the 
following graph has an $ex_B(\mathcal{L})$-minor. 
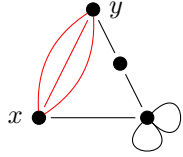
\begin{center}
\begin{tikzpicture}[scale=0.8]

\begin{scope}[scale=0.9]%
\node [vertex, label=left:{$x$}] (0) at (-1,0){};
\node [vertex] (1) at (1,0){};
\node [vertex, label = 0:{$y$}] (2) at (0,2){};
\node [vertex] (s) at (0.5,1){};

\foreach \from/\to in {0/1, 1/s, s/2}
\draw [edge] (\from) to (\to);

\draw [edge, red] (0) to (2);

\path[-, red]          (0)  edge   [bend left]   (2);
\path[-, red]          (0)  edge   [bend right]   (2);
\path[-]         (1) edge [min distance=1cm, in = 45, out = -45] (1);
\path[-]         (1) edge [min distance=1cm, in = -45, out = -135] (1);

\end{scope}
\end{tikzpicture}
\end{center}
Contract one of the $xy$-edges of the graph above. This minor
is a triangle with two classes of parallel loops, i.e., this minor equals
$R_3$. Thus, the bicircular matroid of the graph above has an $R^3$-minor. 
Which implies that if $G$ is an outerplanar subdivision of $K_3(r,j,l)$
where $r,l\ge 2$, then $G$ is a subdivision of at most 
one red edge. Therefore,  $G\in\mathcal{F}_3$. 
This completes all possible cases, so the claim follows.
\end{proof}

Having proved all these technical lemmas, we are now ready to show that 
Proposition~\ref{prop:F0-F4} holds for $2$-connected graphs.

\begin{proposition}\label{prop:2connected}
Let $G$ be a $2$-connected graph. If $B(G)$ is an
$ex_B(\mathcal{L})$-minor free connected matroid, then
$G$ belongs to
$\mathcal{F}_0\cup \mathcal{F}_1\cup \mathcal{F}_2 \cup \mathcal{F}_3$.
\end{proposition}
\begin{proof}
When $G$ is not an outerplanar graph, the claim follows by
Proposition~\ref{prop:nonouterplanar}. Also,
Proposition~\ref{prop:boundedsubdivision} asserts that if a $2$-connected
graph $G$ satisfies that $B(G)$ is an $ex_B(\mathcal{L})$-minor
free matroid, then $G$ is a subdivision of a graph of at most four
vertices. Moreover, the same proposition settles the case when $G$ is a subdivision
of a graph on four vertices of minimum degree three.
Similarly, Lemma~\ref{lem:3vertexsubdivision} shows that if $G$ is a subdivision of
a graph on $3$ vertices  of minimum degree three, then
$G\in \mathcal{F}_2\cup \mathcal{F}_3$.  To conclude this case checking we
consider the case when $G$ is an outerplanar subdivision of a graph on at
most two vertices.  This case follows almost trivially.

A graph $H$ on two vertices $x$ and $y$ consists of a class of
parallel non-loop edges, some loops incident with $x$ or $y$.
Subdividing any loop yields a graph with a cut vertex. Also, subdividing
at least three different parallel edges creates a subdivision of
$K_{2,3}$. Thus,  an outerplanar $2$-connected subdivision of a graph
on $2$ vertices is a subdivision of at most two non parallel edges of $H$. 
If each vertex of $H$ has at most one loop, then any such subdivision
belongs to $\mathcal{F}_3$. The remaining cases are when exactly one
vertex has at least two loops, and when both vertices have two loops. 
In the latter, we must verify that there are no subdivided edges, and
in the former we must verify that there is at most one subdivided edge. 
This follows because the bicircular matroids of the following graphs
are either $R^3$ or $R^4$ (see Figure~\ref{fig:bicircularEX}),
which both belong to $ex_B(\mathcal{L})$.
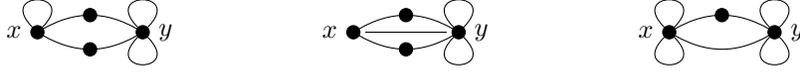
\begin{center}
\begin{tikzpicture}[scale=0.7]
\begin{scope}[xshift = -6cm]
\node [vertex, label =left:{$x$}] (n) at (-1,0){};
\node [vertex, label = right:{$y$}] (m) at (1,0){};
\node [vertex] at (0,0.325){};
\node [vertex] at (0,-0.325){};

\path[-]          (n)  edge   [bend right]   (m);
\path[-]          (n)  edge   [bend left]   (m);

\path[-]	     (n) edge [min distance=1cm]  (n);
\path[-]	     (m) edge [min distance=1cm]  (m);
\path[-]	     (m) edge [min distance=1cm, out = 225, in = 315]  (m);

\end{scope}

\begin{scope}[xshift = 0cm]
\node [vertex, label =left:{$x$}] (n) at (-1,0){};
\node [vertex, label = right:{$y$}] (m) at (1,0){};
\node [vertex] at (0,0.325){};
\node [vertex] at (0,-0.325){};

\path[-]          (n)  edge   [bend right]   (m);
\path[-]          (n)  edge   [bend left]   (m);
\draw[edge]    (n)   to  (m);

\path[-]	     (m) edge [min distance=1cm]  (m);
\path[-]	     (m) edge [min distance=1cm, out = 225, in = 315]  (m);

\end{scope}
\begin{scope}[xshift = 6cm]
\node [vertex, label =left:{$x$}] (n) at (-1,0){};
\node [vertex, label = right:{$y$}] (m) at (1,0){};
\node [vertex] at (0,0.325){};

\path[-]          (n)  edge   [bend right]   (m);
\path[-]          (n)  edge   [bend left]   (m);

\path[-]	     (n) edge [min distance=1cm]  (n);
\path[-]	     (m) edge [min distance=1cm]  (m);
\path[-]	     (m) edge [min distance=1cm, out = 225, in = 315]  (m);
\path[-]	     (n) edge [min distance=1cm, out = 225, in = 315]  (n);
\end{scope}

\end{tikzpicture}
\end{center}

The final case is when $G$ is a $2$-connected subdivision of 
a graph on one vertex. In this case, $G$ is a cycle with at most
one looped vertex, and such a graph belongs to $\mathcal{F}_3$.
The claim now follows.
\end{proof}


\subsection*{\normalfont{\textsc{Outerplanar graphs with cut vertices}}}

Propositions~\ref{prop:nonouterplanar} and~\ref{prop:2connected}
settle Proposition~\ref{prop:F0-F4} for non-outerplanar graphs and for
$2$-connected graphs. We conclude by considering the
complement case.  In particular, we consider graphs that do not contain
subdivisions of $K_4$ nor of $K_{2,3}$. Also, in the 
paragraph before Lemma~\ref{lem:2K3subdivision}, we argued
that if $G$ contains a subdivision $H$ of $2K_3$ and $B(G)$ is 
$ex_B(\mathcal{L})$-minor free, then $G = H$. Which implies that
 $G$ is a $2$-connected graph. With analogous arguments, we conclude
 that if $G$ contains a subdivision of $G_1$ then $G$ is a $2$-connected graph.
 These observations, together with Lemma~\ref{lem:2connected}
 imply the following statement. 

\begin{observation}\label{obs:finalcase}
Let $G$ be an outerplanar graph such that $B(G)$ is an $ex_B(\mathcal{L})$-minor
free connected matroid. If $G$ has a cut vertex and $EB$ is an end block of
$G$, then one of the following holds:
	\begin{enumerate}
		\item $EB$ is a cycle with possible loops on the corresponding cut vertex, 
		\item $EB$ is a subdivision of a graph on $2$ vertices of minimum
		degree $3$, or
		\item $EB$ is a subdivision of a graph on three vertices
		but contains no subdivision of $2K_3$.
	\end{enumerate}
\end{observation}

Recall that, by Proposition~\ref{prop:blockpath},
if the bicircular matroid of graph $G$ is $C^{2,4}$-minor free, 
then the block tree of $G$ is a path, and every middle block contains
exactly two vertices. By definition of $\mathcal{F}_4$, such a graph
$G$ belongs to $\mathcal{F}_4$ if the end blocks 
of $G$ are subgraphs of a subdivision of some red edge of the following graphs,
where the cut vertex is corresponding cut vertex is $x$.
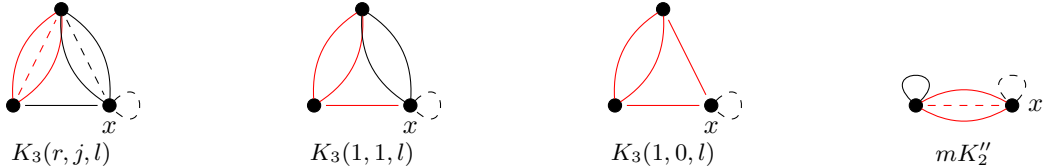
\begin{center}
\begin{tikzpicture}[scale=0.8]

\begin{scope}[xshift=-7.5cm, yshift=0cm, scale=0.8]
\node [vertex] (0) at (-1,0){};
\node [vertex, label = 270:{$x$}] (1) at (1,0){};
\node [vertex] (2) at (0,2){};
\node[] at (0,-1){\small{$K_3(r,j,l)$}};

\draw [edge] (0) to (1);
\draw [edge, red, dashed] (0) to (2);
\draw [edge, dashed] (2) to (1);

\path[-, red]          (0)  edge   [bend left]   (2);
\path[-, red]          (0)  edge   [bend right]   (2);
\path[-]          (1)  edge   [bend right]   (2);
\path[-]          (1)  edge   [bend left]   (2);

\path[-, dashed]	     (1) edge [min distance=1cm, out = 45, in = 315]  (1);

\end{scope}

\begin{scope}[xshift=-2.5cm, yshift=0cm, scale=0.8]
\node [vertex] (0) at (-1,0){};
\node [vertex, label = 270:{$x$}] (1) at (1,0){};
\node [vertex] (2) at (0,2){};
\node[] at (0,-1){\small{$K_3(1,1,l)$}};

\draw [edge, red] (0) to (1);

\path[-, red]          (0)  edge   [bend left]   (2);
\path[-, red]          (0)  edge   [bend right]   (2);
\path[-]          (1)  edge   [bend right]   (2);
\path[-]          (1)  edge   [bend left]   (2);

\path[-, dashed]	     (1) edge [min distance=1cm, out = 45, in = 315]  (1);

\end{scope}

\begin{scope}[xshift=2.5cm, yshift=0cm, scale=0.8]
\node [vertex] (0) at (-1,0){};
\node [vertex, label = 270:{$x$}] (1) at (1,0){};
\node [vertex] (2) at (0,2){};
\node[] at (0,-1){\small{$K_3(1,0,l)$}};

\foreach \from/\to in {0/1, 1/2}
\draw [edge, red] (\from) to (\to);
\path[-, dashed]	     (1) edge [min distance=1cm, out = 45, in = 315]  (1);

\path[-, red]          (0)  edge   [bend left]   (2);
\path[-, red]          (0)  edge   [bend right]   (2);

\end{scope}

\begin{scope}[xshift=7.5cm, yshift=0cm, scale=0.8]
\node [vertex] (0) at (-1,0){};
\node [vertex, label = 0:{$x$}] (2) at (1,0){};
\node[] at (0,-1){\small{$mK_2''$}};

\draw [edge, red, dashed] (0) to (2);
\path[-, red]          (0)  edge   [bend left]   (2);
\path[-, red]          (0)  edge   [bend right]   (2);
\path[-]	     (0) edge [min distance=1cm]  (0);
\path[-, dashed]	     (2) edge [min distance=1cm]  (2);
\end{scope}
\end{tikzpicture}
\end{center}

\begin{proposition}\label{prop:excludedcase}
Let $G$ be an outerplanar graph with some cut vertex. 
If $B(G)$ is an $ex_B(\mathcal{L})$-minor free connected
matroid, then $G\in\mathcal{F}_4$.
\end{proposition}
\begin{proof}
By the arguments in this subsection, it suffices to show
the end blocks of $G$ are a subgraph of some subdivision of at most
one red edge of either of $K_3(r,j,l)$, $K_3(1,1,l)$, $K_3(1,0,l)$ or $mK_2''$
depicted above. We divide this proof in
the cases laid out by  Observation~\ref{obs:finalcase}. In particular, when both
end blocks satisfy the first statement of Observation~\ref{obs:finalcase}, then
the claim follows immediately. Also, since $B(G)$ is a connected matroid, no end
block of $G$ consists of exactly one edge. This implies that each end block
can be contracted either to a pair of parallel edges, or to a leaf with
a loop on the end vertex. We use this simple observation throughout this proof.
Also, with out loss of generality we suppose that there are no loops
incident with the cut vertex $x$
(any minor of $G$ that does not contain the loops on $x$, is also a minor of $G$).

Suppose that an end block $EB$ is a subdivision of a graph
on two vertices of minimum degree three. Denote by $x$ the
corresponding cut vertex and by $y$ the other vertex in $EB$ of 
degree at least $3$.  If $EB$ is an empty subdivision, i.e., $EB$ is a
block on two vertices, then there is nothing to prove. Now we show that
if $EB$ contains at least one subdivided edge, then $y$ has at most one loop
and $EB$ has at most one subdivision class; in other words, 
$EB$ is a subdivision of at most one red edge of some $mK_2''$. 
First suppose that there are at least two loops incident with $y$. 
Since $x$ is a cut vertex and $B(G)$ is a connected matroid,
then $G$ must contain one  of the graphs as a minor.

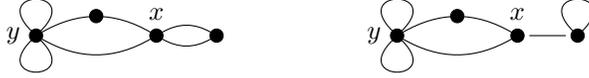
\begin{center}

\begin{tikzpicture}[scale=0.8]

\begin{scope}[xshift=-3cm]%
\node [vertex, label = 180:{$y$}] (n) at (-2,0){};
\node [vertex, label=90:{$x$}] (m) at (0,0){};
\node [vertex] (m') at (1,0){};
\node [vertex] at (-1,0.325){};

\foreach \from/\to in {n/m, m/n, m/m', m'/m}
\path[-]          (\from)  edge   [bend right]   (\to);

\path[-]	     (n) edge [min distance=1cm]  (n);
\path[-]	     (n) edge [min distance=1cm, out = 225, in = 315]  (n);
\end{scope}
\begin{scope}[xshift=3cm]
\node [vertex, label = 180:{$y$}] (n) at (-2,0){};
\node [vertex, label=90:{$x$}] (m) at (0,0){};
\node [vertex] (m') at (1,0){};
\node [vertex] at (-1,0.325){};

\path[-]          (n)  edge   [bend right]   (m);
\path[-]          (n)  edge   [bend left]   (m);

\draw[edge] (m) to (m');

\path[-]	     (n) edge [min distance=1cm]  (n);
\path[-]	     (m') edge [min distance=1cm]  (m');

\path[-]	     (n) edge [min distance=1cm, out = 225, in = 315]  (n);

\end{scope}
\end{tikzpicture}
\end{center}
%
Since the graph to the left is $R_4$ and the graph to the right is $R_4'$, 
we conclude that there is at most one loop incident with $y$. 
Finally, we show that there is at most one subdivision class. To this end, 
notice that by contracting the edge $xy$ in the following graphs,
we obtain $P_2$ (on the left) and $P_3$ (on the right)  as a minor of $G$. 
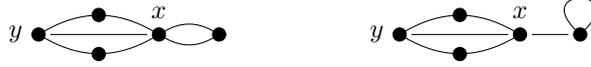
\begin{center}

\begin{tikzpicture}[scale=0.8]

\begin{scope}[xshift=-3cm]%
\node [vertex, label = 180:{$y$}] (n) at (-2,0){};
\node [vertex, label=90:{$x$}] (m) at (0,0){};
\node [vertex] (m') at (1,0){};
\node [vertex] at (-1,0.325){};
\node [vertex] at (-1,-0.325){};

\foreach \from/\to in {n/m, m/n, m/m', m'/m}
\path[-]          (\from)  edge   [bend right]   (\to);

\draw[edge]    (n)   to  (m);
\end{scope}
\begin{scope}[xshift=3cm]
\node [vertex, label = 180:{$y$}] (n) at (-2,0){};
\node [vertex, label=90:{$x$}] (m) at (0,0){};
\node [vertex] at (-1,0.325){};
\node [vertex] at (-1,-0.325){};
\node [vertex] (m') at (1,0){};

\path[-]          (n)  edge   [bend right]   (m);
\path[-]          (n)  edge   [bend left]   (m);
\draw[edge]    (n)   to  (m);

\draw[edge] (m) to (m');
\path[-]	     (m') edge [min distance=1cm]  (m');

\end{scope}
\end{tikzpicture}
\end{center}
In both cases, this contradicts the fact that $B(G)$ is $ex_B(\mathcal{L})$-minor
free. So, if and end block $EB$ of $G$ satisfies the second statement
of Observation~\ref{obs:finalcase}, then $EB$ is a subgraph of a subdivision
of at most one red edge of a graph $mK_2''$.

To conclude this proof we verify the case when $EB$ is a subdivision 
of a graph on three vertices of minimum degree $3$. Denote by $x$, $y$ and
$z$ the vertices of degree at least $3$ of $EB$, with $x$ the cut vertex. 
Since $EB$ is a block, then $xy,~yz,~xz\in E(G)$. 
If  $y$ or $z$ are incident with some loop, then by contracting $xy$
we find $R_3$ as a minor of $G$. Thus, $y$ and $z$ are loopless vertices.
Moreover, since we are assuming that $EB$ contains no subdivision of
$2K_3$, then without loss of generality, we suppose that $yz$ has no
parallel edges. These arguments show that $EB$ is a subdivision
of $K_3(r,j,0)$ for some non-negative integer $j$ and $k$, with $j\ge 1$.
Now we show that $EB$ is a subdivision of at most one red edge. 
Recall that subdividing a pair of parallel edges in $K_3(r,j,0)$,
creates a subdivision of $K_{2,3}$. We do not consider such subdivisions
because we are working with outerplanar graphs. To see that $EB$ has at most
one subdivision class, we argue that the bicircular matroids
of the following graphs have an $ex_B(\mathcal{L})$-minor. 
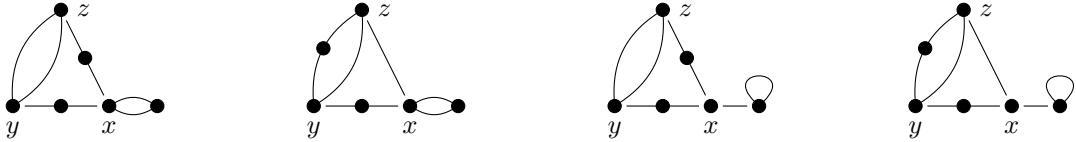
\begin{center}
\begin{tikzpicture}[scale=0.8]
\begin{scope}[xshift=-7.5cm, yshift=0cm, scale=0.8]
\node [vertex, label = 270:{$y$}] (0) at (-1,0){};
\node [vertex, label = 270:{$x$}] (1) at (1,0){};
\node [vertex, label = 0:{$z$}] (2) at (0,2){};
\node[vertex] at (0,0){};
\node [vertex] (s) at (0.5,1){};
\node [vertex] (m) at (2,0){};

\foreach \from/\to in {0/1, 1/2}
\draw [edge] (\from) to (\to);

\foreach \from/\to in {0/2, 2/0, m/1, 1/m}
\path[-]      (\from)  edge   [bend right]   (\to);
\end{scope}
\begin{scope}[xshift=-2.5cm, yshift=0cm, scale=0.8]
\node [vertex, label = 270:{$y$}] (0) at (-1,0){};
\node [vertex, label = 270:{$x$}] (1) at (1,0){};
\node [vertex, label = 0:{$z$}] (2) at (0,2){};
\node[vertex] at (0,0){};
\node [vertex]  at (-0.8,1.2){};
\node [vertex] (m) at (2,0){};

\foreach \from/\to in {0/1, 1/2}
\draw [edge] (\from) to (\to);

\foreach \from/\to in {0/2, 2/0, m/1, 1/m}
\path[-]      (\from)  edge   [bend right]   (\to);

\end{scope}
\begin{scope}[xshift=2.5cm, yshift=0cm, scale=0.8]
\node [vertex, label = 270:{$y$}] (0) at (-1,0){};
\node [vertex, label = 270:{$x$}] (1) at (1,0){};
\node [vertex, label = 0:{$z$}] (2) at (0,2){};
\node [vertex] (m) at (2,0){};
\node[vertex] at (0,0){};
\node [vertex] (s) at (0.5,1){};
\path[-]	     (m) edge [min distance=1cm]  (m);

\foreach \from/\to in {0/1, 1/2, m/1}
\draw [edge] (\from) to (\to);

\path[-]      (0)  edge   [bend left]   (2);
\path[-]        (0)  edge   [bend right]   (2);

\end{scope}
\begin{scope}[xshift=7.5cm, yshift=0cm, scale=0.8]
\node [vertex, label = 270:{$y$}] (0) at (-1,0){};
\node [vertex, label = 270:{$x$}] (1) at (1,0){};
\node [vertex, label = 0:{$z$}] (2) at (0,2){};
\node [vertex] (m) at (2,0){};
\node[vertex] at (0,0){};
\node [vertex] (s') at (-0.8,1.2){};

\foreach \from/\to in {0/1, 1/2, m/1}
\draw [edge] (\from) to (\to);

\path[-]      (0)  edge   [bend left]   (2);
\path[-]        (0)  edge   [bend right]   (2);
\path[-]	     (m) edge [min distance=1cm]  (m);
\end{scope}
\end{tikzpicture}
\end{center}
We proceed from left to right. Contracting one $yz$-edge in the left
most and second left most graphs, yields an $R_4'''$-minor and
an $R_4$-minor, respectively. The bicircular matroid of the middle
right and of the right most graph are isomorphic. So, we conclude
by noticing that we obtain $P_3$ as a minor of the right most graph
by contracting the edges $yz$ and $zx$. This shows that the bicircular matroids
of the graphs above have an $ex_B(\mathcal{L})$-minor.

To conclude the proof, we must consider the cases when 
$EB$ is a subdivision of $K_3(r,j,0)$ with $j\ge 1$ or with $r\ge 2$. 
In these cases, it suffices to observe that the bicircular matroid of the
following graphs have an $ex_B(\mathcal{L})$-minor.
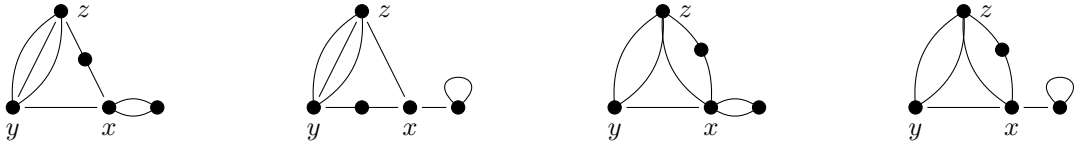
\begin{center}
\begin{tikzpicture}[scale=0.8]
\begin{scope}[xshift=-7.5cm, yshift=0cm, scale=0.8]
\node [vertex, label = 270:{$y$}] (0) at (-1,0){};
\node [vertex, label = 270:{$x$}] (1) at (1,0){};
\node [vertex, label = 0:{$z$}] (2) at (0,2){};
\node [vertex]  at (0.5,1){};
\node [vertex] (m) at (2,0){};
\foreach \from/\to in {0/2, 2/0, m/1, 1/m}
\path[-]      (\from)  edge   [bend right]   (\to);

\foreach \from/\to in {0/1, 1/2, 0/2}
\draw [edge] (\from) to (\to);

\end{scope}
\begin{scope}[xshift=-2.5cm, yshift=0cm, scale=0.8]
\node [vertex, label = 270:{$y$}] (0) at (-1,0){};
\node [vertex, label = 270:{$x$}] (1) at (1,0){};
\node [vertex, label = 0:{$z$}] (2) at (0,2){};
\node [vertex] (m) at (2,0){};
\node[vertex] at (0,0){};

\foreach \from/\to in {0/1, 1/2, 0/2, 1/m}
\draw [edge] (\from) to (\to);

\path[-]      (0)  edge   [bend left]   (2);
\path[-]        (0)  edge   [bend right]   (2);
\path[-]	     (m) edge [min distance=1cm]  (m);

\end{scope}
\begin{scope}[xshift=2.5cm, yshift=0cm, scale=0.8]
\node [vertex, label = 270:{$y$}] (0) at (-1,0){};
\node [vertex, label = 270:{$x$}] (1) at (1,0){};
\node [vertex, label = 0:{$z$}] (2) at (0,2){};
\node [vertex]  at (0.8,1.2){};
\node [vertex] (m) at (2,0){};
\foreach \from/\to in {0/2, 2/0, m/1, 1/m, 1/2, 2/1}
\path[-]      (\from)  edge   [bend right]   (\to);

\foreach \from/\to in {0/1}
\draw [edge] (\from) to (\to);

\end{scope}
\begin{scope}[xshift=7.5cm, yshift=0cm, scale=0.8]%
\node [vertex, label = 270:{$y$}] (0) at (-1,0){};
\node [vertex, label = 270:{$x$}] (1) at (1,0){};
\node [vertex, label = 0:{$z$}] (2) at (0,2){};
\node [vertex]  at (0.8,1.2){};
\node [vertex] (m) at (2,0){};

\foreach \from/\to in {0/1, 1/m}
\draw [edge] (\from) to (\to);

\path[-]      (0)  edge   [bend left]   (2);
\path[-]        (0)  edge   [bend right]   (2);
\path[-]      (1)  edge   [bend left]   (2);
\path[-]        (1)  edge   [bend right]   (2);
\path[-]	     (m) edge [min distance=1cm]  (m);

\end{scope}
\end{tikzpicture}
\end{center}
As we have done several times, we notice that each graph above contains
a minor whose bicircular matroid belongs to $ex_B(\mathcal{L})$. 
Contracting one $yz$-edge of the left most graph above yields
a graph isomorphic to $R_4$. Similarly,
$R_4'$ is the minor obtained from the middle left graph by contracting
one $yz$-edge. We  see that $P_2$ is a minor of the middle right graph
by removing the $xy$-edge and contracting the $xz$-edge.
By applying analogous operations to the right most graph, we see
that this graph contains  $P_3$ as a minor.

After considering all these possible cases, we conclude that the
end blocks of $G$ are subdivisions of at most one red edge
of $K_3(r,j,l)$ or of $mK_2''$. Hence, we conclude that $G\in\mathcal{F}_4$.
\end{proof}


\section{Characterizations}
\label{sec:CAR}

Using the work of previous sections, we now prove
Theorems~\ref{thm:graphs} and~\ref{thm:excludedminors}, 
and we discuss some of their implications.

%

\begingroup
\def\thetheorem{\ref{thm:graphs}}
\begin{theorem}
Let $G$ be a graph such that $B(G)$ is a connected matroid. 
The bicircular matroid $B(G)$ is a lattice path matroid if and only if
$G$ belongs to $\mathcal{F}_0\cup\mathcal{F}_1\cup \mathcal{F}_2 \cup
\mathcal{F}_3 \cup \mathcal{F}_4$.
\end{theorem}
\addtocounter{theorem}{-1}
\endgroup

\begin{proof}
By Propositions~\ref{prop:familyF0}, \ref{prop:familyF1},
\ref{prop:familyF2},~\ref{prop:familyF3}, and~\ref{prop:familyF4}
the bicircular matroid of any graph in
$\mathcal{F}_0\cup\mathcal{F}_1\cup \mathcal{F}_2 \cup
\mathcal{F}_3 \cup \mathcal{F}_4$ is a lattice path matroid.
On the other hand, if $B(G)$ is a lattice path matroid, 
then it is $ex_B(\mathcal{L})$-minor free (Proposition~\ref{prop:exminors}). 
Thus, if $B(G)$ is also connected, then by Proposition~\ref{prop:F0-F4},
we conclude that $G\in \mathcal{F}_0\cup\mathcal{F}_1\cup \mathcal{F}_2 \cup
\mathcal{F}_3\cup\mathcal{F}_4$.
\end{proof}

In this paragraph we claim and show that for any graph $G$ in $\mathcal{F}_0$, 
there is some graph in
$\mathcal{F}_1\cup \mathcal{F}_2\cup\mathcal{F}_3\cup\mathcal{F}_4$
whose bicircular matroid is isomorphic to $B(G)$--- we do so
to show that Theorem~\ref{thm:graphs} implies Corollary~\ref{cor:F1-F4}.
For instance, 
any subdivision of at most two edges of $K_4$ has the same bicircular matroid
as some subdivision of at most two edges of $G_1$. Similarly, 
any subdivision of blue and red edges of $K^\ast_{2,3}$ has the same bicircular
matroid as some subdivision of a pair of non parallel edges of $2K_3$. 
Subdivisions of $K_{2,3}$ are  subgraphs of  subdivisions of red and blue
edges of $K''_{2,3}$, and the latter have the same bicircular matroids
as subdivisions of red and blue edges of $K'_{2,3}$. To conclude this
argumentation, we show that subdivisions of $K'_{2,3}$
have isomorphic bicircular matroids to bicircular matroids of some graph
in $\mathcal{F}_4$. To do so, it suffices to notice that the following graphs
have isomorphic bicircular matroids (to the left, $K_{2,3}'$; to the right, a graph in 
$\mathcal{F}_4$).
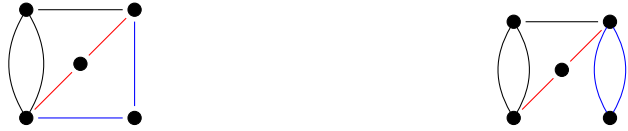
\begin{center}
\begin{tikzpicture}[scale=0.8]

\begin{scope}[xshift=-4cm, scale=0.9]
\node [vertex] (0) at (-1,0){};
\node [vertex] (1) at (1,0){};
\node [vertex] (2) at (-1,2){};
\node [vertex] (3) at (1,2){};
\node [vertex] (s) at (0,1){};

\draw [edge] (2) to (3);

\foreach \from/\to in {0/1, 1/3}
\draw [edge, blue] (\from) to (\to);

\foreach \from/\to in { 0/s, s/3}
\draw [edge, red] (\from) to (\to);

\path[-]          (0)  edge   [bend left]   (2);
\path[-]          (0)  edge   [bend right]   (2);
\end{scope}
\begin{scope}[xshift=4cm, scale = 0.8]
\node [vertex] (0) at (-1,0){};
\node [vertex] (1) at (1,0){};
\node [vertex] (2) at (-1,2){};
\node [vertex] (3) at (1,2){};
\node [vertex] (s) at (0,1){};

\draw [edge] (2) to (3);

\foreach \from/\to in { 0/s, s/3}
\draw [edge, red] (\from) to (\to);

\path[-]          (0)  edge   [bend left]   (2);
\path[-]          (0)  edge   [bend right]   (2);

\path[-, blue]          (1)  edge   [bend left]   (3);
\path[-, blue]          (1)  edge   [bend right]   (3);

\end{scope}

\end{tikzpicture}
\end{center}

\begin{corollary}\label{cor:F1-F4}
Let $L$ be a connected lattice path matroid. Then, $L$ is a bicircular matroid
if and only if there is some graph $G\in \mathcal{F}_1\cup \mathcal{F}_2 \cup
\mathcal{F}_3\cup\mathcal{F}_4$ such that $L\cong B(G)$. 
\end{corollary}

Recall that  Proposition~\ref{prop:nonouterplanar},
asserts that $\mathcal{F}_0$ contains
all non-outerplanar graphs whose bicircular matroids are connected
lattice path matroids.

\begin{corollary}\label{cor:outerplanar}
Let $L$ be a lattice path matroid. If $L$ is a bicircular matroid, then there
is an outerplanar graph $G$  such that $L\cong B(G)$.
\end{corollary}

Observation~\ref{obs:recognition} in Section~\ref{sec:GF}, asserts that we
can  recognize the union $\bigcup_{i=0}^4\mathcal{F}_i$ in linear time
with respect to the edge set of the input graph. Thus, the following claim
is an implication of Theorem~\ref{thm:graphs}.

\begin{corollary}
Given an input graph $G$, there is a linear time algorithm (with respect
to $|E(G)|$) that determines whether $B(G)$ is a lattice path matroid.
\end{corollary}

To conclude this work, we exhibit the list of excluded bicircular minors
to the class of lattice path matroids. We depict an affine presentations
of these in Figure~\ref{fig:allex}, and in Figure~\ref{fig:allexbic} we display 
a bicircular presentation of these matroids.

\begingroup
\def\thetheorem{\ref{thm:excludedminors}}
\begin{theorem}
A bicircular matroid is a lattice path matroid if and only if
it has no one of the following matroids as a minor:
\[
C^{2,4},~\mathcal{W}^3,~A^3,~R^3,~R^4,~D^4,~B^1, \text{ and } S^1.
\]	
\end{theorem}
\addtocounter{theorem}{-1}
\endgroup
\begin{proof}

On the one hand, Proposition~\ref{prop:exminors} shows that
if a bicircular matroid has either of the above as a minor, then 
it is not a lattice path matroid. On the other hand, since lattice path
matroids are closed under disjoint unions, then
we can assume that  the bicircular matroid $B(G)$ is connected. 
Thus, by Proposition~\ref{prop:F0-F4}, if $B(G)$ is $ex_B(\mathcal{L})$-minor
free, then $G\in\mathcal{F}_i$ for some $i\in\{0,1,2,3\}$.
Now, the claim follows by Theorem~\ref{thm:graphs}.
\end{proof}

\begin{figure}[ht!]
\begin{center}

\begin{tikzpicture}[scale=0.8]

\begin{scope}[xshift=-7.5cm, yshift = 3cm, scale=0.8]
\node [vertex] (0) at (-1.5,0){};
\node [vertex] (1) at (1.5,0){};
\node [vertex] (3) at (0,2.5){};
\node [vertex] (01) at (-0.75,1.25){};
\node [vertex] (13) at (0.75,1.25){};
\node [vertex] (30) at (0,0){};
\node[] at (0,-1){$\mathcal{W}^3$};

\foreach \from/\to in {1/0, 0/3, 3/1}
\path[-] (\from) edge (\to);

\end{scope}

\begin{scope}[xshift=-2.5cm, yshift=3cm, scale=0.8]
\coordinate (0) at (-1.5,0){};
\node [vertex] (00) at (-1.5,0.15){};
\node [vertex] (01) at (-1.5,-0.15){};
\coordinate  (1) at (1.5,0){};
\node [vertex] (10) at (1.5,-0.15){};
\node [vertex] (11) at (1.5,0.15){};
\node [vertex] (30) at (-0.5,1.5){};
\node [vertex] (31) at (0.5,1.5){};
\node [vertex] (c) at (0,0){};
\node[] at (0,-1){$R^3$};

\path[-] (0) edge (1);

\end{scope}

\begin{scope}[xshift=2.5cm, yshift=3cm, scale=0.8]
\node [vertex] (0) at (-1,0){};
\node [vertex] (1) at (1,0){};
\node [vertex] (3) at (0,2.5){};
\node [vertex] (01) at (-0.5,1.25){};
\node [vertex] (13) at (0.5,1.25){};
\node [vertex] (30) at (0,-0.3){};
\node[] at (0,-1){$A^3$};

\foreach \from/\to in {0/3, 3/1}
\path[-] (\from) edge (\to);

\end{scope}

\begin{scope}[xshift=7.5cm, yshift=3cm, scale=0.9]
\coordinate (c0) at (-1.5,0){};
\coordinate (c1) at (1.5,0){};
\coordinate (i0) at (-1.5,1.5){};
\coordinate (i1) at (1.5,1.5){};
\node [vertex] (0) at (-1.5,0){};
\node [vertex] (1) at (1.5,0){};
\node [vertex] (2) at (0,0){};
\node [vertex] (1) at (0.1,0.75){};
\node [vertex] (2) at (-0.1,0.75){};
\node [vertex] (0) at (-1.5,1.5){};
\node [vertex] (1) at (1.5,1.5){};
\node [vertex] (2) at (0,1.5){};
\node[] at (0,-1){$(S^1)^\ast$};

\foreach \from/\to in {c0/c1, i0/i1}
\path[-] (\from) edge (\to);
\end{scope}

\begin{scope}[xshift=-7.5cm, yshift=-1.8cm, scale=0.8]
\coordinate (c0) at (-1.5,0){};
\coordinate (c1) at (1.5,0){};
\coordinate (d0) at (-1.5,1){};
\coordinate (d1) at (1.5,1){};
\coordinate (i0) at (-1.5,2){};
\coordinate (i1) at (1.5,2){};
\node [vertex] (0) at (-1.5,0){};
\node [vertex] (1) at (1.5,0){};
\node [vertex] (2) at (0,0){};
\node [vertex] (0) at (-1.5,1){};
\node [vertex] (1) at (1.5,1){};
\node [vertex] (2) at (0,1){};
\node [vertex] (0) at (-1.5,2){};
\node [vertex] (1) at (1.5,2){};
\node [vertex] (2) at (0,2){};
\node[] at (0,-1.2){$B^1$};

\foreach \from/\to in {c0/c1, d0/d1, i0/i1}
\path[-] (\from) edge (\to);
\end{scope}

\begin{scope}[xshift=-2.5cm, yshift=-2cm, scale=0.8]
\node [vertex] (0) at (-1.3,0){};
\node [vertex] (1) at (1.3,0){};
\node [vertex] (3) at (-1.3,3){};
\node [vertex] (4) at (1.3,3){};
\node [vertex] (5) at (0,3.7){};
\node [vertex] (6) at (0,0.7){};
\node[] at (0,-1){$C^{2,4}$};

\foreach \from/\to in {0/1, 1/4, 3/4, 3/0, 5/3, 5/4}
\path[-] (\from) edge (\to);

\foreach \from/\to in {6/0, 6/1, 6/5}
\draw [edge, dotted] (\from) to (\to);

\end{scope}

\begin{scope}[xshift=2.5cm, yshift=-2.2cm, scale=0.8]
\coordinate (c0) at (0,0){};
\coordinate (c1) at (0,3){};
\coordinate (d0) at (-1.8,1){};
\coordinate (d1) at (-1.8,4){};
\coordinate (i0) at (1.8,1){};
\coordinate (i1) at (1.8,4){};
\node [vertex]  at (-0.15,2.3){};
\node [vertex]  at (0.15,2.3){};
\node [vertex] at (0,1){};
\node [vertex] at (-1.2,2.8){};
\node [vertex]  at (1.2,2.8){};
\node [vertex] at (-1.2,1.5){};
\node [vertex] at (1.2,1.5){};
\node[] at (0,-0.8){$R^4$};

\foreach \from/\to in {c0/c1, d0/d1, i0/i1, c0/d0, c0/i0, c1/i1, c1/d1}
\path[-] (\from) edge (\to);
\end{scope}

\begin{scope}[xshift=7.5cm, yshift=-2.2cm, scale=0.8]
\coordinate (c0) at (0,0){};
\coordinate (c1) at (0,3){};
\coordinate (d0) at (-1.8,1){};
\coordinate (d1) at (-1.8,4){};
\coordinate (i0) at (1.8,1){};
\coordinate (i1) at (1.8,4){};
\node [vertex] (1) at (-1.2,1.3){};
\node [vertex] at (-1.2,2.2){};
\node [vertex] (2) at (-1.2,3){};
\node [vertex] (11) at (-0.6,1){};
\node [vertex] at (-0.6,1.9){};
\node [vertex] (22) at (-0.6,2.7){};

\node [vertex]  at (1,2.6){};
\node [vertex] at (1,1.6){};
\node[] at (0,-0.8){$D^4$};

\foreach \from/\to in {c0/c1, d0/d1, i0/i1, c0/d0, c0/i0, c1/i1, c1/d1, 1/2, 11/22}
\path[-] (\from) edge (\to);
\end{scope}

\end{tikzpicture}

\caption{An affine presentation of the eight bicircular excluded minors for lattice
path matroids (we choose to depict the dual of $S^1$ because $S^1$ has rank $5$).}
\label{fig:allex}
\end{center}
\end{figure}
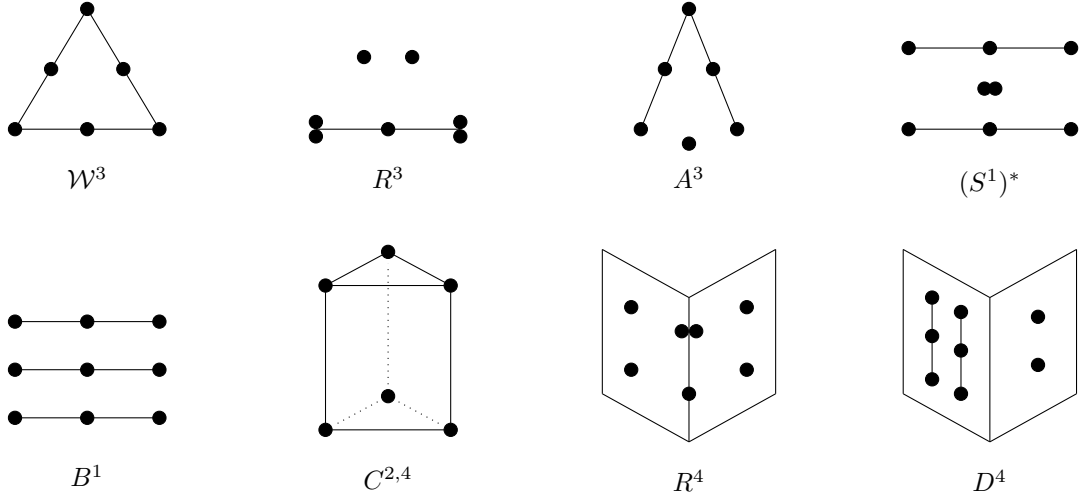

\begin{figure}[ht!]
\begin{center}

\begin{tikzpicture}[scale=0.8]

\begin{scope}[xshift=-7.5cm, yshift = 3cm, scale=0.8]
\node [vertex] (0) at (-1,0){};
\node [vertex] (1) at (1,0){};
\node [vertex] (3) at (0,2){};
\node[] at (0,-1){$B(W)\cong \mathcal{W}^3$};

\foreach \from/\to in {1/0, 0/3, 3/1}
\draw [edge] (\from) to (\to);

\path[-]         (1) edge [min distance=1cm, out=45, in=-45] (1);
\path[-]         (0) edge [min distance=1cm, out=225, in=135] (0);
\path[-]         (3) edge [min distance=1cm] (3);
\end{scope}

\begin{scope}[xshift=-2.5cm, yshift=3cm, scale=0.8]
\node [vertex] (0) at (-1,0){};
\node [vertex] (1) at (1,0){};
\node [vertex] (3) at (0,2){};
\node[] at (0,-1){$B(R_3)\cong R^3$};

\foreach \from/\to in {1/0, 0/3, 3/1}
\draw [edge] (\from) to (\to);

\path[-]         (0) edge [min distance=1cm, out=90, in=175] (0);
\path[-]         (0) edge [min distance=1cm, out=185, in=270] (0);
\path[-]         (1) edge [min distance=1cm, out=90, in=5] (1);
\path[-]         (1) edge [min distance=1cm, out=355, in=270] (1);

\end{scope}

\begin{scope}[xshift=2.5cm, yshift=3cm, scale=0.8]
\node [vertex] (0) at (-1,0){};
\node [vertex] (1) at (1,0){};
\node [vertex] (3) at (0,2){};
\node[] at (0,-1){$B(A_3)\cong A^3$};

\foreach \from/\to in {0/1}
\draw [edge] (\from) to (\to);

\path[-]         (3) edge [min distance=1cm] (3);

\foreach \from/\to in {3/1, 1/3, 3/0, 0/3}
\path[-]          (\from)  edge   [bend left]   (\to);

\end{scope}

\begin{scope}[xshift=7.5cm, yshift=3cm, scale=0.9]
\node [vertex] (0) at (-1,0){};
\node [vertex] (1) at (1,0){};
\node [vertex] (2) at (-1,2){};
\node [vertex] (3) at (1,2){};
\node [vertex] (s) at (0,1){};
\node[] at (0,-1){$B(S_1)\cong S^1$};

\foreach \from/\to in {0/1, 3/2, 0/s, s/3}
\draw [edge] (\from) to (\to);

\path[-]          (0)  edge   [bend left]   (2);
\path[-]          (0)  edge   [bend right]   (2);
\path[-]          (1)  edge   [bend right]   (3);
\path[-]          (1)  edge   [bend left]   (3);
\end{scope}

\begin{scope}[xshift=-7.5cm, yshift=-1.5cm, scale=0.8]
\node [vertex] (0) at (-1,0){};
\node [vertex] (1) at (1,0){};
\node [vertex] (2) at (0,2){};
\node[] at (0,-1){$B(3K_3)\cong B^1$};

\path[-]          (0)  edge   [bend left]   (2);
\path[-]          (0)  edge   [bend right]   (2);
\path[-]          (1)  edge   [bend right]   (2);
\path[-]          (1)  edge   [bend left]   (2);
\path[-]          (1)  edge   [bend right]   (0);
\path[-]          (1)  edge   [bend left]   (0);

\foreach \from/\to in {0/1, 0/2, 1/2}
\draw [edge] (\from) to (\to);
\end{scope}

\begin{scope}[xshift=-2.5cm, yshift=-1.5cm, scale=0.8]
\node [vertex] (0) at (-1,0){};
\node [vertex] (1) at (1,0){};
\node [vertex] (c) at (0,0.9){};
\node [vertex] (3) at (0,2.24){};
\node[] at (0,-1){$B(P_2)\cong C^{2,4}$};

\foreach \from/\to in {c/1,c/0, c/3, 0/c, 1/c, 3/c}
\path[-]          (\from)  edge   [bend left]   (\to);

\end{scope}

\begin{scope}[xshift=2.5cm, yshift=-1.5cm, scale=0.8]
\node [vertex] (0) at (-1,0){};
\node [vertex] (1) at (1,0){};
\node [vertex] (2) at (-1,2){};
\node [vertex] (3) at (1,2){};
\node[] at (0,-1){$B(R_4)\cong R^4$};

\foreach \from/\to in {0/1, 0/3, 1/3}
\draw [edge] (\from) to (\to);

\path[-]          (0)  edge   [bend left =20]   (2);
\path[-]          (0)  edge   [bend right =20]   (2);

\path[-]         (3) edge [min distance=1cm] (3);
\path[-]         (3) edge [min distance=1cm, out=40, in=320] (3);
\end{scope}

\begin{scope}[xshift=7.5cm, yshift = -1.5cm, scale=0.8]
\node [vertex] (0) at (-1.1,0){};
\node [vertex] (1) at (1.1,0){};
\node [vertex] (3) at (0,2.2){};
\node [vertex] at (0,0){};
\node[] at (0,-1){$B(D_4)\cong D^4$};

\foreach \from/\to in {0/1, 0/3, 1/3}
\draw [edge] (\from) to (\to);

\foreach \from/\to in {3/1, 1/3, 3/0, 0/3}
\path[-]          (\from)  edge   [bend left]   (\to);

\end{scope}

\end{tikzpicture}

\caption{A bicircular presentation of the eight bicircular excluded minors for lattice
path matroids. In each case we choose the presentation that minimizes the
number of loops.}
\label{fig:allexbic}
\end{center}
\end{figure}
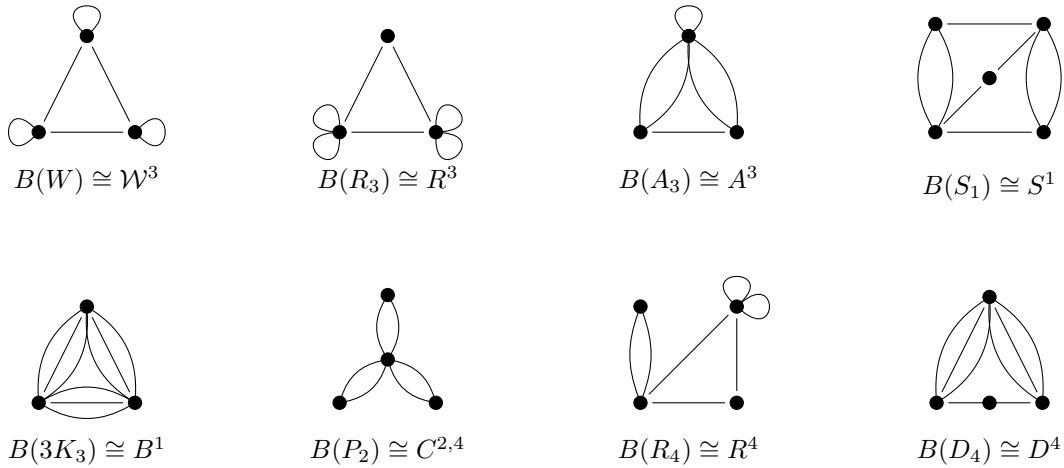

It was recently proved that there are finitely many excluded minors
to the class of bicircular matroids \cite{deVosAR}. This, together with
Theorem~\ref{thm:excludedminors} implies that there are
finitely many excluded minors to the class of lattice path bicircular matroids.

\begin{corollary}\label{cor:finitelatbic}
There are only finitely many excluded minors to the class of lattice path 
bicircular matroids.
\end{corollary}

Allow us to discuss one more implication of the characterizations proposed
in this section. By Corollary~\ref{cor:F1-F4}, every lattice path bicircular
matroid admits a bicircular presentation by a graph in $\mathcal{F}_i$
for some $i\in\{1,2,3,4\}$. By definition of $\mathcal{F}_4$, all graphs in
this family have a cut vertex. This implies that their bicircular matroids
are not $2$-connected. Also, by Lemma~\ref{lem:familyF3}, the bicircular
matroids of graphs in $\mathcal{F}_3$ are minors of bicircular
matroids of graphs in $\mathcal{F}_1\cup \mathcal{F}_2$. Thus, every
$2$-connected lattice path bicircular matroid is a minor of the bicircular
matroid of some graph in $\mathcal{F}_1\cup\mathcal{F}_2$. 

Recall that the bicircular matroid of some $2$-connected graph $G$ is
cosimple whenever
$G$ has no subdivided edges. The subfamilies of graphs with no subdivided edges
of $\mathcal{F}_1\cup \mathcal{F}_2$ are the graph families
$G_1(r,b,d)$ and $2K_3(r,b)$. The bicircular matroids of these graphs
have simple geometric descriptions. The bicircular matroid of $2K_3(r,b)$
is the rank $3$ matroid with one $(r+2)$-point line, one $(b+2)$-point line, 
and two points in general position. The bicircular matroid
of $G_1(r,d,b)$ is the rank $4$ matroid with two hyperplanes $H_1$ and $H_2$
intersecting in a $d$-point line $\ell$, where $H_1$ and $H_2$ contain
$\ell$, one point that belongs to no $3$-point line, and 
an $(r+2)$-point line and a $(b+2)$-point line, respectively.
The following are affine representations of $B(2K_3(2,1))$ and of $B(G_1(1,4,0))$.
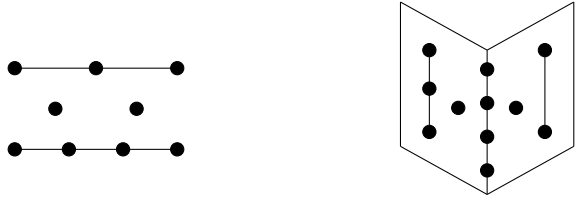
\begin{center}
\begin{tikzpicture}[scale=0.8]

\begin{scope}[xshift=-3.5cm, yshift = 0.75cm, scale=0.9]%
\coordinate (c0) at (-1.5,0){};
\coordinate (c1) at (1.5,0){};
\coordinate (i0) at (-1.5,1.5){};
\coordinate (i1) at (1.5,1.5){};
\node [vertex]  at (-1.5,0){};
\node [vertex]  at (1.5,0){};
\node [vertex]  at (-0.5,0){};
\node [vertex]  at (0.5,0){};
\node [vertex]  at (-0.75,0.75){};
\node [vertex]  at (0.75,0.75){};
\node [vertex] (0) at (-1.5,1.5){};
\node [vertex] (1) at (1.5,1.5){};
\node [vertex] (2) at (0,1.5){};

\foreach \from/\to in {c0/c1, i0/i1}
\path[-] (\from) edge (\to);
\end{scope}
\begin{scope}[xshift=3cm, scale = 0.8]%
\coordinate (c0) at (0,0){};
\coordinate (c1) at (0,3){};
\coordinate (d0) at (-1.8,1){};
\coordinate (d1) at (-1.8,4){};
\coordinate (i0) at (1.8,1){};
\coordinate (i1) at (1.8,4){};
\node [vertex] (1) at (-1.2,1.3){};
\node [vertex] at (-1.2,2.2){};
\node [vertex] (2) at (-1.2,3){};

\node [vertex] (11) at (1.2,1.3){};
\node [vertex] (22) at (1.2,3){}; 

\node [vertex]  at (-0.6,1.8){};
\node [vertex] at (0.6,1.8){};
\node [vertex]  at (0,0.5){};
\node [vertex] at (0,1.2){};
\node [vertex]  at (0,1.9){};
\node [vertex] at (0,2.6){};

\foreach \from/\to in {c0/c1, d0/d1, i0/i1, c0/d0, c0/i0, c1/i1, c1/d1, 1/2, 11/22}
\path[-] (\from) edge (\to);
\end{scope}

\end{tikzpicture}
\end{center}

\begin{proposition}\label{prop:2conectedcosimple}
For a $2$-connected cosimple matroid $M$ the following statements are equivalent:
\begin{enumerate}
	\item $M$ is a lattice path bicircular matroid,
	\item $M$ is a minor of $B(G_1(r,b,d))$ or of $B(2K_3(r,b))$, for 
	some non-negative integers $r$, $b$ and $d$.
\end{enumerate}
\end{proposition}
\begin{proof}
This statement is proved in the two preceding paragraphs.
\end{proof}

\begin{corollary}
The cosimplification of  $2$-connected  lattice path bicircular matroid has rank at
most $4$.
\end{corollary}


\section{Conclusions}
\label{sec:conclusions}

Theorem~\ref{thm:excludedminors} lists all excluded bicircular minors
to the class of lattice path matroids. A natural extension of this work is to
exhibit the list of excluded lattice path minors to the class of bicircular matroids.
Moreover, Corollary~\ref{cor:finitelatbic} asserts that there are only a finite number
of excluded minors for the class of  lattice path bicircular matroids, which
are these excluded minors? 
\begin{problem}
Exhibit the excluded minors for lattice path bicircular matroids.
\end{problem}
As Sivaraman and Slilaty \cite{sivaramanGC2022} suggested,
we also considered to study the intersection of bicircular matroids and
multipath matroids; but it turned out to be (even more) technical than the
present work. The interested reader could try to find nicer arguments to
characterize multipath bicircular matroids.
Also, the dual class of transversal matroids is the class of strict gammoids.
Since lattice path matroids and cobicircular matroids are duals of transversal
matroids,  studying the intersection of bicircular matroids and strict gammoids
would be a nice extension of the present work and of the work of
Sivaraman and Slilaty \cite{sivaramanGC2022}.

Finally, Corollary~\ref{cor:outerplanar} asserts that every lattice
path bicircular matroid is the bicircular matroid of some
outerplanar graph; is there an interesting characterization of
bicircular matroids of outerplanar graphs?
We believe it could also be interesting to consider super classes
of outerplanar graphs. For instance, studying the class $\mathcal{C}$
of bicircular matroids of series-parallel graphs could yield a nice
problem to investigate. In particular, all lattice path bicircular matroids
belong to $\mathcal{C}$.  Also, if $B(G)$ is representable over $GF(4)$
then it is $U_{4,6}$-minor free, and thus $G$ is $K_4$-minor free, i.e.,
$G$ is a series-parallel graph. We stress out that $\mathcal{C}$ is not the class
of $U_{4,6}$-minor free bicircular matroids since $U_{4,6}\cong B(G_1)$
and $G_1$ is a series-parallel graph.

\section*{Acknowledgements}
This work was carried out during a visit of the first author
at FernUniversit\"at in Hagen, supported by DAAD grant 57552339.

\end{document}